\pdfoutput=1
\documentclass[twoside]{article}
\usepackage[letterpaper, left=3cm, right=3cm, top=4cm, bottom=2.5cm]{geometry}
\usepackage{graphicx}
\usepackage{amsmath,amssymb,amsthm}
\usepackage{authblk}
\usepackage[utf8]{inputenc}
\usepackage{microtype}
\usepackage{mathrsfs} 
\usepackage{enumitem}
\usepackage{calc}
\usepackage{esint}
\usepackage{fancyhdr}
\usepackage{xcolor}
\usepackage[labelfont = bf]{caption}
\usepackage{doi}
\usepackage{subfigure}
\usepackage[numbers]{natbib}
\usepackage{stmaryrd}
\usepackage{mathtools}
\usepackage{booktabs}
\usepackage{colortbl}
\usepackage{tabulary}
\usepackage{diagbox}
\usepackage{caption}
\usepackage{float}

\usepackage{paper_diening}

\numberwithin{equation}{section}
\newtheorem{theorem}{Theorem}[section]
\newtheorem{proposition}[equation]{Proposition}

\newtheorem{corollary}[equation]{Corollary}
\newtheorem{remark}[equation]{Remark}
\newtheorem{lemma}[equation]{Lemma}

\newtheorem{assumption}[equation]{Assumption}

\providecommand{\Vo}{{\mathaccent23 V}}
\providecommand{\Qo}{{\mathaccent23 Q}}

\DeclareMathOperator*{\esssup}{ess\,sup}
\DeclareMathOperator*{\essinf}{ess\,inf}

\DeclareMathAlphabet\mathbfcal{OMS}{cmsy}{b}{n}

\usepackage{titlesec}
\titleformat{\section}{\normalfont\scshape\centering}{\thesection.}{0.5em}{}
\titleformat*{\subsection}{\itshape}
\titleformat*{\subsubsection}{\itshape}

\providecommand{\keywords}[1]
{
	{\small\textit{Keywords:~~} #1}
}

\providecommand{\MSC}[1]
{
	{\small\textit{AMS MSC (2020): ~~} #1}
}

\usepackage{hyperref}
\hypersetup{
	colorlinks=true,
	linkcolor=blue,
	citecolor = blue,
	filecolor=magenta, 
	urlcolor=cyan,
}

\begin{document}
	\setlength{\abovedisplayskip}{5.5pt}
	\setlength{\belowdisplayskip}{5.5pt}
	\setlength{\abovedisplayshortskip}{5.5pt}
	\setlength{\belowdisplayshortskip}{5.5pt}

	\title{\vspace{-2cm}Error analysis for a fully-discrete  finite element approximation of the unsteady $p(\cdot,\cdot)$-Stokes equations}
	\author[1]{Luigi C. Berselli\thanks{Email: \texttt{luigi.carlo.berselli@unipi.it}}\thanks{funded by  INdAM GNAMPA and Ministero dell'istruzione, dell'università e della ricerca (MIUR, Italian Ministry of Education, University and Research)
			within PRIN20204NT8W4: Nonlinear evolution PDEs, fluid dynamics and transport equations: theoretical foundations and applications.}}
	\author[2]{Alex Kaltenbach\thanks{Email: \texttt{{kaltenbach@math.tu-berlin.de}}}\thanks{funded by the Deutsche Forschungsgemeinschaft (DFG, German Research
			Foundation) - 525389262.}}
	\author[3]{Seungchan Ko\thanks{Email: \texttt{scko@inha.ac.kr}}\thanks{funded by National Research Foundation of Korea Grant funded by the Korean Government (RS-2023-00212227)}}
	\date{\today\vspace{-2.5mm}}
	\affil[1]{\small{Department of Mathematics, University of Pisa, Largo Bruno Pontecorvo 5, 56127~Pisa,~{Italy}}}
	\affil[2]{\small{Institute of Mathematics, Technical University of Berlin, Straße des 17. Juni 136, 10623~Berlin,~{Germany}}}
	\affil[3]{\small{Department of Mathematics,  Inha University, 100 Inha-ro, Michuhol-gu, 22201 Incheon,~{Republic~of~Korea}}}
	\maketitle
	 
	\pagestyle{fancy}
	\fancyhf{}
	\fancyheadoffset{0cm}
	\addtolength{\headheight}{-0.25cm}
	\renewcommand{\headrulewidth}{0pt} 
	\renewcommand{\footrulewidth}{0pt}
	\fancyhead[CO]{\textsc{Error analysis for a FE approximation of the unsteady $p(\cdot,\cdot)$-Stokes equations}}
	\fancyhead[CE]{\textsc{L. C. Berselli, A. Kaltenbach, and S. Ko}}
	\fancyhead[R]{\thepage}
	\fancyfoot[R]{}
	
	\begin{abstract}
		In this paper, we examine a fully-discrete finite element approximation of the unsteady $p(\cdot,\cdot)$-Stokes~\mbox{equations} (\textit{i.e.}, $p(\cdot,\cdot)$ is time- and space-dependent), employing a backward Euler step in time and conforming, discretely inf-sup stable finite elements in space. More precisely,
		we derive 
		error decay rates for the vector-valued velocity field 
        imposing 
        fractional regularity assumptions on the velocity  and~the~kinematic~pressure. In addition, we carry out numerical experiments that confirm the optimality of the derived error decay rates in the case $p(\cdot,\cdot)\ge 2$.
	\end{abstract}
	
	
	\MSC{35J60; 35Q35;  65N15; 65N30; 76A05.}
	
	\section{Introduction}
	\thispagestyle{empty}

	\hspace{5mm} In the present paper, we examine a fully-discrete finite element (FE) approximation of the unsteady \textit{$p(\cdot,\cdot)$-Stokes~equations}, \textit{i.e.},\enlargethispage{10mm}
	\begin{equation}
		\label{eq:ptxStokes}
		\begin{aligned}
			\partial_{\mathrm{t}}\bfv-\mathrm{div}_{\mathrm{x}}\mathbf{S}(\cdot,\cdot,\bfD_{\mathrm{x}}\bfv)+\nabla_{\mathrm{x}} q&=\bfg+\textup{div}_{\mathrm{x}}\, \mathbf{G}  &&\quad\text{ in }Q_T\,,\\
			\mathrm{div}_{\mathrm{x}}\bfv&=0 &&\quad\text{ in }Q_T\,,
			\\
			\bfv &= \mathbf{0} &&\quad\text{ on } \Gamma_T\,,\\
			\bfv(0) &= \mathbf{v}_0 &&\quad\text{ in } \Omega\,,
		\end{aligned}
	\end{equation} 
	employing a backward Euler step in time and conforming, discretely inf-sup stable finite elements~in~space, for error decay rates.
In the system \eqref{eq:ptxStokes}, for a given external force $\bfg+\mathrm{div}_{\mathrm{x}}\bfG\colon Q_T\to \mathbb{R}^d$, the incompressi-bility constraint \eqref{eq:ptxStokes}$_2$, a no-slip boundary condition~\eqref{eq:ptxStokes}$_3$, and an initial velocity  $\mathbf{v}_0\colon \Omega\to \mathbb{R}^d$,  one seeks for a \textit{velocity vector field} $\bfv\coloneqq (v_1,\ldots,v_d)^\top\colon \overline{Q_T}\to
\mathbb{R}^d$ and a \textit{kinematic~\mbox{pressure}}~${q\colon Q_T\to \mathbb{R}}$~\mbox{solving}~\eqref{eq:ptxStokes}.
Here, $\Omega\subseteq \mathbb{R}^d$, $d\in \{2,3\}$, is a bounded~polyhedral~Lipschitz~domain, $I\coloneqq (0,T)$, $T<\infty$ is the time interval of interest, $Q_T\coloneqq I\times\Omega$ is the corresponding space-time cylinder, and $\Gamma_T\coloneqq I\times\partial\Omega$. The \textit{extra-stress tensor}\footnote{Here, $\mathbb{R}^{d\times d}_{\textup{sym}}\coloneqq \{\mathbf{A}\in \mathbb{R}^{d\times d}\mid \mathbf{A}=\mathbf{A}^\top\}$ and $\mathbf{A}^{\textup{sym}}\coloneqq[\mathbf{A}]^{\textup{sym}}\coloneqq \frac{1}{2}(\mathbf{A}+\mathbf{A}^\top)\in \mathbb{R}^{d\times d}_{\textup{sym}}$ for all $\mathbf{A}\in\mathbb{R}^{d\times d}$.} $\mathbf{S}(\cdot,\cdot,\bfD_{\mathrm{x}}\bfv)\colon  Q_T\to \mathbb{R}^{d\times d}_{\textup{sym}}$  depends on the \textit{strain-rate} \mbox{tensor} $\smash{\bfD_{\mathrm{x}}\bfv \coloneqq  [\nabla_{\mathrm{x}}  \bfv]^{\textup{sym}}
		\colon  Q_T\to  \mathbb{R}^{d\times d}_{\textup{sym}}}$, \textit{i.e.}, the symmetric part of the velocity gradient  $\nabla_{\mathrm{x}}  \bfv \coloneqq(\partial_{x_j} v_i)_{i,j=1,\ldots,d}
	\colon Q_T\to \mathbb{R}^{d\times d}$, and assumes the form
    \begin{equation}\label{S_form}
        \mathbf{S}(\cdot,\cdot,\bfD_{\mathrm{x}}\bfv)\coloneqq \nu_0\,(\delta+|\bfD_{\mathrm{x}}\bfv|)^{p(\cdot,\cdot)-2}\bfD_{\mathrm{x}}\bfv\quad \text{ in }Q_T\,,
    \end{equation}
     where $\nu_0>0$, $\delta\geq0$, and  the  \textit{power-law index} $p\colon Q_T\rightarrow[0,\infty)$ is a (Lebesgue) measurable~function~with
    \begin{equation}
        \label{p_bound}
            1\leq p^-\coloneqq {\essinf}_{(t,x)^\top\in Q_T}p(t,x)\leq {\esssup}_{(t,x)^\top\in Q_T}p(t,x)\eqqcolon p^+<\infty\,.
    \end{equation}
    
    The unsteady $p(\cdot,\cdot)$-Stokes equations \eqref{eq:ptxStokes} is a prototype example of a non-linear system with non-standard growth conditions and is the simplified version of the unsteady $p(\cdot,\cdot)$-Navier--Stokes equations by considering the case of a slow (\textit{i.e.}, laminar) flow, for which the convective~term~$[\nabla_{\mathrm{x}}\mathbf{v}]\mathbf{v}$~can~be~neglected. The unsteady $p(\cdot,\cdot)$-Navier--Stokes~equations naturally appear in mathematical models~for~\mbox{\textit{smart~fluids}}, \textit{e.g.}, electro-rheological fluids (\textit{cf}.\ \cite{rubo}),~micro-polar electro-rheological fluids (\textit{cf}.\ \cite{win-r}), magneto-rheological fluids (\textit{cf}.\ \cite{bia_2005}),~\mbox{thermo-rheological}~\mbox{fluids}~(\textit{cf}.~\cite{AR06}), and chemically-reacting fluids (\textit{cf}.\ \cite{HMPR10}). For all these mathematical models, the variable power-law index $p(\cdot,\cdot)$ \hspace{-0.1mm}depends \hspace{-0.1mm}on \hspace{-0.1mm}certain~\hspace{-0.1mm}physical~\hspace{-0.1mm}quantities \hspace{-0.1mm}including \hspace{-0.1mm}an \hspace{-0.1mm}electric \hspace{-0.1mm}field, \hspace{-0.1mm}a \hspace{-0.1mm}magnetic \hspace{-0.1mm}field,~\hspace{-0.1mm}a~\hspace{-0.1mm}\mbox{temperature}~\hspace{-0.1mm}field,  a concentration of a specific molecule and, in this way, implicitly~the~time-space~variable $(t, x)^\top\in Q_T$. Smart fluids have a wide range of potential applications in various areas of science and engineering such as automotive,~heavy~machinery, electronics, aerospace, and biomedical industries (\textit{cf}.\ \cite[Chap.\ 6]{smart_fluids} and references therein).\vspace{-2mm}

\subsection{Related contributions}\vspace{-1mm}
    \hspace{5mm}Let us recall some known results in the numerical analysis of models related~to~the~unsteady $p(\cdot, \cdot)$-Stokes equations \eqref{eq:ptxStokes}. The numerical analysis for fluids with shear-dependent viscosity~started~in~\cite{sandri}. 
    In \cite{bdr-phi-stokes-preprint}, for  conforming, discretely inf-sup stable FE approximations~of~the~steady $p$-Stokes~equations,~(\textit{i.e.}, ${p=\textrm{const}}$), 
    \textit{a priori}  error estimates measured in the `natural distance'~(or~`\mbox{quasi-norms}',~\mbox{respectively}) were derived. In \cite{sarah}, for a fully-discrete FE approximation of the unsteady $p$-Stokes equations \eqref{eq:ptxStokes} (\textit{i.e.}, ${p=\textrm{const}}$), employing a backward Euler step in time and conforming, discretely inf-sup stable FEs in space, \textit{a priori} error \mbox{estimates}~were~\mbox{derived}. However, the numerical analysis of problems~with~a variable power-law~index 
    is less developed. In fact, we are merely aware of the following contributions:\vspace{-1mm}\enlargethispage{11mm}

    \begin{itemize}[noitemsep,topsep=2pt,leftmargin=!,labelwidth=\widthof{$\bullet$},font=\itshape] \item[$\bullet$] \textit{The steady case:}
    \begin{itemize}[noitemsep,topsep=2pt,leftmargin=!,labelwidth=\widthof{(a)},font=\itshape]
        \item[(a)] \textit{Convergence analyses:} 
        \begin{itemize}[noitemsep,topsep=2pt,leftmargin=!,labelwidth=\widthof{\itshape(ii)},font=\itshape] 
        \item[(i)] In \cite{DPLM12} and \cite{BalciOrtnerStorn2022}, $\Gamma$-convergence analyses for an 
        Interior Penalty Discontinuous~Galerkin~(IPDG)  approximation and a Crouzeix--Raviart approximation of the   steady $p(\cdot)$-Laplacian equation, respectively, were conducted;
        \item[(ii)] In \cite{ko_1} and \cite{ko_2}, for a conforming, discretely inf-sup stable FE approximation of a model describing the steady motion of a chemically reacting fluid, where the power-law index depends on the concentration~of~a~specific~molecule, weak convergence analyses were~carried~out.  
        \end{itemize}
        \item[(b)] \textit{\textit{A priori} error analyses:}
        \begin{itemize}[noitemsep,topsep=2pt,leftmargin=!,labelwidth=\widthof{\itshape(iii)},font=\itshape] 
        \item[(i)] In \cite{BDS15} and \cite{BK23_pxDirichlet},  for a conforming FE approximation and a Crouzeix--Raviart approximation of the steady $p(\cdot)$-Laplace equation, respectively, \textit{a priori}  error estimates were~derived;
        \item[(ii)] In \hspace{-0.1mm}\cite{BBD15}, \hspace{-0.1mm}for \hspace{-0.1mm}a \hspace{-0.1mm}conforming, \hspace{-0.1mm}discretely \hspace{-0.1mm}inf-sup \hspace{-0.1mm}stable \hspace{-0.1mm}FE \hspace{-0.1mm}approximation \hspace{-0.1mm}of
\hspace{-0.1mm}the \hspace{-0.1mm}steady \hspace{-0.1mm}$p(\cdot)$-Stokes~\hspace{-0.1mm}\mbox{equations}, \textit{a priori}  error estimates were~derived;
\item[(iii)]  In \cite{berselli2023error}, for a conforming, discretely inf-sup stable FE approximation of  
the steady $p(\cdot)$-Navier--Stokes equations, \textit{a priori}  error estimates~were~derived.  
\end{itemize}
    \end{itemize}
    
    \item[$\bullet$] \textit{The unsteady case:}
    \begin{itemize}[noitemsep,topsep=2pt,leftmargin=!,labelwidth=\widthof{(a)},font=\itshape]
        \item[(a)] \textit{Convergence analyses:}

        \begin{itemize}[noitemsep,topsep=2pt,leftmargin=!,labelwidth=\widthof{\itshape(ii)},font=\itshape]

\item[(i)] In \cite{CHP10}, for a fully-discrete FE approximation of the  unsteady $p(\cdot)$-Navier--Stokes equations (\textit{i.e.}, the power-law index is time-independent), employing a backward Euler step in time and \hspace{-0.1mm}conforming, \hspace{-0.1mm}discretely \hspace{-0.1mm}inf-sup \hspace{-0.1mm}stable \hspace{-0.1mm}FEs \hspace{-0.1mm}in \hspace{-0.1mm}space,
\hspace{-0.1mm}a \hspace{-0.1mm}weak \hspace{-0.1mm}convergence~\hspace{-0.1mm}\mbox{analysis}~\hspace{-0.1mm}was~\hspace{-0.1mm}\mbox{carried}~\hspace{-0.1mm}out;

\item[(ii)] In \cite{berselli2024convergence}, 
 for a fully-discrete FE approximation of the unsteady $p(\cdot,\cdot)$-Navier--Stokes equations (\textit{i.e.}, the power-law index is both time- and space-dependent), employing a backward Euler step in time and conforming, discretely inf-sup stable FEs in space, a
(weak) convergence analysis was carried out. 
\end{itemize}

        \item[(b)] \textit{\textit{A priori} error analyses:}
         \begin{itemize}[noitemsep,topsep=2pt,leftmargin=!,labelwidth=\widthof{(ii)},font=\itshape]
        \item[(i)] In~\cite{din_diss}, for semi-discrete  approximations of  the unsteady $p(\cdot,\cdot)$-Navier--Stokes~\mbox{equations}, employing  a backward or a forward Euler step in time, \textit{a priori} error~estimates~were~\mbox{derived};
\item[(ii)] In \cite{ptx_lap}, for a fully-discrete FE approximation of the unsteady $p(\cdot,\cdot)$-Laplace equation, employing a backward Euler step in time and conforming, element-wise affine FEs in space, \textit{a priori} error estimates  were derived.
 \end{itemize}
    \end{itemize}
    
    \end{itemize}

\newpage
\subsection{Novel contributions of the paper}

\begin{itemize}
[noitemsep,topsep=2pt,leftmargin=!,labelwidth=\widthof{\itshape(iii)},font=\itshape] 
    \item[(i)] \textit{Unsteady case:} Compared to the contributions \cite{BBD15,berselli2023error}, 
    which derived \textit{a priori} error estimates for conforming, discretely inf-sup stable FE approximations of the 
    $p(\cdot)$-(Navier--)Stokes~equations, one novelty of the present paper is that it considers~the~unsteady~case. In fact, the present paper provides the first \textit{a priori} error analysis for a fully-discrete FE approximation of the unsteady $p(\cdot,\cdot)$-Stokes equations \eqref{eq:ptxStokes} as only weak convergence was established in the contributions  \cite{CHP10,berselli2024convergence} and only semi-discretizations were considered in the contribution \cite{din_diss};

    \item[(ii)] \textit{Incompressibility and pressure term:} Compared to the contribution \cite{ptx_lap}, which derived \textit{a priori} error estimates for a fully-discrete FE approximation of the $p(\cdot,\cdot)$-Laplace equation, one novelty of the present paper is that it also includes the incompressibility constraint~\eqref{eq:ptxStokes}$_2$~and~the~pressure~term; 
    

    \item[(iii)] \textit{Quasi-optimality:} In the case $p^-\ge 2$, we confirm the optimality of the derived error decay rates (with
   \hspace{-0.1mm}respect \hspace{-0.1mm}to \hspace{-0.1mm}natural \hspace{-0.1mm}fractional \hspace{-0.1mm}regularity \hspace{-0.1mm}assumptions \hspace{-0.1mm}on \hspace{-0.1mm}solutions)~\hspace{-0.1mm}via~\hspace{-0.1mm}\mbox{numerical}~\hspace{-0.1mm}\mbox{experiments}.
    In addition, imposing an alternative  natural fractional regularity assumption on the pressure, \textit{i.e.}, $(\delta+ \vert \mathbf{D}_{\mathrm{x}}\mathbf{v}\vert)^{2-p(\cdot)}\vert \nabla_{\mathrm{x}}^{\gamma_\mathrm{x}} q\vert^2 \in L^1(Q_T)$, $\gamma_{\mathrm{x}} \in (0,1]$, we derive an \textit{a priori} error estimate with an error decay rate that does not depend critically on the maximal (\textit{i.e.}, $p^+$) and minimal (\textit{i.e.}, $p^-$) value of the power-law index $p\in C^{0,\alpha_{\mathrm{t}},\alpha_{\mathrm{x}}}(Q_T)$, $\alpha\in (0,1]$, but is constant and also optimal.
\end{itemize}

\emph{The 
\hspace{-0.1mm}paper \hspace{-0.1mm}is \hspace{-0.1mm}organized \hspace{-0.1mm}as \hspace{-0.1mm}follows:} In \hspace{-0.1mm}Section \hspace{-0.1mm}\ref{sec:preliminaries}, \hspace{-0.1mm}we \hspace{-0.1mm}introduce \hspace{-0.1mm}the \hspace{-0.1mm}relevant~\hspace{-0.1mm}function~\hspace{-0.1mm}spaces~\hspace{-0.1mm}and~\hspace{-0.1mm}\mbox{notations}, and recall the related definitions with the results concerning the extra-stress tensor and  (generalized) $N$-functions. In Section \ref{sec:p-navier-stokes}, we introduce equivalent weak formulations of the unsteady $p(\cdot,\cdot)$-Stokes~equations \eqref{eq:ptxStokes}
and discuss natural fractional regularity assumptions~on~weak~solutions. In Section \ref{sec:discrete_p-navier-stokes}, we
introduce a fully-discrete FE approximation of the unsteady $p(\cdot, \cdot)$-Stokes equations \eqref{eq:ptxStokes} and  
 investigate certain FE projection operators for their~stability~properties. In Section \ref{sec:fractional_interpolation_estimates}, we~derive~several~fractional~interpola-tion error estimates for these FE projection operators. In Section \ref{sec:a_priori}, we state and prove~the~main result of the present paper concerning~\textit{a~priori}~error estimates for the numerical approximation of the unsteady $p(\cdot,\cdot)$-Stokes equations \eqref{eq:ptxStokes}. In Section~\ref{sec:exp}, we review the derived error decay rates obtained~in~Section~\ref{sec:a_priori}~for~their~optimality.

\section{Preliminaries}\label{sec:preliminaries}\enlargethispage{13mm}
	
	\hspace*{5mm}Throughout the entire paper, let $\Omega\subseteq \mathbb{R}^d$, $d\in \{2,3\}$, denote a bounded polyhedral Lipschitz~domain. The average of a function $f\colon \omega\to \mathbb{R}$ over a (Lebesgue) measurable set $\omega\subseteq \mathbb{R}^n$, $n\in \mathbb{N}$, with $\vert \omega\vert>0$\footnote{For a (Lebesgue) measurable set $\omega\subseteq \mathbb{R}^n$, $n\in \mathbb{N}$, its $n$-dimensional Lebesgue measure is denote by $\vert \omega\vert$.} is denoted by $\langle f\rangle_\omega\coloneqq {\frac 1 {|\omega|}\int_\omega f \,\textup{d}x}$. Moreover, for a (Lebesgue) measurable set $\omega\subseteq \mathbb{R}^n$,~$n\in \mathbb{N}$,~and~for (Lebesgue) measurable functions~${f,g\colon \omega\to \mathbb{R}}$, we write   ${(f,g)_\omega\coloneqq \int_\omega f g\,\textup{d}x}$~if~the~integral~is~\mbox{well-defined}.
 
	\subsection{\!Variable Lebesgue and Sobolev spaces, Nikolski\u{\i} spaces, and variable Calder\'on~spaces}
	
	\hspace*{5mm}Let $\omega\subseteq \mathbb{R}^n$, $n\in \mathbb{N}$, be a (Lebesgue) measurable set
	and $p\colon \omega\to [1,+\infty]$ a (Lebesgue) measurable function, a so-called  \textit{variable
		exponent}. By $\mathcal{P}(\omega)$, we denote the \textit{set of variable exponents}. Then, for $p\in \mathcal{P}(\omega)$, we denote by
	${p^+\coloneqq \textup{ess\,sup}_{x\in
			\omega}{p(x)}}$~and~${p^-\coloneqq \textup{ess\,inf}_{x\in
			\omega}{p(x)}}$ its constant~\textit{limit~exponents}. Moreover,~by
	$\mathcal{P}^{\infty}(\omega)\coloneqq \{p\in\mathcal{P}(\omega)\mid
	p^+<\infty\}$, we denote the \textit{set of bounded variable exponents}. For ${p\in\mathcal{P}^\infty(\omega)}$ and a (Lebesgue) measurable function $f\in L^0(\omega)$, we define the \textit{modular (with~respect~to~$p$)}~by 
	\begin{align*}
		\rho_{p(\cdot),\omega}(f)\coloneqq \int_{\omega}{\vert f\vert^{p(\cdot)}\,\mathrm{d}x}\,.
	\end{align*}
	Then, for  $p\in \mathcal{P}^\infty(\omega)$,
	the \textit{variable Lebesgue space} is defined by
	\begin{align*}
	\smash{L^{p(\cdot)}(\omega)\coloneqq \big\{ f\in L^0(\omega)\mid \rho_{p(\cdot),\omega}(f)<\infty\big\}\,.}
	\end{align*}
	The corresponding \textit{Luxembourg norm}, for every $f\in L^{p(\cdot)}(\omega)$ defined by 
	\[\| f\|_{p(\cdot),\omega}\coloneqq \inf\left\{\lambda> 0\;\bigg|\;  \rho_{p(\cdot),\omega}\left(\frac{f}{\lambda}\right)\leq 1\right\},
	\]
	turns $L^{p(\cdot)}(\omega)$~into~a~\mbox{Banach}~space (\textit{cf}.\ \cite[Thm.\  3.2.7]{dhhr}).\newpage
     
     For an open set $\omega\subseteq \mathbb{R}^n$, $n\in \mathbb{N}$, and  $p\in\mathcal{P}^\infty(\omega)$, the \textit{variable Sobolev space} is defined by
	\begin{align*}
		\smash{W^{1,p(\cdot)}(\omega)\coloneqq \big\{ f\in L^{p(\cdot)}(\omega)\mid\nabla f\in (L^{p(\cdot)}(\omega))^n\big\}\,.}
	\end{align*} 
	The \textit{variable Sobolev norm}  $\| \cdot\|_{1,p(\cdot),\omega}\coloneqq \| \cdot\|_{p(\cdot),\omega}+\| \nabla (\cdot)\|_{p(\cdot),\omega}$ 
	 turns $W^{1,p(\cdot)}(\omega)$ into a Banach space (\textit{cf}.\  \cite[Thm.\  8.1.6]{dhhr}). The closure of $C^\infty_c(\omega)$ in $W^{1,p(\cdot)}(\omega)$~is~denoted~by~$W^{1,p(\cdot)}_0(\omega)$.
	
	To describe the fractional regularity in space of the velocity vector field, we~resort~to~Nikolski\u{\i} spaces. 
    For an open set $\omega\hspace{-0.1em}\subseteq \hspace{-0.1em}\mathbb{R}^n$, $n\hspace{-0.1em}\in\hspace{-0.1em}\mathbb{N}$, $p \hspace{-0.1em}\in\hspace{-0.1em}  [1, \infty)$, $\beta\hspace{-0.1em}\in\hspace{-0.1em} (0,1]$, and  $f\hspace{-0.1em}\in \hspace{-0.1em}L^p(\omega)$, the  \textit{Nikolski\u{\i} semi-norm}~is~defined~by
	\begin{align}\label{eq:nikolski_semi-norm}
		[f]_{N^{\beta,p}(\omega)}\coloneqq {\sup}_{ \tau\in \mathbb{R}^n\setminus\{0\}}{\big\{\tfrac{1}{\vert \tau\vert^{\beta}}
        \|f(\cdot+  \tau)-f\|_{p,\omega\cap (\omega- \tau)}\big\}}<\infty\,.
	\end{align}
	The \textit{Nikolski\u{\i} norm} $\|\cdot\|_{N^{\beta,p}(\omega)}\coloneqq \| \cdot\|_{p,\omega}+[\cdot]_{N^{\beta,p}(\omega)}$ turns $N^{\beta,p}(\omega)$  into a Banach~space. 
	
	To describe fractional regularity in space of the kinematic pressure, we resort~to~variable~Calder\'on spaces (\textit{cf}.\  \cite{DeVoreSharpley1984}). For  an open set $\omega\subseteq \mathbb{R}^n$, $n\in \mathbb{N}$, $p\in \mathcal{P}^\infty(\omega)$, and $\gamma\in (0,1]$, a function  ${f\in L^{p(\cdot)}(\omega)}$ has a \textit{($\gamma$-order) upper Calder\'on gradient} if there exists a non-negative~\mbox{function}~${g\in L^{p(\cdot)}(\omega)}$~such~that 
	\begin{align}\label{eq:hajlasz_gradient}
		\vert f(x)-f(y)\vert \leq (g(x)+g(y))\,\vert x-y\vert^{\gamma}\quad\text{ for a.e. }x,y\in \omega\,.
	\end{align}
	For each function $f\in L^{p(\cdot)}(\omega)$,  the \textit{set of all ($\gamma$-order) upper Calder\'on gradients} (of $f$) is defined by $\mathrm{G}_{\gamma}(f)\coloneqq \{g\in L^{p(\cdot)}(\omega;\mathbb{R}_{\ge 0})\mid \eqref{eq:hajlasz_gradient}\text{ holds}\}$.
	Then, for $\gamma\in (0,1]$, the \textit{Calder\'on  space} is defined by
	\begin{align*}
		\smash{C^{\gamma,p(\cdot)}(\omega)\coloneqq \big\{f\in L^{p(\cdot)}(\omega)\mid  \mathrm{G}_{\gamma}(f)\neq \emptyset\big\}\,.}
	\end{align*}
	The  \textit{Calder\'on norm} 
	$\|\cdot\|_{\gamma,p(\cdot),\omega}\coloneqq \|\cdot
 \|_{p(\cdot),\omega}+\inf_{g\in \mathrm{G}_{\gamma}(\cdot)}{\big\{\|g\|_{p(\cdot),\omega}\big\}}$ turns $C^{\gamma,p(\cdot)}(\omega)$~into~a~\mbox{Banach}~space. 
 If $p^->1$, for every $f\in C^{\gamma,p(\cdot)}(\omega)$, we  denote by 
 \begin{align*}
     \vert \nabla^{\gamma} f\vert \coloneqq\textup{arg\,min}_{g\in \mathrm{G}_{\gamma}(f)}\big\{\|g\|_{p(\cdot),\omega}\big\}\in L^{p(\cdot)}(\omega)\,,
 \end{align*}
 the \textit{minimal $\gamma$-order Calder\'on gradient}.\enlargethispage{14mm}

	\subsection{(Generalized) $N$-functions}
	
	\hspace{5mm}We call a convex function $\psi\colon\mathbb{R}_{\geq 0} \to \mathbb{R}_{\geq 0}$ an
	\textit{$N$-function} if it satisfies ${\psi(0)=0}$,~${\psi(r)>0}$~for~all~${r>0}$,
	$\frac{\psi(r)}{r}\hspace{-0.1em}\to\hspace{-0.1em} 0$ $(r\hspace{-0.1em}\rightarrow\hspace{-0.1em} 0)$, and
	$\frac{\psi(r)}{r}\hspace{-0.1em}\to\hspace{-0.1em}+\infty$ $(r\hspace{-0.1em}\rightarrow\hspace{-0.1em}+\infty)$.  
    For~each \mbox{$N$-function} ${\psi \colon\hspace{-0.1em}\mathbb{R}_{\geq 0} \hspace{-0.1em}\to \hspace{-0.1em}\mathbb{R}_{\geq 0}}$, we define the corres-ponding 
    \textit{(Fenchel) conjugate \mbox{$N$-function}} $\psi^*\colon\hspace{-0.1em}\mathbb{R}_{\geq 0} \hspace{-0.1em}\to \hspace{-0.1em}\mathbb{R}_{\geq 0}$,~for~every~$r\hspace{-0.1em}\ge\hspace{-0.1em} 0$,~by
	${\psi^*(r)\hspace{-0.1em}\coloneqq\hspace{-0.1em} \sup_{s \geq 0} \{r\,s
		-\psi(s)\}}$. An $N$-function $\psi$ satisfies~the~\textit{$\Delta_2$-condition} if there exists a constant $c> 2$ such that  $\psi(2\,r) \leq c\, \psi(r)$ for~all~${r \ge 0}$. We shall denote the smallest such~constant~by~$\Delta_2(\psi)>0$.

   If $\psi,\psi^*\colon\mathbb{R}_{\ge 0}\to \mathbb{R}_{\ge 0}$ satisfy the $\Delta_2$-condition, then there holds 
	the following \textit{$\varepsilon$-Young inequality}: for every
	$\varepsilon> 0$, there exists a constant $c_\varepsilon>0 $, depending only on $\varepsilon>0$ and 
	$\Delta_2(\psi),\Delta_2( \psi ^*)<\infty$,~such~that for every $ r,s\geq0 $, there holds 
	\begin{align}
		\label{ineq:young}
		s\,r\leq  c_\varepsilon \,\psi^*(s)+\varepsilon \, \psi(r)\,.
	\end{align}
	For a (Lebesgue) measurable set $\omega\subseteq \mathbb{R}^n$, $n\in \mathbb{N}$, we call a function $\psi \colon \omega \times \mathbb{R}_{\ge 0} \to \mathbb{R}_{\ge 0}$  \textit{generalized $N$-function} if it is a Carath\'eodory mapping\footnote{For a (Lebesgue) measurable set $\omega\subseteq \mathbb{R}^n$, $n\in \mathbb{N}$, a function $\psi\colon \omega\times\mathbb{R}_{\ge 0}\to \mathbb{R}_{\ge 0}$ is called \textit{Carath\'eodory mapping} if $\psi(x,\cdot)\colon \mathbb{R}_{\ge 0}\to \mathbb{R}_{\ge 0}$ for a.e.\ $x\in \omega$ is continuous and $\psi(\cdot,t)\colon\omega\to \mathbb{R}_{\ge 0}$ for all $t\in   \mathbb{R}_{\ge 0}$ is (Lebesgue) measurable.} and $\psi(x,\cdot)\colon \mathbb{R}_{\ge 0}\to \mathbb{R}_{\ge 0}$ is an
	$N$-function for a.e.\ $x \in \omega$. For a (Lebesgue) measurable function $f\in L^0(\omega)$ and a generalized $N$-function $\psi \colon  \omega\times\mathbb{R}_{\ge 0} \to\mathbb{R}_{\ge 0}$,   the \textit{modular (with respect to $\psi$)} is defined by
	\begin{align*}
		\rho_{\psi,\omega}(f)\coloneqq \int_\omega \psi(\cdot,\vert f\vert )\,\textup{d}x\,.
	\end{align*} For a generalized 
	$N$-function  $\psi \colon  \omega\times\mathbb{R}_{\ge 0} \to\mathbb{R}_{\ge 0}$,  the \textit{generalized Orlicz space} is defined by 
	\begin{align*}
		L^{\psi}(\omega)\coloneqq \big\{f\in L^0(\omega)\mid
		\rho_{\psi,\omega}(f)<\infty\big\}\,.
	\end{align*}
	The \textit{Luxembourg norm}, for every $f\in L^{\psi}(\omega) $ defined by 
	\[\smash{\norm {f}_{\psi,\omega}}\coloneqq  \inf \bigg\{\lambda >0\;\bigg|\;
			\rho_{\psi,\omega}\bigg(\frac{f}{\lambda}\bigg) \le 1\bigg\}\,,\] turns
	$L^\psi(\omega)$ 
    into a Banach space (\textit{cf}.\ \cite[Thm.\  2.3.13]{dhhr}).\vspace{-1mm} 

    \subsection{Bochner--Nikoski\u{\i} spaces}

    \hspace{5mm}To \hspace{-0.1mm}describe \hspace{-0.1mm}fractional \hspace{-0.1mm}regularity \hspace{-0.1mm}in \hspace{-0.1mm}time \hspace{-0.1mm}of \hspace{-0.1mm}the \hspace{-0.1mm}velocity \hspace{-0.1mm}vector \hspace{-0.1mm}field,  \hspace{-0.1mm}we \hspace{-0.1mm}resort \hspace{-0.1mm}to \hspace{-0.1mm}\mbox{Bochner--Nikolski\u{\i}} spaces. 
    For a (real) Banach space $(X,\|\cdot\|_X)$,  an open interval $J\subseteq \mathbb{R}$, $p\in [1,\infty)$, and $f\in L^p(J;X)$, the \emph{Bochner--Nikolski\u{\i} semi-norm} is defined by
    \begin{align*}
        [f]_{N^{\beta,p}(J;X)}\coloneqq {\sup}_{ \tau\in \mathbb{R}\setminus\{0\}}{\big\{\tfrac{1}{\vert \tau\vert^\beta}\|f(\cdot+  \tau)-f\|_{L^p(J\cap (J-\tau);X)}\big\}}<\infty\,.
    \end{align*} 
    Then, the \textit{Bochner--Nikolski\u{\i} space} is defined by
	\begin{align*}
		\smash{N^{\beta,p}(J;X)\coloneqq \big\{ f\in L^p(J;X)\mid [f]_{N^{\beta,p}(J;X)}<\infty\big\}\,.}
	\end{align*}
	The \hspace{-0.1mm}\textit{Bochner--Nikolski\u{\i} \hspace{-0.1mm}norm} \hspace{-0.1mm}$\|\cdot\|_{N^{\beta,p}(J;X)}\hspace{-0.15em}\coloneqq\hspace{-0.15em} \| \cdot\|_{L^p(J;X)}+[\cdot]_{N^{\beta,p}(J;X)}$ \hspace{-0.1mm}turns  $\smash{N^{\beta,p}(J;X)}$~\hspace{-0.1mm}into~\hspace{-0.1mm}a~\hspace{-0.1mm}\mbox{Banach}~\hspace{-0.1mm}space.

	\subsection{Basic properties of the non-linear operators}\label{sec:basic} 
	
	\hspace*{5mm} For a (Lebesgue) measurable set $\omega\hspace{-0.15em}\subseteq\hspace{-0.15em} \mathbb{R}^n$, $n\hspace{-0.15em}\in\hspace{-0.15em} \mathbb{N}$, and  (Lebesgue) measurable~functions~${f,g\colon\hspace{-0.15em} \omega \hspace{-0.15em}\to\hspace{-0.15em} \mathbb{R}_{\ge 0}}$, we write
	$f\sim g$ (or $f\lesssim g$) if~there~exists a constant $c>0$ such that $c^{-1}\,g\leq f\leq c\,g$ a.e.\ in $\omega$ (or $f\leq c\,g$ a.e.\ in $\omega$). In~particular, if not otherwise stated, we always assume that the implicit constants in `$\sim$' and `$\lesssim$'   depend only on $p^-,p^+>1$, $\delta\ge 0$, and $\nu_0 >0$.
	
	Throughout the entire paper, for the extra-stress tensor $\bfS\colon Q_T\times \mathbb{R}^{d\times d}\to \mathbb{R}^{d\times d}_{\textup{sym}}$, 
    we~will~assume~that there exist $p\hspace{-0.1em}\in\hspace{-0.1em} \mathcal{P}^{\infty}(Q_T)$ with $p^-\hspace{-0.1em}>\hspace{-0.1em}1$,  $\nu_0\hspace{-0.1em}>\hspace{-0.1em}0$, and $\delta\hspace{-0.1em}\ge\hspace{-0.1em} 0$ such that  for a.e.\ $(t,x)^\top\hspace{-0.1em}\in\hspace{-0.1em} Q_T$~and~every~${\bfA\hspace{-0.1em}\in \hspace{-0.1em}\mathbb{R}^{d\times d}}$, we have that
	\begin{align}\label{def:A}
		\bfS(t,x,\bfA)\coloneqq \nu_0\,(\delta+\vert \bfA^{\textup{sym}}\vert )^{p(t,x)-2}\bfA^{\textup{sym}}\,.
	\end{align} 
	For the same $p\in \mathcal{P}^{\infty}(Q_T)$,  $\nu_0>0$, and $\delta\ge 0$ as in the definition \eqref{def:A}, we introduce the \textit{special generalized $N$-function}~$\varphi\colon Q_T\times\mathbb{R}_{\ge 0}\to \mathbb{R}_{\ge 0}$, for a.e.\ $(t,x)^\top\in Q_T$ and every $r\ge 0$, defined by
	\begin{align} 
		\label{eq:def_phi} 
		\varphi(t,x,r)\coloneqq \int _0^r \varphi'(t,x,s)\, \mathrm ds\,,\quad\text{where}\quad
		\varphi'(t,x,r) \coloneqq (\delta +r)^{p(t,x)-2} r\,.
	\end{align}
	For a given generalized $N$-function $\psi\colon Q_T\times \mathbb{R}_{\ge 0}\to \mathbb{R}_{\ge 0}$, let us introduce \textit{shifted generalized $N$-functions} $\psi_a\colon Q_T\times \mathbb{R}_{\ge 0}\to \mathbb{R}_{\ge 0}$, ${a\ge 0}$, for a.e.\ $(t,x)^\top\in Q_T$ and every $a,r\ge 0$, defined by
	\begin{align}
		\label{eq:phi_shifted}
		\psi_a(t,x,r)\coloneqq \int _0^t \psi_a'(t,x,s)\, \mathrm ds\,,\quad\text{where}\quad
		\psi'_a(t,x,r)\coloneqq \psi'(t,x,a+r)\frac {r}{a+r}\,.
	\end{align}
	
	\begin{remark} \label{rem:phi_a}
		Based on the description above, for the special $N$-function $\varphi\colon Q_T\times \mathbb{R}_{\ge 0}\to \mathbb{R}_{\ge 0}$ defined in \eqref{eq:def_phi}, uniformly with respect to ${a,r\ge 0}$ and  a.e.\ $(t,x)^\top\in Q_T$, we have that
		\begin{align}
			\varphi_a(t,x,r)& \sim (\delta+a+r)^{p(t,x)-2} r^2\,,\label{rem:phi_a.1}\\
			(\varphi_a)^*(t,x,r)
			&	\sim ((\delta+a)^{p(t,x)-1} + r)^{\smash{p'(t,x)-2}} r^2\,.\label{rem:phi_a.2}
		\end{align}
    Note that the families $\{\varphi_a\}_{\smash{a \ge 0}},\{(\varphi_a)^*\}_{\smash{a \ge 0}}\colon Q_T\times  \mathbb{R}_{\ge 0}\to \mathbb{R}_{\ge 0}$, uniformly with respect to ${a\ge 0}$ and  a.e.\ $(t,x)^\top\in Q_T$, satisfy the $\Delta_2$-condition 
    with 
		\begin{align}
			\textup{ess\,sup}_{(t,x)^\top\in Q_T}{\big\{{\sup}_{a \ge  0}{\{\Delta_2(\varphi_a(t,x,\cdot))\}}\big\}}& \lesssim 2^{\max \{2,p^+\}}\,,\label{rem:phi_a.3}\\
			\textup{ess\,sup}_{(t,x)^\top\in Q_T}{\big\{{\sup}_{a \ge  0}{\big\{\Delta_2((\varphi_a)^*(t,x,\cdot))\}}\big\}} &\lesssim 2^{\max \{2,(p^-)'\}}\,.\label{rem:phi_a.4}
		\end{align}
\end{remark}
	
	Now, motivated by the definition \eqref{def:A} of the extra-stress tensor $\bfS\colon Q_T\times\mathbb{R}^{d\times d}\to \mathbb{R}^{d\times d}_{\textup{sym}}$ , we consider the  mappings $\bfF,\bfF^*\colon Q_T\times\mathbb{R}^{d\times d}\to \mathbb{R}^{d\times d}_{\textup{sym}}$, for a.e.\ $(t,x)^\top\in Q_T$~and~every~${\bfA\in \mathbb{R}^{d\times d}}$,~defined~by
	\begin{align}
		\begin{aligned}
			\bfF(t,x,\bfA)&\coloneqq (\delta+\vert \bfA^{\textup{sym}}\vert)^{\smash{\frac{p(t,x)-2}{2}}}\bfA^{\textup{sym}}\,,\\ \bfF^*(t,x,\bfA)&\coloneqq (\delta^{p(t.x)-1}+\vert \bfA^{\textup{sym}}\vert)^{\smash{\frac{p'(t,x)-2}{2}}}\bfA^{\textup{sym}}
			\,.\label{eq:def_F}
		\end{aligned}
	\end{align}
    
    The relations between 
	$\bfS,\bfF,\bfF^*\colon Q_T\times\mathbb{R}^{d\times d}
	\to \mathbb{R}^{d\times d}_{\textup{sym}}$ and
	$\varphi_a,(\varphi^*)_a,(\varphi_a)^*\colon Q_T\times\mathbb{R}_{\ge
		0}\to \mathbb{R}_{\ge
		0}$,~${a\ge 0}$, are summarized in
	the following proposition.
	
	\begin{proposition}
		\label{lem:hammer}
		For every  $\bfA, \bfB \in \mathbb{R}^{d\times d}$, $r\ge 0$, and  a.e.\ $ (t,x)^\top, (t',x')^\top\in Q_T$, we have that
		\begin{align}
			(\bfS(t,x,\bfA) - \bfS(t,x,\bfB))
			\cdot(\bfA-\bfB ) &\sim \smash{\vert \bfF(t,x,\bfA) - \bfF(t,x,\bfB)\vert^2}\notag
			\\&\sim \varphi_{\vert \bfA^{\textup{sym}} \vert }(t,x,\vert \bfA^{\textup{sym}} - \bfB^{\textup{sym}} \vert )\label{eq:hammera}
			\\&\sim (\varphi_{\vert \bfA^{\textup{sym}} \vert })^*(t,x,\vert \bfS(t,x,\bfA)-\bfS(t,x,\bfB)\vert )
			\,,\notag\\
			\smash{\vert \bfF^*(t,x,\bfA) - \bfF^*(t,x,\bfB)\vert^2}
			\label{eq:hammerf}
			&\sim \smash{\smash{(\varphi^*)}_{\smash{\vert \bfA^{\textup{sym}} \vert }}(t,x,\vert \bfA^{\textup{sym}}  - \bfB^{\textup{sym}} \vert )}\,,\\
			\label{eq:hammerg}
			\smash{\smash{(\varphi^*)}_{\smash{\vert \bfS(t,x,\bfA)\vert }}(t,x,r)}
			&\sim \smash{\smash{(\varphi}_{\smash{\vert \bfA^{\textup{sym}} \vert }})^*(t,x,r)}\,,\\
			\label{eq:hammerh}
			\smash{\vert \bfF^*(t,x,\bfS(t,x,\bfA)) - \bfF^*(t,x,\bfS(t',x',\bfB))\vert^2}
			&\sim \smash{\smash{(\varphi}_{\smash{\vert \bfA^{\textup{sym}} \vert }})^*(t,x,\vert \bfS(t,x,\bfA)-\bfS(t',x',\bfB)\vert)}\,.
		\end{align}
	\end{proposition} 
	
	\begin{proof}
        For the equivalences \eqref{eq:hammera}--\eqref{eq:hammerg}, see  \cite[Rem.\ A.9]{BDS15}. 
         The equivalence \eqref{eq:hammerh} is a consequence of the equivalence \eqref{eq:hammerf} together with equivalence \eqref{eq:hammerg}.
	\end{proof}
	
	Furthermore, throughout the entire paper, we will frequently utilize the following $\varepsilon$-Young~type~result on a change of shift in generalized $N$-functions.
	
	\begin{lemma}\label{lem:shift-change}
		For each $\varepsilon>0$, there exists a constant $c_\varepsilon \geq 1$, depending on $\varepsilon>0$, $p^-,p^+>1$, and $\delta\ge 0$, such that for every $\bfA, \bfB \in \mathbb{R}^{d\times d}$, $r\ge 0$, 
		 and a.e.\ $  (t,x)^\top\in Q_T$, we have that
		\begin{align}
			\varphi_{\vert \bfA^{\textup{sym}} \vert}(t,x,r)&\leq c_\varepsilon\, \varphi_{\vert \bfB^{\textup{sym}} \vert }(t,x,r)
			+\varepsilon\, \vert \bfF(t,x,\bfA) - \bfF(t,x,\bfB)\vert^2\,,\label{lem:shift-change.1}
			\\
			(\varphi_{\vert \bfA^{\textup{sym}} \vert})^*(t,x,r)&\leq c_\varepsilon\, (\varphi_{\vert \bfB^{\textup{sym}} \vert })^*(t,x,r)
			+\varepsilon\, \vert \bfF(t,x,\bfA) - \bfF(t,x,\bfB)\vert^2\,.\label{lem:shift-change.3}
		\end{align}
	\end{lemma}

    \begin{proof}
        See \cite[Rem.\ A.9]{BDS15}.
    \end{proof}

	\subsection{Log-Hölder continuity and important related results}\enlargethispage{5mm}
	
	\hspace*{5mm}In this section, we discuss the minimum regularity requirement on a variable power-law index for some relevant function space theory, which is known as  $\log$-H\"older 
continuity, and collect some related results that will be used in the later analysis. 

    A bounded exponent $p\in \mathcal P^\infty (Q_T)$ is 
	\textit{$\log$-Hölder continuous}, written $p\in \mathcal P^\textup{log}(Q_T)$, if there exists a  constant $c>0$ such that for every $(t,x)^\top,(t',x')^\top\in Q_T$ with $0<
    \vert t-t'\vert+\vert x-x'\vert
     <\frac{1}{2}$, there holds
	\begin{align}\label{def:parabolic_log_hoelder}
		\vert p(t,x)-p(t',x')\vert \leq \frac{c}{-\log( \vert t-t'\vert+\vert x-x'\vert)}\,.
	\end{align}
	The smallest  constant $c>0$ such that \eqref{def:parabolic_log_hoelder} holds is called  \textit{$\log$-Hölder constant} and is~denoted~by~$[p]_{\log,Q_T}$.
		
Let us recall the following  substitute of Jensen's inequality for shifted generalized $N$-functions, where the shift is constant, the so-called \textit{key estimate}.	

\begin{lemma}[key estimate]\label{lem:key-estimate}
Assume that $p\in \mathcal{P}^{\log}(Q_T)$. Then, for each $n>0$, there exists a constant $c>0$, depending on $n$, $[p]_{\log,Q_T}$, and $p^-$, such that for every cube (or ball) $Q\subseteq Q_T$ with side-length $\ell(Q)\leq 1$, $a\ge 0$, $f\in L^{p'(\cdot,\cdot)}(Q)$ with $a+\langle \vert f\vert \rangle_Q\leq \vert Q\vert^{-n}$, and for a.e.\ $(t,x)^\top\in Q$, there holds
			\begin{align*}
				(\varphi_a)^*(t,x,\langle  \vert f\vert \rangle_Q)\leq c\,\big( \langle (\varphi_a)^*(t,\cdot,\vert f\vert )\rangle_Q+\vert Q\vert^n\big)\,.
			\end{align*}
		\end{lemma}
		
		\begin{proof}
            See \cite[Lem.\ 2.24]{berselli2023error}.
\end{proof}

To derive \textit{a priori} error estimates for a fully-discrete FE approximation~of~the unsteady $p(\cdot,\cdot)$-Stokes \hspace{-0.1mm}equations \hspace{-0.1mm}\eqref{eq:ptxStokes}, \hspace{-0.1mm}however, \hspace{-0.1mm}we \hspace{-0.1mm}need \hspace{-0.1mm}more \hspace{-0.1mm}regularity \hspace{-0.1mm}than \hspace{-0.1mm}just \hspace{-0.1mm}$\log$-Hölder~\hspace{-0.1mm}\mbox{continuity}.~\hspace{-0.1mm}More~\hspace{-0.1mm}\mbox{precisely}, we will need that power-law index is \textit{parabolically Hölder continuous}, \textit{i.e.}, $p\hspace{-0.1em}\in \hspace{-0.1em}C^{0,\alpha_{\mathrm{t}},\alpha_{\mathrm{x}}}(\overline{Q_T})$,~${\alpha_{\mathrm{t}},\alpha_{\mathrm{x}}\hspace{-0.1em}\in\hspace{-0.1em} (0,1]}$, which means that there exists a constant $c>0$ such that for every $(t,x)^\top,(\tilde{t},\tilde{x})^\top\in \overline{Q_T}$, there holds
\begin{align}\label{def:parabolic_hoelder}
    \vert p(t,x)-p(\tilde{t},\tilde{x})\vert \leq c\,\big(\vert t-\tilde{t}\vert^{\alpha_{\mathrm{t}}}+\vert x-\tilde{x}\vert^{\alpha_{\mathrm{x}}}\big)\,.
\end{align}
The smallest constant $c>0$ such that \eqref{def:parabolic_hoelder} holds is called \textit{$(\alpha_{\mathrm{t}},\alpha_{\mathrm{x}})$-Hölder semi-norm} and is denoted by $[p]_{\alpha_{\mathrm{t}},\alpha_{\mathrm{x}},Q_T}$.
Note that each parabolically Hölder continuous power-law index is also $\log$-Hölder~continuous and the $\log$-Hölder constant can be estimated by the $(\alpha_{\mathrm{t}},\alpha_{\mathrm{x}})$-Hölder semi-norm.
 
	\section{The  unsteady $p(\cdot,\cdot)$-Stokes equations}\label{sec:p-navier-stokes}\vspace*{-1mm}\enlargethispage{7.5mm}
	
	\hspace{5mm}In this section, we discuss different weak formulations of the unsteady $p(\cdot,\cdot)$-Stokes equations~\eqref{eq:ptxStokes}. For a relevant analytical study, we refer the reader to the contributions \cite{seungchan-thesis-article,alex-book}.\vspace*{-1mm}

\subsection{Weak formulations}\label{sec:wf}\vspace{-1mm}
	
	\hspace*{5mm}
	Let us first introduce the following function spaces:
	\begin{align*} 
			\mathbfcal{V}
   &\coloneqq \big\{\bfv \in (L^{p(\cdot,\cdot)}(Q_T))^d\mid \bfD_{\mathrm{x}}\bfv \in (L^{p(\cdot,\cdot)}(Q_T))^{d\times d}\,,\; \bfv(t)\in (W^{1,p(t,\cdot)}_0(\Omega))^d\text{ for a.e.\ }t\in I\big\}\,,\\
   \mathbfcal{Q}&\coloneqq \smash{W^{-1,\infty}(I;L^{(p^+)'}(\Omega))}\,.  
	\end{align*}
    \hspace{5mm}The corresponding weak formulation of the unsteady $p(\cdot,\cdot)$-Stokes equations \eqref{eq:ptxStokes}~as~a~parabolic non-linear saddle point like problem is the following:
	
	\textit{Problem (Q).}\hypertarget{Q}{} For  given $\bfg\hspace{-0.1em}\in\hspace{-0.1em} (L^{(p^-)'}(Q_T))^d$, $\bfG\hspace{-0.1em}\in\hspace{-0.1em} (L^{p'(\cdot,\cdot)}(Q_T))^{d\times d}_{\textrm{sym}}$, and $\bfv_0\hspace{-0.1em}\in \hspace{-0.1em}H$, find ${(\bfv,q)^\top\hspace{-0.1em}\in\hspace{-0.1em} \mathbfcal{V}\hspace{-0.1em}\times\hspace{-0.1em} \mathbfcal{Q}(0)}$ \hphantom{.}\hspace{5mm}with $\bfv(0)=\bfv_0$ in $H$ such that for every $(\boldsymbol{\varphi},\eta)^\top\in (C^\infty_c(Q_T))^d\times C^\infty_c(Q_T) $, there holds
	\begin{align*}
		\langle \partial_{\mathrm{t}}\bfv,\boldsymbol{\varphi}\rangle_{Q_T}+(\bfS(\cdot,\cdot,\bfD_{\mathrm{x}}\bfv),\bfD_{\mathrm{x}}\boldsymbol{\varphi})_{Q_T}-\langle q,\mathrm{div}_{\mathrm{x}} \boldsymbol{\varphi}\rangle_{Q_T}&=(\bfg,\boldsymbol{\varphi})_{Q_T}+(\bfG,\bfD_{\mathrm{x}}\boldsymbol{\varphi})_{Q_T}\,,\\
		( \eta,\mathrm{div}_{\mathrm{x}} \bfv)_{Q_T}&=0\,, 
	\end{align*}
    \hspace{5mm}where $\mathbfcal{Q}(0)\coloneqq \smash{W^{-1,\infty}(I;L^{(p^+)'}_0(\Omega))}$ and $H\coloneqq \smash{\overline{\{\boldsymbol{\varphi}\in (C_c^\infty(\Omega))^d\mid \textup{div}_{\mathrm{x}}\boldsymbol{\varphi}=0\textup{ in }\Omega\}}^{\smash{\|\cdot\|_{2,\Omega}}}}$. 
    
	\hspace{-5mm}Equivalently, one can reformulate Problem (\hyperlink{P}{P}) in a hydro-mechanical sense (\textit{i.e.},  hiding the pressure):
	
	\textit{Problem (P).}\hypertarget{P}{} For  given $\bfg\hspace{-0.1em}\in\hspace{-0.1em} (L^{(p^-)'}(Q_T))^d$, $\bfG\hspace{-0.1em}\in\hspace{-0.1em} (L^{p'(\cdot,\cdot)}(Q_T))^{d\times d}_{\textrm{sym}}$, and $\bfv_0\hspace{-0.1em}\in \hspace{-0.1em}H$, find $\bfv\in \mathbfcal{V}(0)$ with\linebreak \hphantom{.}\hspace{5mm}$\bfv(0)=\bfv_0$~in~$H$~such~that for every $\boldsymbol{\varphi}\in (C^\infty_c(Q_T))^d$ with $\textup{div}_{\mathrm{x}}\boldsymbol{\varphi}=0$ in $Q_T$, there holds
	\begin{align*}
		\langle \partial_{\mathrm{t}}\bfv,\boldsymbol{\varphi}\rangle_{Q_T}+(\bfS(\cdot,\cdot,\bfD_{\mathrm{x}}\bfv),\bfD_{\mathrm{x}}\boldsymbol{\varphi})_{Q_T}&=(\bfg,\boldsymbol{\varphi})_{Q_T}+(\bfG,\bfD_{\mathrm{x}}\boldsymbol{\varphi})_{Q_T}\,,
	\end{align*}
	\hspace{5mm}where $\mathbfcal{V}(0)\coloneqq \big\{\boldsymbol{\varphi}\in \mathbfcal{V}\mid \mathrm{div}_{\mathrm{x}}\boldsymbol{\varphi}=0\text{ a.e.\ in }Q_T\big\}$.\enlargethispage{1mm}
	
	In the case $p^->\frac{2d}{d+2}$, the well-posedness of Problem (\hyperlink{Q}{Q}) and Problem (\hyperlink{P}{P}) is proved~in~two~steps:~first, using~pseudo-monotone operator theory (\textit{cf}.\ \cite[Thm.\ 8.2]{alex-book}), the well-posedness of Problem~(\hyperlink{P}{P})~is~shown; given the well-posedness of Problem (\hyperlink{P}{P}), the well-posedness of Problem (\hyperlink{Q}{Q})~follows~as~in~\mbox{\cite[Prop.~6.1]{alex-book}}.

    \begin{remark}[equivalent weak formulation]\label{rem:equiv_form}
        Problem (\hyperlink{Q}{Q}) can equivalently be reformulated as follows:\\[1mm]
        For given $\bfg\in (L^{(p^-)'}(Q_T))^d$, $\bfG\in (L^{p'(\cdot,\cdot)}(Q_T))^{d\times d}_{\textup{sym}}$, and $\bfv_0\in H$, find $(\bfv,q)^\top\in \mathbfcal{V}\times \mathbfcal{Q}(0)$ such that for every interval $J\subseteq I$ and $(\boldsymbol{\varphi}_J,\eta_J)^\top\in (W^{1,\smash{p_J^+}(\cdot)}_0(\Omega))^d\times L^{\smash{(p_J^-)'}(\cdot)}_0(\Omega)$, where $p_J^+\coloneqq \sup_{t\in J}{p(t,\cdot)}$ and $p_J^-\coloneqq \inf_{t\in J}{p(t,\cdot)}$~in~$\Omega$, setting $Q_J\coloneqq J\times \Omega$,
        there holds
        \begin{align*}
            (\bfv(\textup{sup}\,J),\boldsymbol{\varphi}_J)_{\Omega}\hspace{-0.1em}+\hspace{-0.1em}(\bfS(\cdot,\cdot,\bfD_{\mathrm{x}}\bfv),\bfD_{\mathrm{x}}\boldsymbol{\varphi}_J)_{Q_J}\hspace{-0.1em}+\hspace{-0.1em}(q,\textup{div}_{\mathrm{x}}\boldsymbol{\varphi}_J)_{Q_J}&\hspace{-0.1em}=\hspace{-0.1em} (\bfv(\textup{inf}\,J),\boldsymbol{\varphi})_{\Omega}\hspace{-0.1em}+\hspace{-0.1em}(\bfg,\boldsymbol{\varphi}_J)_{Q_J}\hspace{-0.1em}+\hspace{-0.1em}(\bfG,\bfD_{\mathrm{x}}\boldsymbol{\varphi}_J)_{Q_J}\,,\\
            (\eta_J,\textup{div}_{\mathrm{x}}\bfv)_{Q_J}&\hspace{-0.1em}=\hspace{-0.1em}0\,.
        \end{align*}\vspace{-7.5mm} 
    \end{remark}

    	\subsection{Regularity assumptions}\vspace{-1mm}

	\hspace*{5mm}In accordance with  \cite[Thm.\ 4.1]{ptx_lap}, for the velocity vector field, 
 it is reasonable~to~expect~the~fractional regularity 
\begin{align}\label{eq:natural_regularity.velocity}
    \hspace{-1.5mm}\left.\begin{aligned} 
        \bfF(\cdot,\cdot,\bfD_{\mathrm{x}}\bfv)&\in N^{\beta_{\mathrm{t}},2}(I;(L^2(\Omega))^{d\times d})\cap L^2(I;(N^{\beta_{\mathrm{x}},2}(\Omega))^{d\times d}) \,,\\
	\bfv&\in L^\infty(I;(N^{\beta_{\mathrm{x}},2}             (\Omega))^d)
    \end{aligned}\;\right\}\;\text{ for some }\beta_{\mathrm{t}}\in (\tfrac{1}{2},1]\,,\;\beta_{\mathrm{x}}\in (0,1] \,.
 \end{align} 
For the kinematic pressure, we propose the fractional regularity
 $q(t)\in C^{\gamma_{\mathrm{x}},p'(t,\cdot)}(\Omega)$~for~a.e.~${t\in I}$~with
 \begin{align}\label{eq:natural_regularity.pressure}
  \vert \nabla_{\mathrm{x}}^{\gamma_{\mathrm{x}}} q\vert \in L^{p'(\cdot,\cdot)}(Q_T)\quad\text{ for some }\gamma_{\mathrm{x}}\in (0,1]\,. 
 \end{align}

An important consequence of the regularity assumption \eqref{eq:natural_regularity.velocity}$_1$ is an improved integrability result.
	
	\begin{lemma}\label{lem:improved_integrability}
		Let $p\in C^0(\overline{Q_T})$ with $p^->1$ and let $\boldsymbol{\varphi} \in \mathbfcal{V}$ with $\bfF(\cdot
        ,\cdot,\bfD_{\mathrm{x}}\boldsymbol{\varphi}) \in N^{\beta_{\mathrm{t}},2}(I;(L^2(\Omega))^{d\times d})\cap L^2(I;(N^{\beta_{\mathrm{x}},2}(\Omega))^{d\times d})$, $\beta_{\mathrm{t}}\in (\frac{1}{2},1]$, $\beta_{\mathrm{x}}\in (0,1]$. Then, there exists a constant $s>1$~such~that~$\bfF(\cdot,\cdot,\bfD_{\mathrm{x}}\boldsymbol{\varphi})\in C^0(\overline{I};(L^{2s}(\Omega))^{d\times d})$.
        In particular, we have that $\smash{\sup_{t\in I}{\big\{\|\bfD_{\mathrm{x}}\boldsymbol{\varphi}\|_{p(t,\cdot)s,\Omega}\big\}}<\infty}$.
	\end{lemma}
	
	\begin{proof}   
        We proceed analogously to \cite{ptx_lap}, \textit{i.e.}, resorting to  \cite[Lem.\ 2.9]{ptx_lap} (with $\theta\in (0,1)$~such~that~$\theta \beta_{\mathrm{t}}> \frac{1}{2}$) together with \cite[Lem.\ 2.5]{ptx_lap}. 
	\end{proof}

    \begin{remark}\label{rem:reg_initial}
        By Lemma \ref{lem:improved_integrability}, if $p\hspace{-0.15em}\in\hspace{-0.15em} C^0(\overline{Q_T})$ with $p^-\hspace{-0.15em}>\hspace{-0.15em}1$, from \eqref{eq:natural_regularity.velocity}, it follows that $\bfv\hspace{-0.15em}\in \hspace{-0.15em}C_{w}^0(\overline{I};(N^{\beta_{\mathrm{x}},2}             (\Omega))^d)$ and, thus, automatically $\bfv_0=\bfv(0)\in (N^{\beta_{\mathrm{x}},2}             (\Omega))^d$.
    \end{remark}

	The following lemma shows that in the case $\alpha_{\mathrm{t}}=\alpha_{\mathrm{t}}=\beta_{\mathrm{t}}=\beta_{\mathrm{x}}=1$ in \eqref{eq:natural_regularity.velocity}, 
   and~if~${p^-\ge  2}$~and~${\delta>0}$, we have 
    that $\gamma_{\mathrm{x}}=1$ in \eqref{eq:natural_regularity.pressure}.\enlargethispage{13.5mm}

	\begin{lemma}\label{lem:pres}
		Let $p\in C^{0,1,1}(\overline{Q_T})$ with $p^-\ge 2$ and $\delta>0$. Moreover, let $(\bfv,q)^\top \in \mathbfcal{V}\times \mathbfcal{Q}(0)$~be~a~solution of Problem~(\hyperlink{Q}{Q}) with
		\begin{align*} 
        \bfF(\cdot,\cdot,\bfD_{\mathrm{x}}\bfv)&\in W^{1,2}(I;(L^2(\Omega))^{d\times d})\cap L^2(I;(W^{1,2}(\Omega))^{d\times d}) \,,\\
	\bfv&\in L^\infty(I;(W^{1,2}             (\Omega))^d)\,.
        \end{align*} 
 Then, the following~statements~apply:
		\begin{itemize}[noitemsep,topsep=2pt,leftmargin=!,labelwidth=\widthof{(ii)}]
			\item[(i)]  If $\bff\in (L^{p'(\cdot,\cdot)}(Q_T))^d$, then we have that $q(t,\cdot)\in W^{1,p'(t,\cdot)}(\Omega)$ for a.e.\ $t\in I$ with 
   $\vert \nabla_{\mathrm{x}}q\vert \in  L^{p'(\cdot,\cdot)}(Q_T)$. 
			\item[(ii)] If $(\delta+\vert \bfD_{\mathrm{x}}\bfv\vert)^{2-p(\cdot,\cdot)}\vert\bff\vert^2\in L^1(Q_T)$, then  we have that $q(t,\cdot)\in W^{1,p'(t,\cdot)}(\Omega)$ for a.e.\ $t\in I$ with 
   $(\delta    +\vert\bfD_{\mathrm{x}}\bfv\vert)^{2-p(\cdot,\cdot)}\vert\nabla_{\mathrm{x}} q\vert ^2 \in L^1(Q_T)$.
		\end{itemize}
	\end{lemma} 

    \begin{remark}
        If $p^-\ge 2$ and $ \delta>0 $, from $\bff\in  (L^2(Q_T))^d$, it follows that ${(\delta+\vert \bfD_{\mathrm{x}}\bfv\vert)^{2-p(\cdot,\cdot)}\vert\bff\vert^2\in L^1(Q_T)}$.
    \end{remark} 
	\begin{proof}
       
		\textit{ad (i).} Analogously to \cite[Lems.\ 2.45-2.47]{BK23_pxDirichlet}, abbreviating  ${\smash{\widehat{\bfS}}\coloneqq \bfS(\cdot,\cdot,\bfD_{\mathrm{x}}\bfv)\in (L^{p'(\cdot,\cdot)}(Q_T))^{d\times d}}$, 
        we deduce that $\bfF^*(\cdot,\cdot,\bfD_{\mathrm{x}}\bfv)\in W^{1,2}(I;(L^2(\Omega))^{d\times d})\cap  L^2(I;(W^{1,2}(\Omega))^{d\times d}) $ with\vspace{-1mm}
		\begin{alignat}{2}
			\vert\nabla_{\mathrm{x}} \bfF(\cdot,\cdot,\bfD_{\mathrm{x}}\bfv)\vert+(1+\vert \bfD_{\mathrm{x}}\bfv\vert^{p(\cdot,\cdot)s}) &\sim \vert\nabla_{\mathrm{x}}  \bfF^*(\cdot,\cdot,\smash{\widehat{\bfS}})\vert
			+(1+\vert\smash{\widehat{\bfS}}\vert^{p'(\cdot,\cdot)s})&&\quad \text{ a.e.\ in }Q_T\,,\label{lem:pres.1}\\
			\vert \nabla_{\mathrm{x}}\bfF(\cdot,\cdot,\bfD_{\mathrm{x}}\bfv)\vert^2+\mu_{\mathrm{x}}(\bfv)&\sim (\delta+\vert \bfD_{\mathrm{x}}\bfv\vert)^{p(\cdot,\cdot)-2}\vert \nabla_{\mathrm{x}} \bfD_{\mathrm{x}}\bfv\vert^2+\mu_{\mathrm{x}}(\bfv)&&\quad \text{ a.e.\ in }Q_T\,,\label{lem:pres.2}\\
			\vert\nabla_{\mathrm{x}}  \bfF^*(\cdot,\cdot,\smash{\widehat{\bfS}})\vert^2+\mu^*_{\mathrm{x}}(\smash{\widehat{\bfS}})&\sim (\delta^{p(\cdot,\cdot)-1}+\vert \smash{\widehat{\bfS}}\vert)^{p'(\cdot,\cdot)-2}\vert \nabla_{\mathrm{x}} \smash{\widehat{\bfS}}\vert^2+\mu^*_{\mathrm{x}}(\smash{\widehat{\bfS}})&&\quad \text{ a.e.\ in }Q_T\,,\label{lem:pres.3}\\
   \vert\partial_{\mathrm{t}} \bfF(\cdot,\cdot,\bfD_{\mathrm{x}}\bfv)\vert+(1+\vert \bfD_{\mathrm{x}}\bfv\vert^{p(\cdot,\cdot)s}) &\sim \vert\partial_{\mathrm{t}}  \bfF^*(\cdot,\cdot,\smash{\widehat{\bfS}})\vert
			+(1+\vert\smash{\widehat{\bfS}}\vert^{p'(\cdot,\cdot)s})&&\quad \text{ a.e.\ in }Q_T\,,\label{lem:pres.4}\\
			\vert \partial_{\mathrm{t}}\bfF(\cdot,\cdot,\bfD_{\mathrm{x}}\bfv)\vert^2+\mu_{\mathrm{t}}(\bfv)&\sim (\delta+\vert \bfD_{\mathrm{x}}\bfv\vert)^{p(\cdot,\cdot)-2}\vert \partial_{\mathrm{t}}\bfD_{\mathrm{x}}\bfv\vert^2+\mu_{\mathrm{t}}(\bfv)&&\quad \text{ a.e.\ in }Q_T\,,\label{lem:pres.5}\\
			\vert\partial_{\mathrm{t}}\bfF^*(\cdot,\cdot,\smash{\widehat{\bfS}})\vert^2+\mu^*_{\mathrm{t}}(\smash{\widehat{\bfS}})&\sim (\delta^{p(\cdot,\cdot)-1}+\vert \smash{\widehat{\bfS}}\vert)^{p'(\cdot,\cdot)-2}\vert \partial_{\mathrm{t}} \smash{\widehat{\bfS}}\vert^2+\mu^*_{\mathrm{t}}(\smash{\widehat{\bfS}})&&\quad \text{ a.e.\ in }Q_T\,,\label{lem:pres.6}
		\end{alignat}
		for some $s>1$, so that, by Lemma \ref{lem:improved_integrability}, we have that $(1+\vert\smash{\widehat{\bfS}}\vert^{p'(\cdot,\cdot)s})\hspace{-0.1em}\sim\hspace{-0.1em} (1+\vert \bfD_{\mathrm{x}}\bfv\vert^{p(\cdot,\cdot)s})\hspace{-0.1em}\in\hspace{-0.1em} L^1(Q_T)$,~and~where\vspace{-1mm}
		\begin{align*}
			\mu_{\mathrm{x}}(\bfv)&\coloneqq \vert \ln(\delta+\vert \bfD_{\mathrm{x}}\bfv\vert)\vert^2(\delta+\vert \bfD_{\mathrm{x}}\bfv\vert)^{p(\cdot,\cdot)-2}\vert \bfD_{\mathrm{x}}\bfv\vert^2  \vert \nabla_{\mathrm{x}} p\vert^2 
			\in L^1(Q_T)\,,\\
			\mu^*_{\mathrm{x}}(\smash{\widehat{\bfS}})&\coloneqq \vert \ln(\delta^{p(\cdot,\cdot)-1}+\vert \smash{\widehat{\bfS}}\vert)\vert^2(\delta^{p(\cdot,\cdot)-1}+\vert \smash{\widehat{\bfS}}\vert)^{p'(\cdot,\cdot)-2}\vert \smash{\widehat{\bfS}}\vert^2 \vert \nabla_{\mathrm{x}} p\vert^2 
			\in L^1(Q_T)\,,\\
            \mu_{\mathrm{t}}(\bfv)&\coloneqq \vert \ln(\delta+\vert \bfD_{\mathrm{x}}\bfv\vert)\vert^2(\delta+\vert \bfD_{\mathrm{x}}\bfv\vert)^{p(\cdot,\cdot)-2}\vert \bfD_{\mathrm{x}}\bfv\vert^2  \vert \partial_{\mathrm{t}} p\vert^2 
			\in L^1(Q_T)\,,\\
			\mu^*_{\mathrm{t}}(\smash{\widehat{\bfS}})&\coloneqq \vert \ln(\delta^{p(\cdot,\cdot)-1}+\vert \smash{\widehat{\bfS}}\vert)\vert^2(\delta^{p(\cdot,\cdot)-1}+\vert \smash{\widehat{\bfS}}\vert)^{p'(\cdot,\cdot)-2}\vert \smash{\widehat{\bfS}}\vert^2 \vert \partial_{\mathrm{t}} p\vert^2 
			\in L^1(Q_T)\,.
		\end{align*}  
        Due to\vspace{-1.5mm} 
        $$(\delta^{p(\cdot,\cdot)-1}+\vert \smash{\widehat{\bfS}}\vert)^{p'(\cdot,\cdot)-2}\sim (\delta+\vert \bfD_{\mathrm{x}}\bfv\vert)^{2-p(\cdot,\cdot)}\quad \text{ a.e.\ in }Q_T\,,$$ from \eqref{lem:pres.1} and \eqref{lem:pres.3},~it~follows~that $(\delta+\vert \bfD_{\mathrm{x}}\bfv\vert)^{p(\cdot,\cdot)-2}\vert \nabla_{\mathrm{x}} \smash{\widehat{\bfS}}\vert^2\in L^1(Q_T)$ with
		\begin{align}
				\vert\nabla_{\mathrm{x}}  \bfF(\cdot,\cdot,\smash{\widehat{\bfS}})\vert^2+	\mu^*_{\mathrm{x}}(\smash{\widehat{\bfS}})\sim (\delta+\vert \bfD_{\mathrm{x}}\bfv\vert)^{2-p(\cdot,\cdot)}\vert \nabla_{\mathrm{x}} \smash{\widehat{\bfS}}\vert^2+	\mu^*_{\mathrm{x}}(\smash{\widehat{\bfS}})\quad \text{ a.e.\ in }Q_T\,.\label{lem:pres.4.2}
		\end{align}
		Then, from \eqref{lem:pres.4.2}, using the $\varepsilon$-Young inequality \eqref{ineq:young} with $\psi\hspace{-0.1em}=\hspace{-0.1em}\smash{\vert\cdot\vert^{\smash{\frac{2}{p'(\cdot,\cdot)}}}}$ (as $p^-\hspace{-0.1em}\ge\hspace{-0.1em} 2$, \textit{i.e.}, $p'(\cdot,\cdot)\hspace{-0.1em}\leq\hspace{-0.1em} 2$~in~$Q_T$) and $\varepsilon=1$, we obtain $\nabla_{\mathrm{x}}\smash{\widehat{\bfS}}\in (L^{p'(\cdot,\cdot)}(Q_T))^{d\times d}$ with\vspace{-1mm}
		\begin{align*}
			\left.	\begin{aligned}
				\vert\nabla_{\mathrm{x}} \smash{\widehat{\bfS}}\vert^{p'(\cdot,\cdot)}&=\vert\nabla_{\mathrm{x}} \smash{\widehat{\bfS}}\vert^{p'(\cdot,\cdot)}(\delta+\vert \bfD_{\mathrm{x}}\bfv\vert)^{p'(\cdot,\cdot)\frac{2-p(\cdot,\cdot)}{2}}(\delta+\vert \bfD_{\mathrm{x}}\bfv\vert)^{-p'(\cdot,\cdot)\frac{2-p(\cdot,\cdot)}{2}}
				\\&\lesssim  (\delta+\vert \bfD_{\mathrm{x}}\bfv\vert)^{2-p(\cdot,\cdot)}\vert \nabla_{\mathrm{x}}\smash{\widehat{\bfS}}\vert^2+(\delta+\vert \bfD_{\mathrm{x}}\bfv\vert)^{p(\cdot,\cdot)}
			\end{aligned}\right\}\quad\text{ a.e.\ in }Q_T\,.
		\end{align*}
        From \eqref{lem:pres.5}, using $p^-\hspace{-0.1em}\ge\hspace{-0.1em} 2$,  $\delta\hspace{-0.1em}>\hspace{-0.1em}0$, and   Poincar\'e's and Korn's  inequality, we obtain $\partial_{\mathrm{t}}\bfv\hspace{-0.1em}\in\hspace{-0.1em} L^2(I;(W^{1,2}(\Omega))^d)$.
		Eventually, using Problem (\hyperlink{Q}{Q}) with $\bff\in (L^{p'(\cdot,\cdot)}(Q_T))^d$, we conclude that  $\nabla_{\mathrm{x}} q \in
		(L^{p'(\cdot,\cdot)}(Q_T))^d$ with\vspace{-1mm}
		\begin{align}\label{lem:pres.5.2}
				\vert \nabla_{\mathrm{x}} q\vert \leq \vert \partial_{\mathrm{t}}\bfv\vert+\vert \nabla_{\mathrm{x}} \smash{\widehat{\bfS}}\vert+ \vert\bff\vert\quad\text{ a.e. in }Q_T\,.
		\end{align}

		\textit{ad (ii).} Multiplying \eqref{lem:pres.5.2} with $\smash{(\delta+\vert  \bfD_{\mathrm{x}}\bfv\vert)^{2-p(\cdot,\cdot)}}$ and using that $p^-\ge 2$ and that $\delta>0$,~we~find~that
		\begin{align}\label{lem:pres.6.2} 
				(\delta+\vert  \bfD_{\mathrm{x}}\bfv\vert)^{2-p(\cdot,\cdot)}\vert \nabla_{\mathrm{x}} q\vert^2 &\lesssim \delta^{2-p(\cdot,\cdot)}\vert \partial_{\mathrm{t}}\bfv\vert^2 +(\delta+\vert  \bfD_{\mathrm{x}}\bfv\vert)^{2-p(\cdot,\cdot)}(\vert \nabla_{\mathrm{x}} \smash{\widehat{\bfS}}\vert^2+\vert\bff\vert^2)
			\quad\text{ a.e. in }Q_T \,.
		\end{align}
		Then, using in \eqref{lem:pres.6.2} that $(\delta+\vert \bfD_{\mathrm{x}}\bfv\vert)^{2-p(\cdot,\cdot)}\vert\bff\vert^2\hspace{-0.1em}\in\hspace{-0.1em} L^1(Q_T)$ and \eqref{lem:pres.4} together with \eqref{lem:pres.1},~we~\mbox{conclude}~that $	(\delta+\vert  \bfD_{\mathrm{x}}\bfv\vert)^{2-p(\cdot,\cdot)}\vert \nabla_{\mathrm{x}} q\vert^2\in L^1(Q_T)$.
	\end{proof}
	
	\section{The discrete unsteady $p(\cdot,\cdot)$-Stokes equations}\label{sec:discrete_p-navier-stokes}\vspace{-1mm}
	
	\hspace*{5mm}In this section, we introduce a fully-discrete FE approximation of the unsteady $p(\cdot,\cdot)$-Stokes~\mbox{equations} \eqref{eq:ptxStokes}, employing a backward Euler step in time and conforming, discretely inf-sup stable FEs in space. Moreover, we collect some relevant assumptions and auxiliary results.\enlargethispage{10mm}\vspace{-1mm}

    \subsection{Space discretization}\vspace{-1mm}
	
	\subsubsection{Triangulations}\vspace{-1mm}
	
	\hspace*{5mm}Throughout the entire paper, let us denote by $\{\mathcal{T}_h\}_{h>0}$ a family of shape-regular conforming triangulations of $\Omega\subseteq \mathbb{R}^d$, $d\in\{2,3\}$, which consists of $d$-dimensional~\mbox{simplices}~(\textit{cf}.\ \cite{EG21}), where 
	$h>0$ denotes the \textit{maximal~mesh-size}, \textit{i.e.}, $h=\max_{K\in \mathcal{T}_h}{h_K}$, where $h_K\coloneqq  \textup{diam}(K)$~for~all~${K\in \mathcal{T}_h}$.~More~precisely, for~every~$K\in \mathcal{T}_h$,
	denoting the supremum of diameters of inscribed balls~in~$K\in\mathcal{T}_h$ by $\rho_K>0$, we assume that there exists a constant $\omega_0>0$ independent of $h>0$, such that $\max_{K\in \mathcal{T}_h}{\{{h_K}{\rho_K^{-1}}\}}\le\omega_0$. We call the smallest such constant  the \textit{chunkiness} of $\{\mathcal{T}_h\}_{h>0}$. Moreover, for every  $K\in \mathcal{T}_h$, we   define the corresponding \textit{element patch} by $\omega_K\coloneqq \bigcup\{K'\in\mathcal{T}_h\mid K'\cap K\neq \emptyset\}$.

    For
	$n \hspace{-0.15em}\in\hspace{-0.15em} \mathbb N\cup\{0\}$ and $h\hspace{-0.15em}>\hspace{-0.15em}0$, let us denote by $\mathbb{P}^n(\mathcal{T}_h)$ the family of (possibly discontinuous)~\mbox{scalar-valued} functions that are polynomials of degree at most $n$ on each  $K\in \mathcal{T}_h$,~and~let~${\mathbb{P}^n_c(\mathcal{T}_h)\coloneqq \mathbb{P}^n(\mathcal{T}_h)\cap C^0(\overline{\Omega})}$. Then, for  $n\hspace{-0.15em} \in\hspace{-0.15em} \mathbb N\cup \{0\}$, the \textit{spatial (local) $L^2$-projection operator} ${\Pi_h^{n,\mathrm{x}}\colon \hspace{-0.15em}L^1(\Omega)\hspace{-0.15em}\to \hspace{-0.15em}\mathbb{P}^n(\mathcal{T}_h)}$,~for~\mbox{every}~${\eta\hspace{-0.15em}\in \hspace{-0.15em}L^1(\Omega)}$, is defined by\vspace{-1mm}
    \begin{align*}
        (\Pi_h^{n,\mathrm{x}}\eta,\eta_h)_{\Omega}=(\eta,\eta_h)_{\Omega}\quad\text{ for all }\eta_h\in  \mathbb{P}^n(\mathcal{T}_h)\,.
    \end{align*}
	
	\subsubsection{Discretely inf-sup stable FE spaces}\vspace{-1mm}
	
	\hspace{5mm}For given $k\in\mathbb{N}$ and $\ell \in \mathbb N\cup\{0\}$, we denote by
	\begin{align}
		\begin{aligned}
			V_h&\subseteq {(\mathbb{P}^k_c(\mathcal{T}_h))^d}\,, &&\,\Vo_h\coloneqq V_h\cap (W^{1,1}_0(\Omega))^d\,,\\[-0.5mm]
			Q_h&\subseteq \mathbb{P}^{\ell}(\mathcal{T}_h)\,, &&\Qo_h \coloneqq Q_h\cap L^1_0(\Omega)\,,
		\end{aligned}
	\end{align}
	FE spaces satisfying the following two assumptions on the existence of suitable FE~projection~operators:\vspace{-1mm}

 \begin{assumption}[projection operator $\Pi_h^Q$]
		\label{ass:PiQ}
		We assume that $\setR
		\subseteq Q_h$ and there exists~a~linear~projection operator
		$\Pi_h^Q\colon L^1(\Omega) \to Q_h$, \textit{i.e.}, $\Pi_h^Q \eta_h= \eta_h$ for all $\eta_h\in Q_h$, that is \textup{locally $L^1$-stable}: for every $\eta\in L^1(\Omega)$ and $K\in \mathcal{T}_h$, there holds\vspace{-1mm}
		\begin{align}
			\label{eq:PiQstab}
			\langle \vert\Pi_h^Q \eta\vert\rangle_K  \lesssim  \langle\vert \eta\vert\rangle_{\omega_K}\,.
		\end{align}
	\end{assumption}
 
	\begin{assumption}[projection operator $\Pi_h^V$]\label{ass:PiV}
		We assume that $\mathbb{P}^1_c(\mathcal{T}_h) \subseteq V_h$ and there exists a linear projection operator $\Pi_h^V\colon \hspace{-0.1em} (W^{1,1}(\Omega))^d\hspace{-0.1em} \to\hspace{-0.1em} V_h$, \textit{i.e.}, $\Pi_h^V\boldsymbol{\varphi}_h\hspace{-0.1em}=\hspace{-0.1em}\boldsymbol{\varphi}_h$ for all $\boldsymbol{\varphi}_h\hspace{-0.1em}\in\hspace{-0.1em} V_h$,~with~the~following~properties:
		\begin{itemize}[noitemsep,topsep=2pt,leftmargin=!,labelwidth=\widthof{(iii)}]
			\item[(i)] \textup{Preservation of divergence in the sense of $Q_h^*$:} For every $\boldsymbol{\varphi} \in (W^{1,1}(\Omega))^d$ and  $\eta_h \in Q_h$, there holds\vspace{-0.5mm}
			\begin{align}
				\label{eq:div_preserving}
				(\mathrm{div}_{\mathrm{x}}  \boldsymbol{\varphi},\eta_h)_{\Omega} &= (\mathrm{div}_{\mathrm{x}} \Pi_h^V
				\boldsymbol{\varphi},\eta_h)_{\Omega} \,;\\[-5.5mm]\notag
			\end{align}
			\item[(ii)] \textup{Preservation of zero boundary values:} $\Pi_h^V((W^{1,1}_0(\Omega))^d) \subseteq \Vo_h$;
			\item[(iii)] \textup{Local $W^{1,1}$-stability:} For every $\boldsymbol{\varphi} \in (W^{1,1}(\Omega))^d$ and $K\in \mathcal{T}_h$, there holds\vspace{-0.5mm}
			\begin{align}
				\label{eq:Pidivcont}
				\langle\vert\Pi_h^V\boldsymbol{\varphi}\vert\rangle_K &\lesssim \langle
				\vert\boldsymbol{\varphi}\vert\rangle_{\omega_K} + h_K\, \langle  \vert\nabla_{\mathrm{x}} \boldsymbol{\varphi}\vert\rangle_{\omega_K} \,.
			\end{align}
		\end{itemize}
	\end{assumption}	

    For a detailed list of mixed FE spaces $\{V_h\}_{h>0}$ and $\{Q_h\}_{h>0}$ with projectors $\{\Pi_h^V\}_{h>0}$~and~$\{\Pi_h^Q\}_{h>0}$ 
    satisfying  Assumption~\ref{ass:PiQ}~and~Assumption~\ref{ass:PiV}, 
    we refer the reader to  
    the textbook \cite{BBF13}.
	 
	\subsection{Time discretization}\vspace{-1mm}
	
	\hspace{5mm}In this section, we describe the time discretization and 
    relevant
	definitions.~Let~$X$~be~a~\mbox{Banach}~space, $M\in\mathbb{N}$,  $\tau\coloneqq \frac{T}{M}$, $t_m\coloneqq \tau\,m$, $I_m\coloneqq \left(t_{m-1},t_m\right]$, $m=1,\ldots,M$,~$\mathcal{I}_\tau \coloneqq \{I_m\}_{m=1,\ldots,M}$, and $\mathcal{I}_\tau^0 \coloneqq \mathcal{I}_\tau\cup\{I_0\}$, where $I_0\coloneqq (t_{-1},t_0]\coloneqq (-\tau,0]$. Denote \textit{the spaces of (in-time) piece-wise~constant~\mbox{$X$-valued}~functions}~by 
	\begin{align*}
		\mathbb{P}^0(\mathcal{I}_\tau;X)&\coloneqq \big\{f\colon I\to X\mid f(s)
		=f(t)\text{ in }X\text{ for all }t,s\in I_m\,,\;m=1,\ldots,M\big\}\,,\\
        \mathbb{P}^0(\mathcal{I}_\tau^0;X)&\coloneqq \big\{f\colon I\to X\mid f(s)
		=f(t)\text{ in }X\text{ for all }t,s\in I_m\,,\;m=0,\ldots,M\big\}\,.
	\end{align*}
	For every $f^\tau\in 	\mathbb{P}^0(\mathcal{I}_\tau^0;X)$, 
		 the \textit{backward difference quotient}  $\mathrm{d}_\tau f^\tau\in \mathbb{P}^0(\mathcal{I}_\tau;X)$~is~defined~by\enlargethispage{10mm}
		\begin{align*}
		\mathrm{d}_\tau f^\tau|_{I_m}\coloneqq \tfrac{1}{\tau}(f^\tau(t_m)-f^\tau(t_{m-1}))\quad\text{ in }X \quad\text{ for all }m=1,\ldots,M\,.
		\end{align*}
		
When $X$ is a Hilbert space equipped with inner product $(\cdot,\cdot)_X$, for every $f^\tau\in \mathbb{P}^0(\mathcal{I}_\tau^0;X)$, we have the following version of \textit{discrete integration-by-parts formula}: for every $m,n = 0,\ldots,M$~with~$n\ge m$, there holds
	\begin{align}
		\int_{t_m}^{t_n}{( \mathrm{d}_\tau f^\tau(t),
			f^\tau(t))_X\,\mathrm{d}t}
		= \frac{1}{2}\|f^\tau(t_n)\|_X^2-\frac{1}{2}\|f^\tau(t_m)\|_X^2+\sum_{i=m}^n{\tfrac{\tau^2}{2}\|\mathrm{d}_\tau f^\tau(t_i)\|_X^2}\,,\label{eq:4.2}
	\end{align}
	which follows from the identity
	$(\mathrm{d}_\tau f^\tau(t_i),f^\tau(t_i))_X
	\hspace{-0.1em} =\hspace{-0.1em} \frac{1}{2}\mathrm{d}_\tau\|f^\tau(t_i)\|_X^2+\frac{\tau}{2}\|\mathrm{d}_\tau f^\tau(t_i)\|_X^2$~for~all~${i\hspace{-0.1em} = \hspace{-0.1em} 0,\ldots,M}$.\enlargethispage{5mm} 
	
	The \textit{temporal (local) $L^2$-projection operator} $\Pi^{0,\mathrm{t}}_{\tau}\colon L^1(I;X)\to \mathbb{P}^0(\mathcal{I}_\tau;X)$, for every $f\in L^1(I;X)$, is defined by
	\begin{align}\label{def:Pit} \Pi^{0,\mathrm{t}}_{\tau}f|_{I_m}\coloneqq \langle f\rangle_{I_m}\quad\textup{ in }X\quad \text{ for all }m=1,\ldots,M\,,
	\end{align}
    where $\langle f\rangle_{I_m}\coloneqq  \fint_{I_m}{f(t)\,\mathrm{d}t}\in X$ for all $m=1,\ldots,M$ is a Bochner integral,
	while the \textit{temporal~(lowest order) nodal 
    interpolation operator} $\mathrm{I}^{0,\mathrm{t}}_{\tau}\colon C^0(\overline{I};X)\to \mathbb{P}^0(\mathcal{I}_\tau;X)$, for every $f\in C^0(\overline{I};X)$~is~defined~by
	\begin{align}\label{def:Pit}
		\mathrm{I}^{0,\mathrm{t}}_{\tau}f|_{I_m}\coloneqq  f(t_m)\quad\textup{ in }X\quad \text{ for all }m=1,\ldots,M\,.
	\end{align}

	\subsection{Discrete weak formulations}\vspace*{-0.5mm}
	
	\hspace{5mm}By analogy with \cite{ptx_lap,berselli2024convergence}, we employ a simple one-point quadrature rule to discretize the power-law index and, in this way, all related non-linear mappings. More precisely, if at least $p\in C^0(\overline{Q_T})$~with~${p^->1}$, then  the element-wise constant approximations of power-law index $p_h^\tau\in \mathbb{P}^0(\mathcal{I}_\tau;\mathbb{P}^0(\mathcal{T}_h))$, the generalized $N$-function $\varphi_h^{\tau}\colon\Omega\times\mathbb{R}_{\ge 0}\to \mathbb{R}_{\ge 0}$, and the non-linear mappings $\bfS_h^\tau,\bfF_h^\tau,(\bfF_h^\tau)^*\colon Q_T\times\mathbb{R}^{d\times d}\to  \mathbb{R}^{d\times d}_{\textup{sym}}$, for every $m=1,\ldots,M$, $K\in  \mathcal{T}_h$,  $\bfA\in \mathbb{R}^{d\times d}$, and a.e.\ $(t,x)^\top\in I_m\times K$,  are defined  by
	\begin{align}
		\hspace{-2mm}\begin{aligned}
			p_h^\tau(t,x)&\coloneqq p(t_m,\xi_K)\,,&&\varphi_h^{\tau}(t,x,\vert \bfA\vert )\coloneqq \varphi(t_m,\xi_K,\vert \bfA\vert)\,,\\
			\bfS_h^\tau(t,x,\bfA)&\coloneqq \bfS(t_m,\xi_K,\bfA)\,,
			&&\hspace{1.75mm}\bfF_h^{\tau}(t,x,\bfA)\coloneqq \bfF(t_m,\xi_K,\bfA)\,,
			\quad(\bfF_h^\tau)^*(t,x,\bfA)\coloneqq \bfF^*(t_m,\xi_K,\bfA)\,,
		\end{aligned}\label{def:A_h}
	\end{align}
	where $(t_m,\xi_K)^\top\in I_m\times K$ is an arbitrary quadrature point.
	
\begin{remark}\label{rem:uniform}
		Note that, since all the implicit constants in the equivalences introduced in Section~\ref{sec:basic} depend only on $p^-,p^+>1$ and $\delta\ge 0$ and  $p^-\leq p_h^\tau\leq p^+$ a.e.\ in $Q_T$ for all $\tau,h>0$,     
		the same~equivalences hold for the approximations \eqref{def:A_h} with implicit constants depending~only~on~${p^-,p^+\in (1,+\infty)}$~and~${\delta\ge 0}$. 
	\end{remark}
	
	By means of discretizations defined in \eqref{def:A_h}, we have now everything at our disposal~to~introduce~the discrete counterparts to Problem (\hyperlink{Q}{Q}) and Problem (\hyperlink{P}{P}), respectively:
	
	\textit{Problem (Q$_h^\tau$).}\hypertarget{Qh}{} For given $\bfg\hspace{-0.1em}\in\hspace{-0.1em} (L^{(p^-)'}(Q_T))^d$, $\bfG\hspace{-0.1em}\in\hspace{-0.1em} (L^{p'(\cdot,\cdot)}(Q_T))^{d\times d}_{\textup{sym}}$,  and $\bfv_0\hspace{-0.1em}\in\hspace{-0.1em} H$, find $(\bfv_h^\tau,q_h^\tau)^\top\in \mathbb{P}^0(\mathcal{I}_\tau^0;\Vo_h)\times \mathbb{P}^0(\mathcal{I}_\tau;\Qo_h)$ with $\bfv_h^\tau(0)=\bfv_h^0$ in $\Vo_h$, such that for every $(\boldsymbol{\varphi}_h^{\tau},\eta_h^{\tau})^\top\in \mathbb{P}^0(\mathcal{I}_\tau;\Vo_h)\times \mathbb{P}^0(\mathcal{I}_\tau;Q_h)$, there holds
	\begin{align*}
		(\mathrm{d}_\tau\bfv_h^\tau,\boldsymbol{\varphi}_h^{\tau} )_{Q_T}+
		(\bfS_h^\tau(\cdot,\cdot,\bfD_{\mathrm{x}}\bfv_h^\tau),\bfD_{\mathrm{x}}\boldsymbol{\varphi}_h^{\tau})_{Q_T}-(q_h^\tau,\mathrm{div}_{\mathrm{x}} \boldsymbol{\varphi}_h^{\tau})_{Q_T}&=(\bfg,\boldsymbol{\varphi}_h^{\tau})_{Q_T}+(\bfG,\bfD_{\mathrm{x}}\boldsymbol{\varphi}_h^{\tau})_{Q_T}\,,\\
		(\mathrm{div}_{\mathrm{x}} \bfv_h^\tau,\eta_h^{\tau})_{Q_T}&=0\,,
	\end{align*}
    where $\bfv_0^h\hspace{-0.175em}\coloneqq\hspace{-0.175em} \Pi_h^{V,L^2}\hspace{-0.175em}\bfv\hspace{-0.175em}\in\hspace{-0.175em} \Vo_h$ and $\Pi_h^{V,L^2}\hspace{-0.175em}\colon \hspace{-0.175em}(L^2(\Omega))^d\hspace{-0.175em}\to \hspace{-0.175em}\Vo_h$ is the spatial (global) $L^2$-projection~of~$(L^2(\Omega))^d$~onto~$\Vo_h$.\\
	We can reformulate the Problem (\hyperlink{Qh}{Q$_h^\tau$}) in a hydro-mechanical sense, \textit{i.e.}, hiding the discrete pressure:
	
	\textit{Problem (P$_h^\tau$).}\hypertarget{Ph}{} For given $\bfg\hspace{-0.1em}\in\hspace{-0.1em} (L^{(p^-)'}(Q_T))^d$, $\bfG\hspace{-0.1em}\in\hspace{-0.1em} (L^{p'(\cdot,\cdot)}(Q_T))^{d\times d}_{\textup{sym}}$,  and $\bfv_0\hspace{-0.1em}\in\hspace{-0.1em} H$, find ${\bfv_h^\tau\hspace{-0.1em}\in\hspace{-0.1em} \mathbb{P}^0(\mathcal{I}_\tau^0;\Vo_{h,0})}$\linebreak \hphantom{.}\hspace{5mm} with $\bfv_h^\tau(0)=\bfv_h^0$~in $\Vo_0$, such that for every $\boldsymbol{\varphi}_h^{\tau}\in \mathbb{P}^0(\mathcal{I}_\tau;\Vo_{h,0})$, there holds
	\begin{align*}
		(\mathrm{d}_\tau\bfv_h^\tau,\boldsymbol{\varphi}_h^{\tau} )_{Q_T}+	(\bfS_h^\tau(\cdot,\cdot,\bfD_{\mathrm{x}}\bfv_h^\tau),\bfD_{\mathrm{x}}\boldsymbol{\varphi}_h^{\tau})_{Q_T}=(\bfg,\boldsymbol{\varphi}_h^{\tau})_{Q_T}+(\bfG,\bfD_{\mathrm{x}}\boldsymbol{\varphi}_h^{\tau})_{Q_T}\,,
	\end{align*}
 where ${\Vo_{h,0}}\coloneqq \{\boldsymbol{\varphi}_h\in \Vo_h\mid  (\mathrm{div}_{\mathrm{x}} \boldsymbol{\varphi}_h,\eta_h)_{\Omega}=0\text{ for all }\eta_h\in Q_h\}$.

	The well-posedness of Problem (\hyperlink{Qh}{Q$_h^\tau$}) and Problem (\hyperlink{Ph}{P$_h^\tau$}) is proved as in the continuous case~in~two~steps:
    using pseudo-monotone operator theory, the well-posedness of Problem~(\hyperlink{Ph}{P$_h^\tau$}) is shown (\textit{cf}.~\mbox{\cite[Thm.~8.3]{berselli2024convergence}}); then, given the well-posedness of Problem~(\hyperlink{Ph}{P$_h^\tau$}), 
	the well-posedness of (\hyperlink{Qh}{Q$_h^\tau$}) follows using discrete inf-sup stability of the FE couple $(V_h,Q_h)$ (\textit{cf}.\ \cite{BBD15,berselli2023error}).\newpage
	
	\subsection{Stability estimates for the FE projection operators}
	
	\hspace*{5mm}First, we have the following stability estimate for the projection operator $\Pi_h^V$ (\textit{cf}.\  Assumption \ref{ass:PiV}), in  modular form with respect to  shifted $N$-functions, where the shift is constant.\enlargethispage{10mm}

\begin{lemma}\label{lem:stab_PiV}
Suppose that $p\in \mathcal{P}^{\log}(Q_T)$  and let $c_0>0$. Then, for every $n\in \mathbb{N}$,  $J\in \mathcal{I}_\tau$, $K\in \mathcal{T}_h$, $\boldsymbol{\varphi}\in (L^{p(\cdot,\cdot)}(J\times\omega_K))^d$ with $\boldsymbol{\varphi}\in \smash{(W^{1,p(t,\cdot)}(\omega_K))^d}$ for a.e.\ $t\in J$ and $\nabla_{\mathrm{x}} \boldsymbol{\varphi}\in (L^{p(\cdot,\cdot)}(J\times\omega_K))^{d\times d}$, and $a\ge  0$ with  $a+\langle\vert \nabla_{\mathrm{x}}\boldsymbol{\varphi}(t)\vert\rangle_{\omega_K}\leq  c_0\,\vert K\vert ^{-n}$ for a.e.\ $t\in J$,
there holds
\begin{align}\label{lem:stab_PiV.1}
\rho_{\varphi_a,J\times K}(\nabla_{\mathrm{x}} \Pi_h^V \boldsymbol{\varphi})\lesssim h_K^n+\rho_{\varphi_a,J\times\omega_K}(\nabla_{\mathrm{x}} \boldsymbol{\varphi})\,,
\end{align}
where the implicit constant in $\lesssim$ depends on  $k$, $n$, $p^+$,~$[p]_{\log,Q_T}$,  $\omega_0$, and $c_0$. Moreover, for every $n\in \mathbb{N}$, $\boldsymbol{\varphi}\in (L^{p(\cdot,\cdot)}(Q_T))^d$ with $\boldsymbol{\varphi}\in \smash{(W^{1,p(t,\cdot)}(\Omega))^d}$ for a.e.\ $t\in I$ and $\nabla_{\mathrm{x}} \boldsymbol{\varphi}\in (L^{p(\cdot,\cdot)}(Q_T))^{d\times d}$, and $a\ge 0$ with $a+\|\nabla_{\mathrm{x}} \boldsymbol{\varphi}(t)\|_{1,\Omega}\leq c_0$ for a.e.\ $t\in I$, there holds
\begin{align}\label{lem:stab_PiV.2}
\rho_{\varphi_a,Q_T}(\nabla_{\mathrm{x}} \Pi_h^V \boldsymbol{\varphi})\lesssim h^n+\rho_{\varphi_a,Q_T}(\nabla_{\mathrm{x}} \boldsymbol{\varphi})\,,
\end{align}
where the implicit constant in $\lesssim$ depends  on  $k$, $n$, $p^+$,~$[p]_{\log,Q_T}$,  $\omega_0$, and $c_0$.
\end{lemma}
	
\begin{proof}
\textit{ad \eqref{lem:stab_PiV.1}.} According to  \cite[Thm.\ 4.2 a)]{BBD15} or \cite[Lem.\ 4.25(4.26)]{berselli2023error}, for a.e.\ $t\in J$, we have that
\begin{align}\label{lem:stab_PiV.3}
\rho_{\varphi_a(t,\cdot,\cdot),K}(\nabla_{\mathrm{x}} \Pi_h^V \boldsymbol{\varphi}(t))\lesssim h_K^n+\rho_{\varphi_a(t,\cdot,\cdot),\omega_K}(\nabla_{\mathrm{x}} \boldsymbol{\varphi}(t))\,,
\end{align}
where the hidden constant in $\lesssim$ depends on  $k$, $n$, $p^+$,~$[p]_{\log,Q_T}$,  $\omega_0$, and $c_0$. Thus, integration~of~\eqref{lem:stab_PiV.3} with respect to $t\in J$ yields the claimed local stability estimate \eqref{lem:stab_PiV.1}.
		
\textit{ad \eqref{lem:stab_PiV.2}.} The estimate follows analogously to \eqref{lem:stab_PiV.1}, but, in this case, by using \cite[Lem.\ 4.25(4.27)]{berselli2023error} instead of \cite[Thm.~4.2~a)]{BBD15} or \cite[Lem.\ 4.25(4.26)]{berselli2023error}, respectively.
\end{proof}

Next, we prove the following stability estimate for the projection operator $\Pi_\tau^{0,\mathrm{t}}\Pi_h^Q$ (\textit{cf}.\  \eqref{def:Pit} and Assumption \ref{ass:PiQ}) in modular form with respect to conjugate shifted $N$-functions,~where~the~shift~is~\mbox{constant}.

\begin{lemma}\label{lem:stab_PiQ}
Suppose that $p\in \mathcal{P}^{\log}(Q_T)$ and let $c_0>0$. Then, for every $n\in \mathbb{N}$, $J\in \mathcal{I}_\tau$, $K\in \mathcal{T}_h$, $\eta\in L^{p'(\cdot,\cdot)}(J\times\omega_K)$, and $a\ge 0$ with $a+\langle\vert \eta\vert\rangle_{J\times \omega_K}\leq c_0\,\vert J\times K\vert ^{-n}$, there holds
\begin{align}\label{lem:stab_PiQ.1}	\rho_{(\varphi_a)^*,J\times K}(\Pi_\tau^{0,\mathrm{t}}\Pi_h^Q \eta)\lesssim (\tau+h_K)^n+\rho_{(\varphi_a)^*,J\times \omega_K}(\eta)\,,
\end{align}
where the implicit constant in $\lesssim$ depends on  $\ell$, $n$, $p^+$, $p^-$, $[p]_{\log,Q_T}$,  $\omega_0$, and $c_0$. Moreover, for every $n\in \mathbb{N}$, $\eta\in L^{p'(\cdot,\cdot)}(Q_T)$, and $a\ge 0$ with $a+\|\eta\|_{1,Q_T}\leq c_0$, there holds
\begin{align}\label{lem:stab_PiQ.2}
\rho_{(\varphi_a)^*,Q_T}(\Pi_\tau^{0,\mathrm{t}}\Pi_h^Q \eta)\lesssim(\tau+ h)^n+\rho_{(\varphi_a)^*,Q_T}(\eta)\,,
\end{align}
where the implicit constant in $\lesssim$ depends on  $\ell$, $n$, $p^+$, $p^-$, $[p]_{\log,Q_T}$,  $\omega_0$, and $c_0$.
\end{lemma}

\begin{remark}
    Note that via adding the temporal (local) $L^2$-projection operator $\Pi_\tau^{0,\mathrm{t}}$ in Lemma \ref{lem:stab_PiQ}, we can avoid the conditions $a+\langle\vert \eta(t)\vert\rangle_{\omega_K}\leq c_0\,\vert K\vert ^{-n}$  for a.e.\ $t\in J$ in \eqref{lem:stab_PiQ.1} and $a+\|\eta(t)\|_{\Omega}\leq c_0$ for a.e.\ $t\in I$ in \eqref{lem:stab_PiQ.2}, which are too restrictive due to the missing additional temporal regularity properties of the kinematic pressure.
    Instead, we merely need to impose the conditions $a+\langle\vert \eta\vert\rangle_{J\times\omega_K}\leq c_0\,\vert J\times K\vert ^{-n}$ in \eqref{lem:stab_PiQ.1} and $a+\|\eta\|_{Q_T}\leq c_0$ in \eqref{lem:stab_PiQ.2}, which better fit the regularity assumptions~on~the~kinematic~pressure.
\end{remark}
 
\begin{proof}[Proof (of Lemma \ref{lem:stab_PiQ}).]
\textit{ad \eqref{lem:stab_PiQ.1}.} By using a (local) inverse inequality (\textit{cf}.\ \cite[Lem.\ 12.1]{EG21}), the (local) $L^1$-stability of the projection operators $\Pi_h^{0,\mathrm{t}}$ 
and $\Pi_h^Q$ (\textit{cf}.\ Assumption \ref{ass:PiQ}\eqref{eq:PiQstab}), and the key estimate (\textit{cf}.\ Lemma \ref{lem:key-estimate}), for a.e.\ $(t,x)^\top \in J\times K$, we observe that
\begin{align}\label{lem:stab_PiQ.3}
\begin{aligned} 
(\varphi_a)^*(t,x,\|\Pi_\tau^{0,\mathrm{t}}\Pi_h^Q \eta\|_{\infty,J\times K})&\leq c\,(\varphi_a)^*(t,x,\langle \Pi_\tau^{0,\mathrm{t}}\Pi_h^Q \eta\rangle_{J\times K})\\&\leq c\,(\varphi_a)^*(t,x,\langle \eta\rangle_{J\times \omega_K})\\&\leq c\,(\tau+h_K)^n+c\,\langle (\varphi_a)^*(\cdot,\cdot,\eta)\rangle_{J\times \omega_K}\,.
\end{aligned}
\end{align}
Integration of \eqref{lem:stab_PiQ.3} with respect to $\smash{(t,x)^\top}\in J\times K$ yields the claimed local stability estimate \eqref{lem:stab_PiQ.1}.
		
\textit{ad \eqref{lem:stab_PiQ.2}.} If $a+\|\eta\|_{1,Q_T}\hspace{-0.15em}\leq\hspace{-0.15em} c_0$, from $\tau+h_K\hspace{-0.15em}\leq\hspace{-0.15em} 1$, it follows that $\vert J\times K\vert ^{-1}\hspace{-0.15em}\leq\hspace{-0.15em} \vert J\times K\vert^{-(d+1+n)}$.~Thus,~there exists a constant $c>0$, depending only on  $\omega_0$ and $c_0$, such that for every $J\in \mathcal{I}_\tau$ and $K\in \mathcal{T}_h$,~we~have~that
\begin{align}\label{lem:stab_PiQ.4}
a+\langle\vert \eta\vert\rangle_{J\times \omega_K}\leq c\,\vert J\times K\vert^{-1}\leq c\,\vert J\times K\vert ^{-(d+1+n)}\,.
\end{align}
Due to \eqref{lem:stab_PiQ.4}, resorting to \eqref{lem:stab_PiQ.1}, we find a constant $c>0$, depending only on  $\ell$, $n$, $p^+$, $p^-$, $[p]_{\log,Q_T}$,~$\omega_0$, and $c_0$, such that for every  $J\in \mathcal{I}_\tau$ and $K\in \mathcal{T}_h$, it follows that
\begin{align}\label{lem:stab_PiQ.5}
\rho_{(\varphi_a)^*,J\times K}(\Pi_\tau^{0,\mathrm{t}}\Pi_h^Q \eta)\leq c\,\big((\tau +h_K)^n\,\vert J\times K\vert +\rho_{(\varphi_a)^*,J\times \omega_K}(\eta)\big)\,.
\end{align}
Summation of \eqref{lem:stab_PiQ.5} with respect to $J\hspace{-0.1em}\in\hspace{-0.1em} \mathcal{I}_\tau$ and $K\hspace{-0.1em}\in\hspace{-0.1em} \mathcal{T}_h$ yields 
the claimed global stability~estimate~\eqref{lem:stab_PiQ.2}.
\end{proof}\vspace{-6mm}

	\section{Fractional interpolation error estimates for the FE projection operators}\label{sec:fractional_interpolation_estimates}\vspace{-1.5mm}
	
	\hspace{5mm}In this section, we  derive fractional interpolation error estimates for the projection~operators~$\smash{\Pi_h^V}$ (\textit{cf}.\ Assumption~\ref{ass:PiV})~and $\Pi_\tau^{0,\mathrm{t}}\Pi_h^Q$ (\textit{cf}.\ Assumption \ref{ass:PiQ}), which play a crucial role for error~decays~rates according to fractional regularity assumptions on the velocity vector field and the~kinematic~pressure.\enlargethispage{10mm}
	
For the fractional regularity of the velocity vector field represented in Bochner--Nikolski\u{\i} spaces (\textit{cf}.\ \eqref{eq:natural_regularity.velocity}), we have the following interpolation estimate for the projection operator $\Pi_h^V$  (\textit{cf}.\ Assumption~\ref{ass:PiV}) with respect to the natural distance.
	
\begin{lemma}\label{lem:PiV_F}
Suppose that $p\in C^{0,\alpha_{\mathrm{t}},\alpha_{\mathrm{x}}}(\overline{Q_T})$, $\alpha_{\mathrm{t}},\alpha_{\mathrm{x}}\in (0,1]$, with  $p^->1$ and let $\boldsymbol{\varphi}\in  \mathbfcal{V}$ be such that $\bfF(\cdot,\cdot,\bfD_{\mathrm{x}}\boldsymbol{\varphi})\in \smash{L^2(I;(N^{\beta_{\mathrm{x}},2}(\Omega))^{d\times d})}$,  $\beta_{\mathrm{x}}\in (0,1]$. Then,  there exists a constant $s>1$ with $s\searrow1$ as $h_K\searrow0$, such that
if $\bfD_{\mathrm{x}}\boldsymbol{\varphi}\in (L^{p(\cdot,\cdot)s}(Q_T))^{d\times d}$ and $\textup{ess\,sup}_{t\in I}{\big\{\|\bfD_{\mathrm{x}}\boldsymbol{\varphi}(t)\|_{p(t,\cdot),\Omega}\big\}}<\infty$, then
for every $J\in \mathcal{I}_\tau$ and $K\in \mathcal{T}_h$, there holds
\begin{align}
\begin{aligned} 
\|\bfF(\cdot,\cdot,\bfD_{\mathrm{x}}\boldsymbol{\varphi})-\bfF(\cdot,\cdot,\bfD_{\mathrm{x}}\Pi_h^V\boldsymbol{\varphi})\|_{2,J\times K}^2&\lesssim h_K ^{2\alpha_{\mathrm{x}}}\,\|1+\vert \bfD_{\mathrm{x}}\boldsymbol{\varphi}\vert^{p(\cdot,\cdot)s}\|_{1,J\times \omega_K }\\
&\quad+ h_K ^{2\beta_{\mathrm{x}}}\,[ \bfF(\cdot,\cdot ,\bfD_{\mathrm{x}}\boldsymbol{\varphi})]_{L^2(J;N^{\beta_{\mathrm{x}},2}(\omega_K ))}^2\,,
\end{aligned}\label{lem:PiV_F.1}
\end{align}
where \hspace{-0.1mm}the \hspace{-0.1mm}implicit \hspace{-0.1mm}constant \hspace{-0.1mm}in \hspace{-0.1mm}$\lesssim$ \hspace{-0.1mm}depends \hspace{-0.1mm}on \hspace{-0.1mm}$k$, \hspace{-0.1mm}$p^-$, \hspace{-0.1mm}$p^+$, \hspace{-0.1mm}$[p]_{\alpha_{\mathrm{t}},\alpha_{\mathrm{x}},Q_T}$, \hspace{-0.1mm}$s$, \hspace{-0.1mm}$\omega_0$, \hspace{-0.1mm}and \hspace{-0.1mm}$\textup{ess\,sup}_{t\in I}{\big\{\|\bfD_{\mathrm{x}}\boldsymbol{\varphi}(t)\|_{p(t,\cdot),\Omega}\big\}}$.
In particular, it follows that
\begin{align}
\begin{aligned} 
\|\bfF(\cdot,\cdot,\bfD_{\mathrm{x}}\boldsymbol{\varphi})-\bfF(\cdot,\cdot,\bfD_{\mathrm{x}}\Pi_h^V\boldsymbol{\varphi})\|_{2,Q_T}^2&\lesssim h^{2\alpha_{\mathrm{x}}}\,\big(1+\rho_{p(\cdot,\cdot)s,Q_T}(\bfD_{\mathrm{x}}\boldsymbol{\varphi})\big)\\&\quad+h^{2\beta_{\mathrm{x}}}\,[ \bfF(\cdot,\cdot,\bfD_{\mathrm{x}}\boldsymbol{\varphi})]_{L^2(I;N^{\beta_{\mathrm{x}},2}(\Omega))}^2\,.
\end{aligned}\label{lem:PiV_F.2}
\end{align}
\end{lemma}
	
	\begin{proof}
		\emph{ad \eqref{lem:PiV_F.1}.} According to \cite[Lem.\  5.6(5.7)]{berselli2023error}, 	there exists a constant $s>1$ with $s\searrow1$ as $h_K\searrow0$, such that for every $J\in \mathcal{I}_\tau$,  $K\in \mathcal{T}_h$, and a.e.\ $t\in J$, there holds 
\begin{align}
\begin{aligned} 
\|\bfF(t,\cdot,\bfD_{\mathrm{x}}\boldsymbol{\varphi})-\bfF(t,\cdot,\bfD_{\mathrm{x}}\Pi_h^V\boldsymbol{\varphi})\|_{2,K}^2&\lesssim 
h_K ^{2\alpha _x}\,\|1+\vert \bfD_{\mathrm{x}}\boldsymbol{\varphi}(t)\vert^{p(t,\cdot)s}\|_{1,\omega_K }   
\\&\quad+ h_K ^{2\beta_{\mathrm{x}}}\,[ \bfF(t,\cdot, \bfD_{\mathrm{x}}\boldsymbol{\varphi}(t))]_{N^{\beta_{\mathrm{x}},2}(\omega_K )}^2\,,
\end{aligned}\label{lem:PiV_F.3}
\end{align}
		where $\lesssim$ depends on  $k$, $p^-$, $p^+$, $[p]_{\alpha_{\mathrm{t}},\alpha_{\mathrm{x}},Q_T}$, $s$, $\omega_0$, and $\|\bfD_{\mathrm{x}}\boldsymbol{\varphi}(t)\|_{p(t,\cdot),\Omega}$. As a result, integrating \eqref{lem:PiV_F.3} with respect to $t\in J$ and using that $\textup{ess\,sup}_{t\in I}{\big\{\|\bfD_{\mathrm{x}}\boldsymbol{\varphi}(t)\|_{p(t,\cdot),\Omega}\big\}}<\infty$,~we~conclude~that~the~claimed local interpolation error estimate \eqref{lem:PiV_F.1} applies.
		
\emph{ad \eqref{lem:PiV_F.2}.} The global interpolation error estimate \eqref{lem:PiV_F.2} is obtained analogously to the proof of the local interpolation error estimate
\eqref{lem:PiV_F.1}, but, in this case, by using \cite[Lem.\  5.6(5.8)]{berselli2023error}.
\end{proof}
	
By  using Lemma \ref{lem:A-Ah}, we can derive an analogue of Lemma \ref{lem:PiV_F} for $\bfF_h^{\tau}\colon \hspace*{-0.1em} Q_T\times \mathbb{R}^{d\times d}\hspace*{-0.1em}\to\hspace*{-0.1em} \smash{\mathbb{R}^{d\times d}_{\textup{sym}}}$,~${\tau, h\hspace*{-0.1em}\in\hspace*{-0.1em} (0,1]}$, instead of  $\bfF\colon Q_T\times \mathbb{R}^{d\times d}\to \smash{\mathbb{R}^{d\times d}_{\textup{sym}}}$.
	
\begin{lemma}\label{lem:PiV_Fh}
Suppose that $p\in C^{0,\alpha_{\mathrm{t}},\alpha_{\mathrm{x}}}(\overline{Q_T})$, $\alpha_{\mathrm{t}},\alpha_{\mathrm{x}}\in (0,1]$,  with  $p^->1$ and let $\boldsymbol{\varphi}\in  \mathbfcal{V}$~be~such~that
$\bfF(\cdot,\cdot,\bfD_{\mathrm{x}}\boldsymbol{\varphi})\in \smash{L^2(I;(N^{\beta_{\mathrm{x}},2}(\Omega))^{d\times d})}$,  $\beta_{\mathrm{x}}\in   (0,1]$. Then,  there exists a constant $s>1$ with ${s\searrow 1}$ as  $\tau+h_K\searrow 0$, such that
if $\bfD_{\mathrm{x}}\boldsymbol{\varphi}\in (L^{p(\cdot,\cdot)s}(Q_T))^{d\times d}$ and $\textup{ess\,sup}_{t\in I}{\big\{\|\bfD_{\mathrm{x}}\boldsymbol{\varphi}(t)\|_{p(t,\cdot),\Omega}\big\}}<\infty$, then
for every $ J\in \mathcal{I}_\tau$ and $K\in  \mathcal{T}_h$, there holds
\begin{align}
\begin{aligned} 
\|\bfF_h^{\tau}(\cdot,\cdot,\bfD_{\mathrm{x}}\boldsymbol{\varphi})-\bfF_h^{\tau}(\cdot,\cdot,\bfD_{\mathrm{x}}\Pi_h^V\boldsymbol{\varphi})\|_{2,J\times K}^2&\lesssim 		(\tau^{2\alpha_{\mathrm{t}}}+h_K ^{2\alpha_{\mathrm{x}}})\,\|1+\vert \bfD_{\mathrm{x}}\boldsymbol{\varphi}\vert^{p(\cdot,\cdot)s}\|_{1,J\times\omega_K }\\&\quad+h_K ^{2\beta_{\mathrm{x}}}\,[ \bfF(\cdot,\cdot,\bfD_{\mathrm{x}}\boldsymbol{\varphi})]_{L^2(J;N^{\beta_{\mathrm{x}},2}(\omega_K ))}^2\,,
\end{aligned}\label{lem:PiV_Fh.1}
\end{align}
where \hspace{-0.1mm}the \hspace{-0.1mm}implicit \hspace{-0.1mm}constant \hspace{-0.1mm}in \hspace{-0.1mm}$\lesssim$ \hspace{-0.1mm}depends \hspace{-0.1mm}on \hspace{-0.1mm}$k$, \hspace{-0.1mm}$p^-$, \hspace{-0.1mm}$p^+$, \hspace{-0.1mm}$[p]_{\alpha_{\mathrm{t}},\alpha_{\mathrm{x}},Q_T}$, \hspace{-0.1mm}$s$, \hspace{-0.1mm}$\omega_0$, \hspace{-0.1mm}and \hspace{-0.1mm}$\textup{ess\,sup}_{t\in I}{\big\{\|\bfD_{\mathrm{x}}\boldsymbol{\varphi}(t)\|_{p(t,\cdot),\Omega}\big\}}$.
In particular, it follows that
\begin{align}
\begin{aligned} 
\|\bfF_h^{\tau}(\cdot,\cdot,\bfD_{\mathrm{x}}\boldsymbol{\varphi})-\bfF_h^{\tau}(\cdot,\cdot,\bfD_{\mathrm{x}}\Pi_h^V\boldsymbol{\varphi})\|_{2,Q_T}^2&\lesssim (\tau^{2\alpha_{\mathrm{t}}}+h^{2\alpha_{\mathrm{x}}})\,\big(1+\rho_{p(\cdot,\cdot)s,Q_T}(\bfD_{\mathrm{x}}\boldsymbol{\varphi})\big)\\&\quad+
			h^{2\beta_{\mathrm{x}}}\,[ \bfF(\cdot,\cdot,\bfD_{\mathrm{x}}\boldsymbol{\varphi})]_{L^2(I;N^{\beta_{\mathrm{x}},2}(\Omega))}^2\,.
		\end{aligned}\label{lem:PiV_Fh.2}
		\end{align}
	\end{lemma}\newpage
	
\begin{proof}
\emph{ad \eqref{lem:PiV_Fh.1}.} Using Lemma \ref{lem:A-Ah}\eqref{eq:Fh-F}, \eqref{lem:PiV_F.3}, again, Lemma \ref{lem:A-Ah}\eqref{eq:Fh-F}, and \textcolor{black}{\eqref{lem:stab_PiV.3}} (with $a=\delta=0$ and $ps$ instead of $p$),
for every $J\in \mathcal{I}_\tau$, $K\in \mathcal{T}_h$, and a.e.\ $t\in J$, we find that
\begin{align}
\begin{aligned} 
\|\bfF_h^{\tau}(t,\cdot,\bfD_{\mathrm{x}}\boldsymbol{\varphi}(t))-\bfF_h^{\tau}(t,\cdot,\bfD_{\mathrm{x}}\Pi_h^V\boldsymbol{\varphi}(t))\|_{2,K}^2&\lesssim \|\bfF_h^{\tau}(t,\cdot,\bfD_{\mathrm{x}}\boldsymbol{\varphi}(t))-\bfF(t,\cdot,\bfD_{\mathrm{x}}\boldsymbol{\varphi}(t))\|_{2,K}^2\\
&\quad+\|\bfF(t,\cdot,\bfD_{\mathrm{x}}\boldsymbol{\varphi}(t))-\bfF(t,\cdot,\bfD_{\mathrm{x}}\Pi_h^V\boldsymbol{\varphi}(t))\|_{2,K}^2\\
&\quad+\|\bfF(t,\cdot,\bfD_{\mathrm{x}}\Pi_h^V\boldsymbol{\varphi}(t))-\bfF_h^{\tau}(t,\cdot,\bfD_{\mathrm{x}}\Pi_h^V\boldsymbol{\varphi}(t))\|_{2,K}^2\\
&\lesssim (\tau^{2\alpha_{\mathrm{t}}}+h_K ^{2\alpha_{\mathrm{x}}})\,\|1+\vert \bfD_{\mathrm{x}}\boldsymbol{\varphi}(t)\vert^{p(t,\cdot)s}\|_{1,\omega_K }\\
&\quad+ h_K ^{2\beta_{\mathrm{x}}}\,[ \bfF(t,\cdot,\bfD_{\mathrm{x}}\boldsymbol{\varphi}(t))]_{N^{\beta_{\mathrm{x}},2}(\omega_K )}^2\\
&\quad +(\tau^{2\alpha_{\mathrm{t}}}+h_K ^{2\alpha_{\mathrm{x}}})\,\|1+\vert \bfD_{\mathrm{x}}\Pi_h^V\boldsymbol{\varphi}(t)\vert^{p(t,\cdot)s}\|_{1,K}
\\
&\lesssim (\tau^{2\alpha_{\mathrm{t}}}+h_K ^{2\alpha_{\mathrm{x}}})\,\|1+\vert \bfD_{\mathrm{x}}\boldsymbol{\varphi}(t)\vert^{p(t,\cdot)s}\|_{1,\omega_K }\\
&\quad+ h_K ^{2\beta_{\mathrm{x}}}\,[ \bfF(t,\cdot,\bfD_{\mathrm{x}}\boldsymbol{\varphi}(t))]_{N^{\beta_{\mathrm{x}},2}(\omega_K )}^2 \,,
\end{aligned}\label{lem:PiV_Fh.3}
\end{align}
where the implicit constant in $\lesssim$ depends  on  $k$, $p^-$, $p^+$, $[p]_{\alpha_{\mathrm{t}},\alpha_{\mathrm{x}},Q_T}$, $s$, $\omega_0$, and $\|\bfD_{\mathrm{x}}\boldsymbol{\varphi}(t)\|_{p(t,\cdot),\Omega}$.~As~a~\mbox{result}, integration of  \eqref{lem:PiV_Fh.3} with respect to $t\in J$ and using that $\textup{ess\,sup}_{t\in I}{\big\{\|\bfD_{\mathrm{x}}\boldsymbol{\varphi}(t)\|_{p(t,\cdot),\Omega}\big\}}<\infty$, we conclude that the claimed local interpolation error estimate \eqref{lem:PiV_Fh.1} applies.
		
\emph{ad \eqref{lem:PiV_Fh.2}.} The global interpolation error estimate \eqref{lem:PiV_Fh.2} is obtained analogously to the proof of the local interpolation error estimate \eqref{lem:PiV_Fh.1}.
\end{proof}

A similar fractional interpolation error estimate can be proved for the nodal interpolation operator~$\mathrm{I}_\tau^{0,\mathrm{t}}$ (\textit{cf}.\ \eqref{def:Pit}).\enlargethispage{11mm}
	
\begin{lemma}\label{lem:Inod_F}
	Suppose that $p\in C^{0,\alpha_{\mathrm{t}},\alpha_{\mathrm{x}}}(\overline{Q_T})$, $\alpha_{\mathrm{t}},\alpha_{\mathrm{x}}\in (0,1]$,  with $p^->1$ and let  $\boldsymbol{\varphi}\in\mathbfcal{V}$ be such that $\bfF(\cdot,\cdot,\bfD_{\mathrm{x}}\boldsymbol{\varphi}) \in N^{\beta_{\mathrm{t}},2}(I;(L^2(\Omega))^{d\times d})\cap L^2(I;(N^{\beta_{\mathrm{x}},2}(\Omega))^{d\times d})$, $\beta_{\mathrm{t}}\in (\frac{1}{2},1]$, $\beta_{\mathrm{x}}\in (0,1]$. Then,  there exists a~constant~$s\hspace{-0.1mm}>\hspace{-0.1mm}1$ with $s\hspace{-0.1mm}\searrow\hspace{-0.1mm}1$ as $\tau\hspace{-0.1mm}\searrow\hspace{-0.1mm}0$, such that 
    for every $J\in \mathcal{I}_\tau$ and $K\in \mathcal{T}_h$, there holds
\begin{align}
\begin{aligned} 
\|\bfF(\cdot,\cdot,\bfD_{\mathrm{x}}\boldsymbol{\varphi})-\bfF(\cdot,\cdot,\bfD_{\mathrm{x}}\mathrm{I}_\tau^{0,\mathrm{t}}\boldsymbol{\varphi})\|_{2,J\times K}^2
&\lesssim \tau^{2\alpha_{\mathrm{t}}+1}\,{\sup}_{t\in J}{\big\{\|1+\vert \bfD_{\mathrm{x}}\boldsymbol{\varphi}(t)\vert^{p(t,\cdot)s}\|_{1,K}\big\}}\\
&\quad+ \tau^{2\beta_{\mathrm{t}}}\,[ \bfF(\cdot,\cdot,\bfD_{\mathrm{x}}\boldsymbol{\varphi})]_{N^{\beta_{\mathrm{t}},2}(J;L^2(K))}^2\,,
\end{aligned}\label{lem:Inod_F.1}
\end{align}
where \hspace{-0.1mm}the \hspace{-0.1mm}implicit \hspace{-0.1mm}constant \hspace{-0.1mm}in \hspace{-0.1mm}$\lesssim$ \hspace{-0.1mm}depends \hspace{-0.1mm}on \hspace{-0.1mm}$k$, \hspace{-0.1mm}$p^-$, \hspace{-0.1mm}$p^+$, \hspace{-0.1mm}$[p]_{\alpha_{\mathrm{t}},\alpha_{\mathrm{x}},Q_T}$, \hspace{-0.1mm}$s$, \hspace{-0.1mm}$\omega_0$, \hspace{-0.1mm}and \hspace{-0.1mm}$\sup_{t\in I}{\big\{\|\bfD_{\mathrm{x}}\boldsymbol{\varphi}(t)\|_{p(t,\cdot),\Omega}\big\}}$.
In particular, it follows that
\begin{align}
\begin{aligned} 
\|\bfF(\cdot,\cdot,\bfD_{\mathrm{x}}\boldsymbol{\varphi})-\bfF(\cdot,\cdot,\bfD_{\mathrm{x}}\mathrm{I}_\tau^{0,\mathrm{t}}\boldsymbol{\varphi})\|_{2,Q_T}^2
&\lesssim \tau^{2\alpha_{\mathrm{t}}}\,{\sup}_{t\in I}{\big\{\|1+\vert \bfD_{\mathrm{x}}\boldsymbol{\varphi}(t)\vert^{p(t,\cdot)s}\|_{1,\Omega}\big\}}\\
&\quad+ \tau^{2\beta_{\mathrm{t}}}
\,[ \bfF(\cdot,\cdot,\bfD_{\mathrm{x}}\boldsymbol{\varphi})]_{N^{\beta_{\mathrm{t}},2}(I;L^2(\Omega))}^2\,.
\end{aligned}\label{lem:Inod_F.2}
\end{align}
\end{lemma}
    
\begin{proof}
\textit{ad \eqref{lem:Inod_F.1}.} First, note that, by Lemma \ref{lem:improved_integrability}, we have that 
$\bfF(\cdot,\cdot,\bfD_{\mathrm{x}}\boldsymbol{\varphi})\in C^0(\overline{I};(L^{2s}(\Omega))^{d\times d})$ for some $s>1$, which implies that $(t\mapsto\vert\bfD_{\mathrm{x}}\bfv(t)\vert^{p(t,\cdot)s})\hspace{-0.175em}\in\hspace{-0.175em} C^0(\overline{I})$, so that applying  $\mathrm{I}_{\tau}^{0,\mathrm{t}}$ (\textit{cf}.\ \eqref{def:Pit})~is~\mbox{well-defined}. Then, 
introducing the non-linear mapping $\bfF^{\tau}\colon Q_T\times \mathbb{R}^{d\times d}\to \mathbb{R}^{d\times d}_{\textup{sym}}$, 
for~every~$t\in I_m$,~$m=1,\ldots,M$, $x\in \Omega$, and $\bfA\in \mathbb{R}^{d\times d}$, 
defined by
\begin{align*}
\bfF^{\tau}(t,x,\bfA)\coloneqq \bfF(t_m,x,\bfA)\,,
\end{align*}
from Lemma \ref{lem:A-Ah}\eqref{eq:Fh-F}, we deduce that
\begin{align*}
\|\bfF(\cdot,\cdot,\bfD_{\mathrm{x}}\boldsymbol{\varphi})-\bfF(\cdot,\cdot,\bfD_{\mathrm{x}}\mathrm{I}_\tau^{0,\mathrm{t}}\boldsymbol{\varphi})\|_{2,J\times K}^2
&\lesssim \|\bfF(\cdot,\cdot,\bfD_{\mathrm{x}}\boldsymbol{\varphi})-\bfF^\tau(\cdot,\cdot,\bfD_{\mathrm{x}}\mathrm{I}_\tau^{0,\mathrm{t}}\boldsymbol{\varphi})\|_{2,J\times K}^2\\
&\quad+\|\bfF(\cdot,\cdot,\bfD_{\mathrm{x}}\mathrm{I}_\tau^{0,\mathrm{t}}\boldsymbol{\varphi})-\bfF^\tau(\cdot,\cdot,\bfD_{\mathrm{x}}\mathrm{I}_\tau^{0,\mathrm{t}}\boldsymbol{\varphi})\|_{2,J\times K}^2\\
&\lesssim\tau^{2\beta_{\mathrm{t}}}\,[ \bfF(\cdot,\cdot,\bfD_{\mathrm{x}}\boldsymbol{\varphi})]^2_{N^{\beta_{\mathrm{t}},2}(J;L^2(K))}\\
&\quad+\tau^{2\alpha_{\mathrm{t}}}\,\|1+\vert\bfD \mathrm{I}_\tau^{0,\mathrm{t}}\boldsymbol{\varphi}\vert^{(\mathrm{I}_\tau^{0,\mathrm{t}}p)(\cdot,\cdot)s}\|_{1,J\times K }\\
&\lesssim\tau^{2\beta_{\mathrm{t}}}\,[ \bfF(\cdot,\cdot,\bfD_{\mathrm{x}}\boldsymbol{\varphi})]_{N^{\beta_{\mathrm{t}},2}(J;L^2(K))}^2\\
&\quad+\tau^{2\alpha_{\mathrm{t}}+1}\,{\sup}_{t\in J}{\big\{\|1+\vert\bfD_{\mathrm{x}}\boldsymbol{\varphi}(t)\vert^{p(t,\cdot)s}\|_{1, K } \big\}}\,,
\end{align*}
where \hspace{-0.15mm}the \hspace{-0.15mm}implicit \hspace{-0.15mm}constant \hspace{-0.15mm}in \hspace{-0.15mm}$\lesssim$ \hspace{-0.15mm}depends \hspace{-0.15mm}only \hspace{-0.15mm}on \hspace{-0.15mm}$k$, \hspace{-0.15mm}$p^-$, \hspace{-0.15mm}$p^+$, \hspace{-0.15mm}$[p]_{\alpha_{\mathrm{t}},\alpha_{\mathrm{x}},Q_T}$, \hspace{-0.15mm}$s$, \hspace{-0.15mm}$\omega_0$, \hspace{-0.15mm}and \hspace{-0.15mm}${\sup}_{t\in I}{\big\{\|\bfD_{\mathrm{x}}\boldsymbol{\varphi}(t)\|_{p(t,\cdot),\Omega}\big\}}$.

\textit{ad \eqref{lem:Inod_F.2}.} The global interpolation error estimate \eqref{lem:Inod_F.2} is obtained analogously to the proof of the local interpolation error estimate \eqref{lem:Inod_F.1}.
\end{proof}

\newpage
By using Lemma \ref{lem:A-Ah}, we can derive an analogue of Lemma \ref{lem:Inod_F} for $\bfF_h^{\tau}\colon\hspace*{-0.1em} Q_T\times \mathbb{R}^{d\times d}\hspace*{-0.1em}\to\hspace*{-0.1em} \smash{\mathbb{R}^{d\times d}_{\textup{sym}}}$,~${\tau,h\in (0,1]}$, instead of  $\bfF\colon Q_T\times \mathbb{R}^{d\times d}\to \smash{\mathbb{R}^{d\times d}_{\textup{sym}}}$.
	
	\begin{lemma}\label{lem:Inod_Fh}
			Suppose that $p\in C^{0,\alpha_{\mathrm{t}},\alpha_{\mathrm{x}}}(\overline{Q_T})$, $\alpha_{\mathrm{t}},\alpha_{\mathrm{x}}\in (0,1]$, with $p^->1$  and let $\boldsymbol{\varphi}\in  \mathbfcal{V}$  be such that
		$\bfF(\cdot,\cdot,\bfD_{\mathrm{x}}\boldsymbol{\varphi}) \in N^{\beta_{\mathrm{t}},2}(I;(L^2(\Omega))^{d\times d})\cap L^2(I;(N^{\beta_{\mathrm{x}},2}(\Omega))^{d\times d})$, $\beta_{\mathrm{t}}\in (\frac{1}{2},1]$, $\beta_{\mathrm{x}}\in (0,1]$. Then,  there exists a constant $s>1$ with $s\searrow1$ as $\tau+h_K\hspace{-0.1em}\searrow\hspace{-0.1em}0$~such~that  
        for every $J\in \mathcal{I}_\tau$~and~${K\in \mathcal{T}_h}$,~there~holds
\begin{align}
\begin{aligned} 
\|\bfF_h^{\tau}(\cdot,\cdot,\bfD_{\mathrm{x}}\boldsymbol{\varphi})-\bfF_h^{\tau}(\cdot,\cdot,\bfD_{\mathrm{x}}\mathrm{I}_\tau^{0,\mathrm{t}}\boldsymbol{\varphi})\|_{2,J\times K}^2&\lesssim (\tau^{2\alpha_{\mathrm{t}}}+h_K ^{2\alpha_{\mathrm{x}}})\,\tau\,{\sup}_{t\in J}{\big\{\|1+\vert \bfD_{\mathrm{x}}\boldsymbol{\varphi}(t)\vert^{p(t,\cdot)s}\|_{1,K}\big\}}\\
&\quad+ \tau^{2\beta_{\mathrm{t}}}\,[ \bfF(\cdot,\cdot,\bfD_{\mathrm{x}}\boldsymbol{\varphi})]_{N^{\beta_{\mathrm{t}},2}(J;L^2(K))}^2\,,
\end{aligned}\label{lem:Inod_Fh.1}
\end{align}
where \hspace{-0.1mm}the \hspace{-0.1mm}implicit \hspace{-0.1mm}constant \hspace{-0.1mm}in \hspace{-0.1mm}$\lesssim$ \hspace{-0.1mm}depends \hspace{-0.1mm}on \hspace{-0.1mm}$k$, \hspace{-0.1mm}$p^-$, \hspace{-0.1mm}$p^+$, \hspace{-0.1mm}$[p]_{\alpha_{\mathrm{t}},\alpha_{\mathrm{x}},Q_T}$, \hspace{-0.1mm}$s$, \hspace{-0.1mm}$\omega_0$, \hspace{-0.1mm}and \hspace{-0.1mm}${\sup}_{t\in I}{\big\{\|\bfD_{\mathrm{x}}\boldsymbol{\varphi}(t)\|_{p(t,\cdot),\Omega}\big\}}$.
In particular, it follows that
\begin{align}
\begin{aligned} 
\|\bfF_h^{\tau}(\cdot,\cdot,\bfD_{\mathrm{x}}\boldsymbol{\varphi})-\bfF_h^{\tau}(\cdot,\cdot,\bfD_{\mathrm{x}}\mathrm{I}_\tau^{0,\mathrm{t}}\boldsymbol{\varphi})\|_{2,Q_T}^2&\lesssim (\tau^{2\alpha_{\mathrm{t}}}+h^{2\alpha_{\mathrm{x}}})\,{\sup}_{t\in I}{\big\{\|1+\vert \bfD_{\mathrm{x}}\boldsymbol{\varphi}(t)\vert^{p(t,\cdot)s}\|_{1,\Omega}\big\}} \\
&\quad+ \tau^{2\beta_{\mathrm{t}}}\,[ \bfF(\cdot,\cdot,\bfD_{\mathrm{x}}\boldsymbol{\varphi})]_{N^{\beta_{\mathrm{t}},2}(I;L^2(\Omega))}^2\,.
\end{aligned}\label{lem:Inod_Fh.2}
\end{align}
\end{lemma}
	
\begin{proof}
\emph{ad \eqref{lem:Inod_Fh.1}.} Using twice Lemma \ref{lem:A-Ah}\eqref{eq:Fh-F},  
for every $J\in \mathcal{I}_\tau$ and $K\in \mathcal{T}_h$, we find that
\begin{align*}
\begin{aligned} 
\|\bfF_h^{\tau}(\cdot,\cdot,\bfD_{\mathrm{x}}\boldsymbol{\varphi})-\bfF_h^{\tau}(\cdot,\cdot,\bfD_{\mathrm{x}}\mathrm{I}_\tau^{0,\mathrm{t}}\boldsymbol{\varphi})\|_{2,J\times K}^2
&\lesssim 	\|\bfF_h^{\tau}(\cdot,\cdot,\bfD_{\mathrm{x}}\boldsymbol{\varphi})-\bfF(\cdot,\cdot,\bfD_{\mathrm{x}}\boldsymbol{\varphi})\|_{2,J\times K}^2\\
&\quad+\|\bfF(\cdot,\cdot,\bfD_{\mathrm{x}}\boldsymbol{\varphi})-\bfF^\tau(\cdot,\cdot,\bfD_{\mathrm{x}}\mathrm{I}_\tau^{0,\mathrm{t}}\boldsymbol{\varphi})\|_{2,J\times K}^2\\
&\quad+\|\bfF^\tau(\cdot,\cdot,\bfD_{\mathrm{x}}\mathrm{I}_\tau^{0,\mathrm{t}}\boldsymbol{\varphi})-\bfF_h^{\tau}(\cdot,\cdot,\bfD_{\mathrm{x}}\mathrm{I}_\tau^{0,\mathrm{t}}\boldsymbol{\varphi})\|_{2,J\times K}^2
\\
&\lesssim (\tau^{2\alpha_{\mathrm{t}}}+h_K ^{2\alpha_{\mathrm{x}}})\,\|1+\vert \bfD_{\mathrm{x}}\boldsymbol{\varphi}\vert^{p(\cdot,\cdot)s}\|_{1,J\times K }\\
&\quad+\tau^{2\beta_{\mathrm{t}}}\,[ \bfF(\cdot,\cdot,\bfD_{\mathrm{x}}\boldsymbol{\varphi})]_{N^{\beta_{\mathrm{t}},2}(J;L^2(K))}^2\\
&\quad +(\tau^{2\alpha_{\mathrm{t}}}+h_K ^{\alpha_{\mathrm{x}}})\,\|1+\vert\bfD \mathrm{I}_\tau^{0,\mathrm{t}}\boldsymbol{\varphi}\vert^{(\mathrm{I}_\tau^{0,\mathrm{t}}p)(\cdot,\cdot)s}\|_{1,J\times K} \\
&\lesssim  (\tau^{2\alpha_{\mathrm{t}}}+h^{2\alpha_{\mathrm{x}}})\,\tau\,{\sup}_{t\in J}{\big\{\|1+\vert \bfD_{\mathrm{x}}\boldsymbol{\varphi}(t)\vert^{p(t,\cdot)s}\|_{1,K}\big\}}
\\
&\quad+ \tau^{2\beta_{\mathrm{t}}}\,[ \bfF(\cdot,\cdot,\bfD_{\mathrm{x}}\boldsymbol{\varphi})]_{N^{\beta_{\mathrm{t}},2}(J;L^2(K))}^2\,,
\end{aligned}
\end{align*}
where the implicit constant in $\lesssim$ depends on  $k$, $p^-$, $p^+$, $[p]_{\alpha_{\mathrm{t}},\alpha_{\mathrm{x}},Q_T}$, $s$, $\omega_0$, and ${\sup}_{t\in I}{\big\{\|\bfD_{\mathrm{x}}\boldsymbol{\varphi}(t)\|_{p(t,\cdot),\Omega}\big\}}$. 

\emph{ad \eqref{lem:Inod_Fh.2}.} The global interpolation error estimate \eqref{lem:Inod_Fh.2} is obtained analogously to the proof of the local interpolation error estimate \eqref{lem:Inod_Fh.1}.
\end{proof}

Combining the previous fractional interpolation error estimates for the projection operator $\Pi_h^V$ (\textit{cf}.\ Assumption \ref{ass:PiV}) and the nodal interpolation operator $\mathrm{I}_\tau^{0,\mathrm{t}}$, we arrive at a fractional interpolation error estimates for the projection operator $\Pi_h^V$ applied to the interpolation operator $\mathrm{I}_\tau^{0,\mathrm{t}}$, which we will frequently use in the derivation of \textit{a priori} error estimates in Section \ref{sec:a_priori}.
	
\begin{lemma}\label{lem:PiV_FI}
Suppose that $p\in C^{0,\alpha_{\mathrm{t}},\alpha_{\mathrm{x}}}(\overline{Q_T})$, $\alpha_{\mathrm{t}},\alpha_{\mathrm{x}}\in (0,1]$, with $p^->1$ and let $\boldsymbol{\varphi}\in  \mathbfcal{V}$ be such that $\bfF(\cdot,\cdot,\bfD_{\mathrm{x}}\boldsymbol{\varphi})\in \smash{N^{\beta_{\mathrm{t}},2}(I;(L^2(\Omega))^{d\times d})}\cap\smash{L^2(I;(N^{\beta_{\mathrm{x}},2}(\Omega))^{d\times d})}$, $\beta_{\mathrm{t}}\in (\frac{1}{2},1]$,  $\beta_{\mathrm{x}}\in (0,1]$. Then,  there exists a constant $s>1$ with $s\searrow1$ as $\tau+h_K\searrow0$, such that
for every $J\in \mathcal{I}_\tau$ and $K\in \mathcal{T}_h$, there holds
\begin{align}
\begin{aligned} 
\|\bfF_h^{\tau}(\cdot,\cdot,\bfD_{\mathrm{x}}\mathrm{I}_\tau^{0,\mathrm{t}}\boldsymbol{\varphi})-\bfF_h^{\tau}(\cdot,\cdot,\bfD_{\mathrm{x}}\Pi_h^V\mathrm{I}_\tau^{0,\mathrm{t}}\boldsymbol{\varphi})\|_{2,J\times K}^2&\lesssim (\tau^{2\alpha_{\mathrm{t}}}+h_K ^{2\alpha_{\mathrm{x}}})\,\tau\,{\sup}_{t\in J}{\big\{\|1+\vert \bfD_{\mathrm{x}}\boldsymbol{\varphi}(t)\vert^{p(t,\cdot)s}\|_{1,\omega_K }\big\}}\\
&\quad+ \tau ^{2\beta_{\mathrm{t}}}\,[ \bfF(\cdot,\cdot,\bfD_{\mathrm{x}}\boldsymbol{\varphi})]_{N^{\beta_{\mathrm{t}},2}(J;L^2(K))}^2\\
&\quad+ h_K ^{2\beta_{\mathrm{x}}}\,[ \bfF(\cdot,\cdot,\bfD_{\mathrm{x}}\boldsymbol{\varphi})]_{L^2(J;N^{\beta_{\mathrm{x}},2}(\omega_K ))}^2\,,
\end{aligned}\hspace{-7.5mm}\label{lem:PiV_FI.1}
\end{align}
where the implicit constant in $\lesssim$ depends on $k$, $p^-$, $p^+$, $[p]_{\alpha_{\mathrm{t}},\alpha_{\mathrm{x}},Q_T}$, $s$, $\omega_0$, and $\sup_{t\in I}{\big\{\|\bfD_{\mathrm{x}}\boldsymbol{\varphi}(t)\|_{p(t,\cdot),\Omega}\big\}}$.
In particular, it follows that
\begin{align}
\begin{aligned} 
\|\bfF_h^{\tau}(\cdot,\cdot,\bfD_{\mathrm{x}}\mathrm{I}_\tau^{0,\mathrm{t}}\boldsymbol{\varphi})-\bfF_h^{\tau}(\cdot,\cdot,\bfD_{\mathrm{x}}\Pi_h^V\mathrm{I}_\tau^{0,\mathrm{t}}\boldsymbol{\varphi})\|_{2,Q_T}^2&\lesssim  	(\tau^{2\alpha_{\mathrm{t}}}+h^{2\alpha_{\mathrm{x}}})\,{\sup}_{t\in I}{\big\{\|1+\vert \bfD_{\mathrm{x}}\boldsymbol{\varphi}(t)\vert^{p(t,\cdot)s}\|_{1,\Omega}\big\}}\\
&\quad+\tau^{2\beta_{\mathrm{t}}}\,[ \bfF(\cdot,\cdot,\bfD_{\mathrm{x}}\boldsymbol{\varphi})]_{N^{\beta_{\mathrm{t}},2}(I;L^2(\Omega))}^2\\
&\quad+h^{2\beta_{\mathrm{x}}}\,[ \bfF(\cdot,\cdot,\bfD_{\mathrm{x}}\boldsymbol{\varphi})]_{L^2(I;N^{\beta_{\mathrm{x}},2}(\Omega))}^2\,.
\end{aligned}\label{lem:PiV_FI.2}
\end{align}
\end{lemma}
\pagebreak
\begin{proof} 
		\textit{ad \eqref{lem:PiV_FI.1}.} Using four times Lemma \ref{lem:A-Ah}\eqref{eq:Fh-F} and once Lemma \ref{lem:stab_PiV}\eqref{lem:stab_PiV.1} (with $a=\delta=0$ and $p|_{J\times \omega_K}$ replaced by $\mathrm{I}_\tau^{0,\mathrm{t}}p|_{J\times \omega_K}\in \mathcal{P}^{\log}(J\times \omega_K)$), we find that\vspace{-0.75mm}
\begin{align*}
	\|\bfF_h^{\tau}(\cdot,\cdot,\bfD_{\mathrm{x}}\mathrm{I}_\tau^{0,\mathrm{t}}\boldsymbol{\varphi})-\bfF_h^{\tau}(\cdot,\cdot,\bfD_{\mathrm{x}}\Pi_h^V\mathrm{I}_\tau^{0,\mathrm{t}}\boldsymbol{\varphi})\|_{2,J\times K}^2&\lesssim \|\bfF_h^{\tau}(\cdot,\cdot,\bfD_{\mathrm{x}}\mathrm{I}_\tau^{0,\mathrm{t}}\boldsymbol{\varphi})-\bfF^\tau(\cdot,\cdot,\bfD_{\mathrm{x}}\mathrm{I}_\tau^{0,\mathrm{t}}\boldsymbol{\varphi})\|_{2,J\times K}^2
	\\&\quad+\|\bfF^{\tau}(\cdot,\cdot,\bfD_{\mathrm{x}}\mathrm{I}_\tau^{0,\mathrm{t}}\boldsymbol{\varphi})-\bfF(\cdot,\cdot,\bfD_{\mathrm{x}}\mathrm{I}_\tau^{0,\mathrm{t}}\boldsymbol{\varphi})\|_{2,J\times K}^2
	\\&\quad+\|\bfF(\cdot,\cdot,\bfD_{\mathrm{x}}\mathrm{I}_\tau^{0,\mathrm{t}}\boldsymbol{\varphi})-\bfF(\cdot,\cdot,\bfD_{\mathrm{x}}\Pi_h^V\mathrm{I}_\tau^{0,\mathrm{t}}\boldsymbol{\varphi})\|_{2,J\times K}^2
	\\&\quad+\|\bfF(\cdot,\cdot,\bfD_{\mathrm{x}}\Pi_h^V\mathrm{I}_\tau^{0,\mathrm{t}}\boldsymbol{\varphi})-\bfF^{\tau}(\cdot,\cdot,\bfD_{\mathrm{x}}\Pi_h^V\mathrm{I}_\tau^{0,\mathrm{t}}\boldsymbol{\varphi})\|_{2,J\times K}^2
	\\&\quad+\|\bfF^\tau(\cdot,\cdot,\bfD_{\mathrm{x}}\Pi_h^V\mathrm{I}_\tau^{0,\mathrm{t}}\boldsymbol{\varphi})-\bfF_h^{\tau}(\cdot,\cdot,\bfD_{\mathrm{x}}\Pi_h^V\mathrm{I}_\tau^{0,\mathrm{t}}\boldsymbol{\varphi})\|_{2,J\times K}^2
    \\&\lesssim (\tau^{2\alpha_{\mathrm{t}}}+h_K ^{2\alpha_{\mathrm{x}}})\,\|1+\vert \bfD_{\mathrm{x}}\mathrm{I}_\tau^{0,\mathrm{t}}\boldsymbol{\varphi}\vert^{(\mathrm{I}_\tau^{0,\mathrm{t}}p)(\cdot,\cdot)s}\|_{1,J\times K} 
    \\&\quad +\|\bfF(\cdot,\cdot,\bfD_{\mathrm{x}}\mathrm{I}_\tau^{0,\mathrm{t}}\boldsymbol{\varphi})-\bfF(\cdot,\cdot,\bfD_{\mathrm{x}}\Pi_h^V\mathrm{I}_\tau^{0,\mathrm{t}}\boldsymbol{\varphi})\|_{2,J\times K}^2
    \\&\quad +(\tau^{2\alpha_{\mathrm{t}}}+h_K ^{2\alpha_{\mathrm{x}}})\,\|1+\vert \bfD_{\mathrm{x}}\Pi_h^V\mathrm{I}_\tau^{0,\mathrm{t}}\boldsymbol{\varphi}\vert^{(\mathrm{I}_\tau^{0,\mathrm{t}}p)(\cdot,\cdot)s}\|_{1,J\times K} 
    \\&\lesssim (\tau^{2\alpha_{\mathrm{t}}}+h_K ^{2\alpha_{\mathrm{x}}})\,\tau\,{\sup}_{t\in J}{\big\{\|1+\vert \bfD_{\mathrm{x}}\boldsymbol{\varphi}(t)\vert^{p(t,\cdot)s}\|_{1,\omega_K }\big\}}
    \\&\quad +\|\bfF(\cdot,\cdot,\bfD_{\mathrm{x}}\mathrm{I}_\tau^{0,\mathrm{t}}\boldsymbol{\varphi})-\bfF(\cdot,\cdot,\bfD_{\mathrm{x}}\Pi_h^V\mathrm{I}_\tau^{0,\mathrm{t}}\boldsymbol{\varphi})\|_{2,J\times K}^2\,.\\[-6mm]
\end{align*} 
Then, using Lemma \ref{lem:Inod_F}\eqref{lem:Inod_F.1} and Lemma \ref{lem:PiV_F}\eqref{lem:PiV_F.1}, we obtain\vspace{-0.75mm}\enlargethispage{12.5mm}
\begin{align*}
	\|\bfF(\cdot,\cdot,\bfD_{\mathrm{x}}\mathrm{I}_\tau^{0,\mathrm{t}}\boldsymbol{\varphi})-\bfF(\cdot,\cdot,\bfD_{\mathrm{x}}\Pi_h^V\mathrm{I}_\tau^{0,\mathrm{t}}\boldsymbol{\varphi})\|_{2,J\times K}^2
	&\lesssim \|\bfF(\cdot,\cdot,\bfD_{\mathrm{x}}\mathrm{I}_\tau^{0,\mathrm{t}}\boldsymbol{\varphi})-\bfF(\cdot,\cdot,\bfD_{\mathrm{x}}\boldsymbol{\varphi})\|_{2,J\times K}^2
	\\&\quad+ \|\bfF(\cdot,\cdot,\bfD_{\mathrm{x}}\boldsymbol{\varphi})-\bfF(\cdot,\cdot,\bfD_{\mathrm{x}}\Pi_h^V\boldsymbol{\varphi})\|_{2,J\times K}^2
	\\&\quad+ \|\bfF(\cdot,\cdot,\bfD_{\mathrm{x}}\Pi_h^V\boldsymbol{\varphi})-\bfF(\cdot,\cdot,\bfD_{\mathrm{x}}\Pi_h^V\mathrm{I}_\tau^{0,\mathrm{t}}\boldsymbol{\varphi})\|_{2,J\times K}^2
	\\&\lesssim (\tau^{2\alpha_{\mathrm{t}}}+h_K ^{2\alpha_{\mathrm{x}}})\,\tau\,{\sup}_{t\in J}{\big\{\|1+\vert \bfD_{\mathrm{x}}\boldsymbol{\varphi}(t)\vert^{p(t,\cdot)s}\|_{1,\omega_K }\big\}}
	\\&\quad+ \tau ^{2\beta_{\mathrm{t}}}\,[ \bfF(\cdot,\cdot,\bfD_{\mathrm{x}}\boldsymbol{\varphi})]_{N^{\beta_{\mathrm{t}},2}(J;L^2(K))}^2
	\\&\quad+ h_K ^{2\beta_{\mathrm{x}}}\,[ \bfF(\cdot,\cdot,\bfD_{\mathrm{x}}\boldsymbol{\varphi})]_{L^2(J;N^{\beta_{\mathrm{x}},2}(\omega_K ))}^2
	\\&\quad+\|\bfF(\cdot,\cdot,\bfD_{\mathrm{x}}\Pi_h^V\boldsymbol{\varphi})-\bfF(\cdot,\cdot,\bfD_{\mathrm{x}}\Pi_h^V\mathrm{I}_\tau^{0,\mathrm{t}}\boldsymbol{\varphi})\|_{2,J\times K}^2\,.\\[-6mm]
\end{align*}
Then, \hspace{-0.1mm}using \hspace{-0.1mm}Lemma \hspace{-0.1mm}\ref{lem:shift-change}\eqref{lem:shift-change.1} \hspace{-0.1mm}together \hspace{-0.1mm}with \hspace{-0.1mm}\eqref{eq:hammera}, \hspace{-0.1mm}\textcolor{black}{Lemma \hspace{-0.1mm}\ref{lem:stab_PiV}\eqref{lem:stab_PiV.1}}, \hspace{-0.1mm}again, \hspace{-0.1mm}Lemma \hspace{-0.1mm}\ref{lem:shift-change}\eqref{lem:shift-change.1}~\hspace{-0.1mm}\mbox{together} with \eqref{eq:hammera}, Lemma \ref{lem:Inod_F}\eqref{lem:Inod_F.1}, Lemma \ref{lem:PiV_F}\eqref{lem:PiV_F.1}, Lemma \ref{lem:poincare_F}\eqref{lem:poincare_F.1}, and the key estimate (\textit{cf}.~Lemma~\ref{lem:key-estimate}), we arrive at\vspace{-0.75mm}
\begin{align*}
	\|\bfF(\cdot,\cdot,\bfD_{\mathrm{x}}\Pi_h^V\boldsymbol{\varphi})-\bfF(\cdot,\cdot,\bfD_{\mathrm{x}}\Pi_h^V\mathrm{I}_\tau^{0,\mathrm{t}}\boldsymbol{\varphi})\|_{2,J\times K}^2
    &\lesssim \rho_{\varphi_{\smash{\vert \langle \bfD_{\mathrm{x}}\Pi_h^V\boldsymbol{\varphi}\rangle_{J\times\omega_K }\vert}},J\times K}(\bfD_{\mathrm{x}}\Pi_h^V\boldsymbol{\varphi}-\bfD_{\mathrm{x}}\Pi_h^V\mathrm{I}_\tau^{0,\mathrm{t}}\boldsymbol{\varphi})
	\\&\quad+\|\bfF(\cdot,\cdot,\bfD_{\mathrm{x}}\Pi_h^V\boldsymbol{\varphi})-\bfF(\cdot,\cdot,\langle \bfD_{\mathrm{x}}\Pi_h^V\boldsymbol{\varphi}\rangle_{J\times\omega_K })\|_{2,J\times K}^2
	\\&\lesssim h_K^n+\rho_{\varphi_{\smash{\vert \langle \bfD_{\mathrm{x}}\Pi_h^V\boldsymbol{\varphi}\rangle_{J\times\omega_K }\vert}},J\times \omega_K }(\bfD_{\mathrm{x}}\boldsymbol{\varphi}-\bfD_{\mathrm{x}}\mathrm{I}_\tau^{0,\mathrm{t}}\boldsymbol{\varphi})
	\\&\quad+\|\bfF(\cdot,\cdot,\bfD_{\mathrm{x}}\Pi_h^V\boldsymbol{\varphi})-\bfF(\cdot,\cdot,\langle \bfD_{\mathrm{x}}\Pi_h^V\boldsymbol{\varphi}\rangle_{J\times\omega_K })\|_{2,J\times K}^2
	\\&\lesssim h_K^n+\|\bfF(\cdot,\cdot,\bfD_{\mathrm{x}}\boldsymbol{\varphi})-\bfF(\cdot,\cdot,\bfD_{\mathrm{x}}\mathrm{I}_\tau^{0,\mathrm{t}}\boldsymbol{\varphi})\|_{2,J\times \omega_K }^2 
	\\&\quad+\|\bfF(\cdot,\cdot,\bfD_{\mathrm{x}}\boldsymbol{\varphi})-\bfF(\cdot,\cdot,\langle \bfD_{\mathrm{x}}\Pi_h^V\boldsymbol{\varphi}\rangle_{J\times\omega_K })\|_{2,J\times K}^2
	\\&\quad+\|\bfF(\cdot,\cdot,\bfD_{\mathrm{x}}\Pi_h^V\boldsymbol{\varphi})-\bfF(\cdot,\cdot,\langle \bfD_{\mathrm{x}}\Pi_h^V\boldsymbol{\varphi}\rangle_{J\times\omega_K })\|_{2,J\times K}^2
	\\&\lesssim h_K^n+\|\bfF(\cdot,\cdot,\bfD_{\mathrm{x}}\boldsymbol{\varphi})-\bfF(\cdot,\cdot,\bfD_{\mathrm{x}}\mathrm{I}_\tau^{0,\mathrm{t}}\boldsymbol{\varphi})\|_{2,J\times \omega_K }^2 
	\\&\quad+\|\bfF(\cdot,\cdot,\bfD_{\mathrm{x}}\boldsymbol{\varphi})-\bfF(\cdot,\cdot, \bfD_{\mathrm{x}}\Pi_h^V\boldsymbol{\varphi})\|_{2,J\times K}^2
	\\&\quad+\|\bfF(\cdot,\cdot,\bfD_{\mathrm{x}}\boldsymbol{\varphi})-\bfF(\cdot,\cdot,\langle \bfD_{\mathrm{x}}\boldsymbol{\varphi}\rangle_{J\times\omega_K })\|_{2,J\times K}^2
	\\&\quad+\rho_{\varphi_{\vert\langle \bfD_{\mathrm{x}}\boldsymbol{\varphi}\rangle_{J\times\omega_K }\vert }}(\langle \bfD_{\mathrm{x}}\boldsymbol{\varphi}-\bfD_{\mathrm{x}}\Pi_h^V\boldsymbol{\varphi}\rangle_{J\times\omega_K })
	\\&\lesssim 	(\tau^{2\alpha_{\mathrm{t}}}+h^{2\alpha_{\mathrm{x}}})\,\tau\,{\sup}_{t\in J}{\big\{\|1+\vert \bfD_{\mathrm{x}}\boldsymbol{\varphi}(t)\vert^{p(t,\cdot)s}\|_{1,\omega_K }\big\}}
	\\&\quad+
	\tau^{2\beta_{\mathrm{t}}}\,[ \bfF(\cdot,\cdot,\bfD_{\mathrm{x}}\boldsymbol{\varphi})]_{N^{\beta_{\mathrm{t}},2}(J;L^2(\omega_K ))}^2
	\\&\quad+
	h^{2\beta_{\mathrm{x}}}\,[ \bfF(\cdot,\cdot,\bfD_{\mathrm{x}}\boldsymbol{\varphi})]_{L^2(J;N^{\beta_{\mathrm{x}},2}(\omega_K ))}^2\,,\\[-6mm]
\end{align*}
	which is the claimed local interpolation error estimate \eqref{lem:PiV_FI.1}.\vspace{-0.5mm}
	
	\emph{ad \eqref{lem:PiV_FI.2}.} The global interpolation error estimate \eqref{lem:PiV_FI.2} is obtained analogously to the proof of the local interpolation error estimate \eqref{lem:PiV_FI.1}.
	\end{proof}

For the fractional regularity of the kinematic pressure represented in Cald\'eron spaces, we have the following interpolation estimate for difference between $\Pi_\tau^{0,\mathrm{t}}\Pi_h^{k-1,\mathrm{x}}$ and $\Pi_\tau^{0,\mathrm{t}}\Pi_h^Q$ (\textit{cf}.\ Assumption~\ref{ass:PiQ}) with respect to the modular of the conjugate of a shifted generalized $N$-function, where~the~shift~is~variable.

\begin{lemma}\label{lem:PiQ}
    Suppose that $p\in C^{0,\alpha_{\mathrm{t}},\alpha_{\mathrm{x}}}(\overline{Q_T})$,  $\alpha_{\mathrm{t}},\alpha_{\mathrm{x}}\in (0,1]$, with $p^->1$,  let $\eta\in L^1(Q_T)$ be such that $\eta(t)\in C^{\gamma_{\mathrm{x}},p'(t,\cdot)}(\Omega)$ for a.e.\ $t\in I$ and $\vert \nabla_{\mathrm{x}}^{\gamma_{\mathrm{x}}} \eta\vert \in L^{p'(\cdot,\cdot)}(Q_T)$, $\gamma_{\mathrm{x}}\in (0,1]$, and let $\bfA\in (L^{p(\cdot,\cdot)}(Q_T))^{d\times d}$. Then, for every $n>0$, $J\in \mathcal{I}_\tau$, and  $K\in \mathcal{T}_h$, there holds
		\begin{align}\label{lem:PiQ.1}
			\begin{aligned}
			     \rho_{(\varphi_{\vert               \bfA\vert})^*,J\times K}\big(\Pi_\tau^{0,\mathrm{t}}(\Pi_h^Q\eta-\Pi_h^{k-1,\mathrm{x}}\eta)\big)&\lesssim h_K^n+\rho_{(\varphi_{\vert \bfA\vert})^*,J\times \omega_K }(h_K ^{\gamma_{\mathrm{x}}}\vert \nabla_{\mathrm{x}}^{\gamma_{\mathrm{x}}} \eta\vert)\\&\quad + \|\bfF(\cdot,\cdot,\bfA)-\bfF(\cdot,\cdot,\langle\bfA\rangle_{J\times\omega_K })\|_{2,J\times \omega_K }^2 \,,
			\end{aligned}
		\end{align}
		where the implicit constant in $\lesssim$ depends on $k$, $\ell$, $n$, $p^-$, $p^+$, $[p]_{\alpha_{\mathrm{t}},\alpha_{\mathrm{x}},Q_T}$, $\omega_0$, $\|\bfA\|_{p(\cdot,\cdot),Q_T}$, and $\|\vert \nabla_{\mathrm{x}}^{\gamma_{\mathrm{x}}}\eta \vert\|_{p'(\cdot,\cdot),Q_T}$. In particular, it follows that\enlargethispage{7.5mm}
		\begin{align}\label{lem:PiQ.2}
			\begin{aligned}
				\rho_{(\varphi_{\vert \bfA\vert})^*,Q_T}\big(\Pi_\tau^{0,\mathrm{t}}(\Pi_h^Q\eta -\Pi_h^{k-1,\mathrm{x}}\eta)\big)&\lesssim h^n+\rho_{(\varphi_{\vert\bfA\vert})^*,Q_T}(h^{\gamma_{\mathrm{x}}}\vert \nabla_{\mathrm{x}}^{\gamma_{\mathrm{x}}} \eta\vert)\\&\quad+
				\sum_{J\in \mathcal{I}_\tau}{\sum_{K\in \mathcal{T}_h}{\|\bfF(\cdot,\cdot,\bfA)-\bfF(\cdot,\cdot,\langle\bfA\rangle_{J\times\omega_K })\|_{2,J\times\omega_K }^2}}\,.  
			\end{aligned}
		\end{align} 
	\end{lemma}
	
	\begin{proof}
		\textit{ad \eqref{lem:PiQ.1}.} Using the shift change Lemma \ref{lem:shift-change}\eqref{lem:shift-change.3}, for every $J\in \mathcal{I}_\tau$ and $K\in \mathcal{T}_h$, we find that 
		\begin{align}\label{eq:PiQ.1}
			\begin{aligned}
				\rho_{(\varphi_{\vert \bfA\vert})^*,J\times K}\big(\Pi_\tau^{0,\mathrm{t}}(\Pi_h^Q \eta-\Pi_h^{k-1,\mathrm{x}}\eta)\big)&\leq c\, \rho_{(\varphi_{\vert \langle\bfA\rangle_{J\times \omega_K }\vert})^*,J\times K}\big(\Pi_\tau^{0,\mathrm{t}}(\Pi_h^Q \eta-\Pi_h^{k-1,\mathrm{x}}\eta)\big)\\&\quad+c\,\|\bfF(\cdot,\cdot,\bfA)-\bfF(\cdot,\cdot,\langle\bfA\rangle_{J\times \omega_K })\|_{2,J\times K}^2
				\,.\end{aligned}
		\end{align}
		Using that $\Pi_h^Q\langle \eta\rangle_{\omega_K }=\Pi_h^{k-1,\mathrm{x}}\langle \eta\rangle_{\omega_K }=\langle \eta\rangle_{\omega_K }$ a.e.\ in $J$, 
        the (local) inverse inequality (\textit{cf}.\ \cite[Lem.\ 12.1]{EG21}), the (local) $L^1$-stability of the projection operators $\Pi_h^{0,\mathrm{t}}$, $\Pi_h^{k-1,\mathrm{x}}$, and $\Pi_h^Q$ (\textit{cf}.\ Assumption \ref{ass:PiQ}\eqref{eq:PiQstab}), and that
        $\vert \eta-\langle \eta\rangle_{\omega_K }\vert\leq h_K ^{\gamma_{\mathrm{x}}}(\vert \nabla_{\mathrm{x}}^{\gamma_{\mathrm{x}}}\eta \vert+\langle \vert \nabla_{\mathrm{x}}^{\gamma_{\mathrm{x}}}\eta \vert\rangle_{\omega_K })$ a.e.\ in $J\times \omega_K $ (\textit{cf}.\ \cite[ineq.\ (5.17)]{berselli2023error}),
		we arrive at
		\begin{align}\label{eq:PiQ.2}
			\begin{aligned}
				\rho_{(\varphi_{\vert \langle\bfA\rangle_{J\times \omega_K }\vert})^*,J\times K}&\big(\Pi_\tau^{0,\mathrm{t}}(\Pi_h^Q \eta-\Pi_h^{k-1,\mathrm{x}}\eta)\big)\\&=	\rho_{(\varphi_{\vert \langle\bfA\rangle_{J\times \omega_K }\vert})^*,J\times K}\big(\Pi_\tau^{0,\mathrm{t}}(\Pi_h^Q \eta-\Pi_h^Q\langle \eta\rangle_{\omega_K }+\Pi_h^{k-1,\mathrm{x}}\langle \eta\rangle_{\omega_K }-\Pi_h^{k-1,\mathrm{x}}\eta)\big)
				\\&\leq c\,
				\rho_{(\varphi_{\vert \langle\bfA\rangle_{J\times \omega_K }\vert})^*,J\times K}\big(\Pi_\tau^{0,\mathrm{t}}\Pi_h^Q(\eta-\langle \eta\rangle_{\omega_K })\big)
				\\&\quad+c\,\rho_{(\varphi_{\vert \langle\bfA\rangle_{J\times \omega_K }\vert})^*,J\times K}\big(\Pi_\tau^{0,\mathrm{t}}\Pi_h^{k-1,\mathrm{x}}(\eta-\langle \eta\rangle_{\omega_K })\big) 
					\\&\leq c\,
				\rho_{(\varphi_{\vert \langle\bfA\rangle_{J\times \omega_K }\vert})^*,J\times K}\big(\|\Pi_\tau^{0,\mathrm{t}}\Pi_h^Q(\eta-\langle \eta\rangle_{\omega_K })\|_{\infty,J\times K}\big)
				\\&\quad+c\,\rho_{(\varphi_{\vert \langle\bfA\rangle_{J\times \omega_K }\vert})^*,J\times K}\big(\|\Pi_\tau^{0,\mathrm{t}}\Pi_h^{k-1,\mathrm{x}}(\eta-\langle \eta\rangle_{\omega_K })\|_{\infty,J\times K}\big) 
					\\&\leq c\,
				\rho_{(\varphi_{\vert \langle\bfA\rangle_{J\times \omega_K }\vert})^*,J\times K}\big(\langle\vert \Pi_\tau^{0,\mathrm{t}}\Pi_h^Q(\eta-\langle \eta\rangle_{\omega_K })\vert\rangle_{J\times K}\big)
				\\&\quad+c\,\rho_{(\varphi_{\vert \langle\bfA\rangle_{J\times \omega_K }\vert})^*,J\times K}\big(\langle\vert \Pi_\tau^{0,\mathrm{t}}\Pi_h^{k-1,\mathrm{x}}(\eta-\langle \eta\rangle_{\omega_K })\vert\rangle_{J\times K}\big) 
					\\&\leq c\,
				\rho_{(\varphi_{\vert \langle\bfA\rangle_{J\times \omega_K }\vert})^*,J\times K}\big(\langle \vert \eta-\langle \eta \rangle_{\omega_K }\vert\rangle_{J\times \omega_K}\big)
\\&\leq
				\rho_{(\varphi_{\vert \langle\bfA\rangle_{J\times \omega_K }\vert})^*,J\times K}\big(h_K ^{\gamma_{\mathrm{x}}}\langle \vert \nabla_{\mathrm{x}}^{\gamma_{\mathrm{x}}}\eta \vert\rangle_{J\times\omega_K }\big)  \,.
			\end{aligned}
		\end{align}
		Since	$\vert \langle\bfA\rangle_{J\times \omega_K }\vert +\langle \vert \nabla_{\mathrm{x}}^{\gamma_{\mathrm{x}}}\eta \vert\rangle_{J\times \omega_K}\leq  c\,\vert J\times K\vert^{-1}\leq c\,\vert J\times K\vert^{-n}$,  where  $c>0$ depends  on $k$, $\ell$, $p^-$, $p^+$, $\omega_0$, $\|\bfA\|_{p(\cdot,\cdot),Q_T}$, and $\| \vert \nabla_{\mathrm{x}}^{\gamma_{\mathrm{x}}}\eta \vert\|_{p'(\cdot,\cdot),Q_T}$, using 
        the key estimate (\textit{cf}.\ Lemma \ref{lem:key-estimate}) in~\eqref{eq:PiQ.2}, and  the shift change Lemma \ref{lem:shift-change}\eqref{lem:shift-change.3}, for every $J\in \mathcal{I}_\tau$ and $ K\in \mathcal{T}_h $, from \eqref{eq:PiQ.1}, we infer that
		\begin{align*}
			\begin{aligned}
			\rho_{(\varphi_{\vert \langle\bfA\rangle_{J\times \omega_K }\vert})^*,J\times K}\big(\Pi_\tau^{0,\mathrm{t}}(\Pi_h^Q \eta-\Pi_h^{k-1,\mathrm{x}}\eta)\big)&\leq 
		c\,\rho_{(\varphi_{\vert \langle\bfA\rangle_{J\times \omega_K }\vert})^*,J\times \omega_K }\big(h^{\gamma_{\mathrm{x}}}\vert \nabla_{\mathrm{x}}^{\gamma_{\mathrm{x}}}\eta \vert\big)+c\,h_K ^n
			\\&\leq c\,\rho_{(\varphi_{\vert\bfA\vert})^*,J\times \omega_K }\big(h^{\gamma_{\mathrm{x}}}\vert \nabla_{\mathrm{x}}^{\gamma_{\mathrm{x}}}\eta \vert\big)\\&\quad+c\,\|\bfF(\cdot,\cdot,\bfA)-\bfF(\cdot,\cdot,\langle\bfA\rangle_{J\times \omega_K })\|_{2,J\times K}^2+c\,h_K ^n\,,
			\end{aligned}
		\end{align*}
		which is the claimed local interpolation error estimate \eqref{lem:PiQ.1}.
		
		\textit{ad \eqref{lem:PiQ.2}.} The claimed global  interpolation error estimate \eqref{lem:PiQ.2} follows from the local interpolation error estimate \eqref{lem:PiQ.1} via summation with respect to $J\in \mathcal{I}_\tau$ and $ K\in \mathcal{T}_h $.
	\end{proof}

	\newpage 
	\section{\emph{A priori} error estimates}\label{sec:a_priori}
	
	\hspace{5mm}In this section, we prove the main results of this paper, \textit{i.e.}, we derive
	 \textit{a priori} error estimates for the approximation of the unsteady $p(\cdot,\cdot)$-Stokes equations  \eqref{eq:ptxStokes} (\textit{i.e.}, Problem (\hyperlink{Q}{Q})~and~Problem~(\hyperlink{P}{P}), respectively)
	via the discrete  unsteady $p(\cdot,\cdot)$-Stokes equations (\textit{i.e.}, Problem (\hyperlink{Qh}{Q$_h^\tau$}) and~Problem~(\hyperlink{Ph}{P$_h^\tau$}), respectively).\enlargethispage{5mm}

\begin{theorem}\label{thm:main}
    Suppose that $p\in C^{0,\alpha_{\mathrm{t}},\alpha_{\mathrm{x}}}(\overline{Q_T})$,  $\alpha_{\mathrm{t}},\alpha_{\mathrm{x}}\in (0,1]$, with $p^->1$, that  $\bfv\in \mathbfcal{V}(0)$ with
    \begin{align*}
        \left.\begin{aligned} 
        \bfF(\cdot,\cdot,\bfD_{\mathrm{x}}\bfv)&\in N^{\beta_{\mathrm{t}},2}(I;(L^2(\Omega))^{d\times d})\cap L^2(I;(N^{\beta_{\mathrm{x}},2}(\Omega))^{d\times d})\,,\\
	\bfv&\in L^\infty(I;(N^{\beta_{\mathrm{x}},2}             (\Omega))^d)
    \end{aligned}\quad\right\}\quad \text{ for some }\beta_{\mathrm{t}}\in (\tfrac{1}{2},1]\,,\;\beta_{\mathrm{x}}\in (0,1]\,,
    \end{align*}
     and that $q\in \mathbfcal{Q}(0)$ with $ q(t)\in C^{\gamma_{\mathrm{x}},p'(t,\cdot)}(\Omega)$~for~a.e.~${t\in I}$~and 
    \begin{align*}
        \vert \nabla_{\mathrm{x}}^{\gamma_{\mathrm{x}}} q\vert \in L^{p'(\cdot,\cdot)}(Q_T)\quad \text{ for some }\gamma_{\mathrm{x}}\in \big(\tfrac{\alpha_{\mathrm{x}}}{\min\{2,(p^+)'\}},1\big]\,.
    \end{align*}
    Moreover, let $h\hspace*{-0.1em}\sim\hspace*{-0.1em} h_K $~for~all~$K\hspace*{-0.1em}\in\hspace*{-0.1em} \mathcal{T}_h$ with $h^2\lesssim \tau$. In the case $p^-\in (1, 2]$, we additionally assume that $\bfv\in L^\infty(I;(N^{1+\beta_{\mathrm{x}},2}             (\Omega))^d)$.
    Then, there exists a~constant~$s\hspace*{-0.1em}>\hspace*{-0.1em}1$~with~$s\hspace*{-0.1em}\searrow\hspace*{-0.1em} 1$~as~$\tau+h\hspace*{-0.1em}\searrow\hspace*{-0.1em}0$~such~that 
    \begin{align*}
        \|\bfv_h^{\tau}-\mathrm{I}_\tau^{0,\mathrm{t}}\bfv\|_{L^\infty(I;L^2(\Omega))}^2
        &+\|\bfF_h^{\tau}(\cdot,\cdot,\bfD_{\mathrm{x}}\bfv_h^{\tau})-\bfF_h^{\tau}(\cdot,\cdot,\bfD_{\mathrm{x}}\mathrm{I}_\tau^{0,\mathrm{t}}\bfv)\|_{2,Q_T}^2\\&\lesssim (\tau^{2\alpha_{\mathrm{t}}}+h^{2\alpha_{\mathrm{x}}})\,\big(1+{\sup}_{t\in I}{\big\{\rho_{p(t,\cdot)s,\Omega}(\bfD_{\mathrm{x}} \mathbf{v}(t))}\big\}\big)
	\\&\quad+\tau^{2\beta_{\mathrm{t}}}\,[\bfF(\cdot,\cdot,\bfD_{\mathrm{x}}\bfv)]_{N^{\beta_{\mathrm{t}},2}(I;L^2(\Omega))}^2
        \\&\quad+h^{2\beta_{\mathrm{x}}}\,\big([\bfF(\cdot,\cdot,\bfD_{\mathrm{x}}\bfv)]_{L^2(I;N^{\beta_{\mathrm{x}},2}(\Omega))}^2+\epsilon^2(\bfv)\big)
        \\&\quad+\rho_{(\varphi_{\vert \bfD_{\mathrm{x}}\bfv\vert})^*,Q_T}\big(h^{\gamma_{\mathrm{x}}}\vert\nabla_{\mathrm{x}}^{\gamma_{\mathrm{x}}}q\vert\big)\,,
    \end{align*}
    where $\epsilon(\bfv)\hspace{-0.1em}\coloneqq\hspace{-0.1em}[ \bfv]_{L^\infty(I;N^{1+\beta_{\mathrm{x}},2}(\Omega))}$ if $p^-\hspace{-0.1em}\in\hspace{-0.1em} (1, 2]$ and $\epsilon(\bfv)\hspace{-0.1em}\coloneqq\hspace{-0.1em}[ \bfv]_{L^\infty(I;N^{\beta_{\mathrm{x}},2}(\Omega))}$ else 
    and the implicit~\mbox{constant} in $\lesssim $ depends on $k$, $\ell$,  $p^- $, $p^+$, $[p]_{\alpha_{\mathrm{t}},\alpha_{\mathrm{x}},Q_T}$, $\delta$, $\omega_0$, $\Omega$,  $s$, $\sup_{t\in I}{\big\{\|\bfD_{\mathrm{x}}\bfv(t)\|_{p(t,\cdot),\Omega}\big\}}$,~and~$\|\vert\nabla_{\mathrm{x}}^{\gamma_{\mathrm{x}}}q\vert\|_{p'(\cdot,\cdot),Q_T}$.
\end{theorem}
	
As a direct consequence of Theorem \ref{thm:main}, we obtain two error estimates with explicit~decay~rates.
	
\begin{corollary}\label{cor:main}
    Let the assumptions of Theorem \ref{thm:main} be satisfied. Then, there exists a constant  $s>1$ with $s\searrow1$  as $\tau+h\searrow0$ such that
    \begin{align}\label{cor:main.1}
        \begin{aligned} 
            \|\bfv_h^{\tau}-\mathrm{I}_\tau^{0,\mathrm{t}}\bfv\|_{L^\infty(I;L^2(\Omega))}^2
            &+\|\bfF_h^{\tau}(\cdot,\cdot,\bfD_{\mathrm{x}} \bfv_h^{\tau})-\bfF_h^{\tau}(\cdot,\cdot,\bfD_{\mathrm{x}}\mathrm{I}_\tau^{0,\mathrm{t}}
		\bfv)\|_{2,Q_T}^2
            \\&\lesssim 
            (\tau^{2\alpha_{\mathrm{t}}}+h^{2\alpha_{\mathrm{x}}})\,\big(1+{\sup}_{t\in I}{\big\{\rho_{p(t,\cdot)s,\Omega}(\bfD_{\mathrm{x}} \mathbf{v}(t))\big\}}\big)
		\\&\quad+\tau^{2\beta_{\mathrm{t}}}\,[\bfF(\cdot,\cdot,\bfD_{\mathrm{x}}\bfv)]_{N^{\beta_{\mathrm{t}},2}(I;L^2(\Omega))}^2
            \\&\quad+h^{2\beta_{\mathrm{x}}}\,\big([\bfF(\cdot,\cdot,\bfD_{\mathrm{x}}\bfv)]_{L^2(I;N^{\beta_{\mathrm{x}},2}(\Omega))}^2+\epsilon^2(\bfv)\big)
		\\[0.5mm]&\quad+h^{\smash{\min\{2,  (p^+)'\}\gamma_{\mathrm{x}}}}\,\rho_{(\varphi_{\vert \bfD_{\mathrm{x}}\bfv\vert})^*,Q_T}(\vert \nabla_{\mathrm{x}}^{\gamma_{\mathrm{x}}}q\vert)\,,
	\end{aligned}
    \end{align}
    where the implicit constant in $\lesssim $ depends on $k$, $\ell$,  $p^- $, $p^+$, $[p]_{\alpha_{\mathrm{t}},\alpha_{\mathrm{x}},Q_T}$, $\delta$, $\omega_0$, $\Omega$,  $s$, $\sup_{t\in I}{\big\{\|\bfD_{\mathrm{x}}\bfv(t)\|_{p(t,\cdot),\Omega}\big\}}$, and $\|\vert\nabla_{\mathrm{x}}^{\gamma_{\mathrm{x}}}q\vert\|_{p'(\cdot,\cdot),Q_T}$. If, in addition, $p^-\ge 2$ and $(\delta+\vert\bfD_{\mathrm{x}}\bfv\vert)^{2-p(\cdot,\cdot)}\vert \nabla_{\mathrm{x}}^{\gamma_{\mathrm{x}}} q\vert^2 \in L^1(Q_T)$, then
    \begin{align}\label{cor:main.2}
	\begin{aligned} 
            \|\bfv_h^{\tau}-\mathrm{I}_\tau^{0,\mathrm{t}}\bfv\|_{L^\infty(I;L^2(\Omega))}^2
            &+\|\bfF_h^{\tau}(\cdot,\cdot,\bfD_{\mathrm{x}}\bfv_h^{\tau})-\bfF_h^{\tau}(\cdot,\cdot,\bfD_{\mathrm{x}}\mathrm{I}_\tau^{0,\mathrm{t}}
		\bfv)\|_{2,Q_T}^2
            \\&\lesssim (\tau^{2\alpha_{\mathrm{t}}}+h^{2\alpha_{\mathrm{x}}})\,\big(1+{\sup}_{t\in I}{\big\{\rho_{p(t,\cdot)s,\Omega}(\bfD_{\mathrm{x}} \mathbf{v}(t))\big\}}\big)
		\\&\quad+\tau^{2\beta_{\mathrm{t}}}\,[\bfF(\cdot,\cdot,\bfD_{\mathrm{x}}\bfv)]_{N^{\beta_{\mathrm{t}},2}(I;L^2(\Omega))}^2
            \\&\quad+h^{2\beta_{\mathrm{x}}}\,\big([\bfF(\cdot,\cdot,\bfD_{\mathrm{x}}\bfv)]_{L^2(I;N^{\beta_{\mathrm{x}},2}(\Omega))}^2+\epsilon^2(\bfv)\big)
		\\&\quad+h^{2\gamma_{\mathrm{x}}} \|  (\delta+\vert\bfD_{\mathrm{x}}\bfv\vert)^{\smash{2-p(\cdot,\cdot)}}\vert\nabla_{\mathrm{x}}^{\gamma_{\mathrm{x}}} q\vert^2\|_{1,Q_T}\,,
	\end{aligned}
    \end{align}
    where the hidden constant in $\lesssim $ depends on $k$, $\ell$,  $p^- $, $p^+$, $[p]_{\alpha_{\mathrm{t}},\alpha_{\mathrm{x}},Q_T}$, $\delta$, $\omega_0$, $\Omega$,  $s$, $\sup_{t\in I}{\big\{\|\bfD_{\mathrm{x}}\bfv(t)\|_{p(t,\cdot),\Omega}\big\}}$, and $\|\vert\nabla_{\mathrm{x}}^{\gamma_{\mathrm{x}}}q\vert\|_{p'(\cdot,\cdot),Q_T}$.
\end{corollary}\newpage

\begin{proof}[Proof (of Theorem \ref{thm:main}).]
	To start with, we introduce the abbreviation $\bfe_h^\tau  \in \mathbb{P}^0(\mathcal{I}_\tau^0;(W^{1,p^-}_0(\Omega))^d)$,~\mbox{defined} by $\bfe_h^\tau \coloneqq \bfv_h^\tau -\mathrm{I}_\tau^{0,\mathrm{t}}\bfv$ a.e.\ in $I$ and $\bfe_h^\tau \coloneqq \bfv_h^0-\bfv(0)$ a.e.\ in $I_0$. Then,  using \eqref{eq:hammera}, the decomposition
    \begin{align}\label{thm:main.1}
		\bfe_h^\tau  = \Pi_h^V\bfe_h^\tau + \Pi_h^V\mathrm{I}_\tau^{0,\mathrm{t}} \bfv -\mathrm{I}_\tau^{0,\mathrm{t}}\bfv\quad\text{in }\mathbb{P}^0(\mathcal{I}_\tau;(W^{1,p^-}_0(\Omega))^d)\,,
	\end{align}
    for every $m=1,\ldots,M$, denoting by $Q_T^m\coloneqq (0,t_m)\times \Omega$ the temporally truncated cylinder,  we find that\enlargethispage{5mm}
	\begin{align}\label{thm:main.2}
			c\, \|\bfF_h^\tau(\cdot,\cdot,\bfD_{\mathrm{x}} \overline{\bfv}^{\tau}_h) - \bfF_h^\tau(\cdot,\cdot,\bfD_{\mathrm{x}}\mathrm{I}_\tau^{0,\mathrm{t}}\bfv)\|_{2,\smash{Q_T^m}}^2
            &\leq (\bfS_h^\tau(\cdot,\cdot,\bfD_{\mathrm{x}} \bfv_h^\tau) - \bfS_h^\tau(\cdot,\cdot,\bfD_{\mathrm{x}}\mathrm{I}_\tau^{0,\mathrm{t}}	\bfv),\bfD_{\mathrm{x}}\bfe_h^\tau)_{\smash{Q_T^m}}
			\\&= (\bfS_h^\tau(\cdot,\cdot,\bfD_{\mathrm{x}}\bfv_h^\tau) - \bfS(\cdot,\cdot,\bfD_{\mathrm{x}}\bfv),\bfD_{\mathrm{x}}\bfe_h^\tau)_{\smash{Q_T^m}}\notag
			\\&\quad+(\bfS(\cdot,\cdot,\bfD_{\mathrm{x}}\bfv)-\bfS_h^\tau(\cdot,\cdot,\bfD_{\mathrm{x}}\mathrm{I}_\tau^{0,\mathrm{t}}\bfv),\bfD_{\mathrm{x}}\bfe_h^\tau)_{\smash{Q_T^m}}\notag
			\\&= (\bfS_h^\tau(\cdot,\cdot,\bfD_{\mathrm{x}} \bfv_h^\tau) - \bfS(\cdot,\cdot,\bfD_{\mathrm{x}}
			\bfv),\bfD_{\mathrm{x}} \Pi_h^V\bfe_h^\tau)_{\smash{Q_T^m}}\notag
			\\&\quad+ (\bfS_h^\tau(\cdot,\cdot,\bfD_{\mathrm{x}}\bfv_h^\tau) - \bfS(\cdot,\cdot,\bfD_{\mathrm{x}} \bfv),\bfD_{\mathrm{x}}\Pi_h^V\mathrm{I}_\tau^{0,\mathrm{t}}\bfv-\bfD_{\mathrm{x}}\mathrm{I}_\tau^{0,\mathrm{t}}\bfv)_{\smash{Q_T^m}}\notag
			\\&\quad+(\bfS(\cdot,\cdot,\bfD_{\mathrm{x}}\bfv)-\bfS_h^\tau(\cdot,\cdot,\bfD_{\mathrm{x}} \bfv),\bfD \bfe_h^\tau)_{\smash{Q_T^m}}\notag
			\\&\quad+(\bfS_h^\tau(\cdot,\cdot,\bfD_{\mathrm{x}} \bfv)-\bfS_h^\tau(\cdot,\cdot,\bfD_{\mathrm{x}}\mathrm{I}_\tau^{0,\mathrm{t}}\bfv),\bfD_{\mathrm{x}}\bfe_h^\tau)_{\smash{Q_T^m}} \notag
			\\&\eqqcolon I_{m,h}^1+ I_{m,h}^2 +I_{m,h}^3 +I_{m,h}^4\,.\notag
	\end{align}
	Therefore, let us next estimate the terms $I_{m,h}^{i}$, $i=1,\cdots,4$, separately for all $m\in \{1,\ldots,M\}$:  
		
	\textit{ad $I_{m,h}^2$.} For every $m = 1,\ldots,M$, we have that
	\begin{align}\label{thm:main.3}
		\begin{aligned} 
		  I_{m,h}^2
            &=(\bfS_h^\tau(\cdot,\cdot,\bfD_{\mathrm{x}} \bfv_h^\tau) - \bfS_h^\tau(\cdot,\cdot,\bfD_{\mathrm{x}} \mathrm{I}_\tau^{0,\mathrm{t}}\bfv),\bfD_{\mathrm{x}}\Pi_h^V\mathrm{I}_\tau^{0,\mathrm{t}}\bfv-\bfD_{\mathrm{x}}\mathrm{I}_\tau^{0,\mathrm{t}}\bfv)_{\smash{Q_T^m}}
		  \\&\quad+(\bfS_h^\tau(\cdot,\cdot,\bfD_{\mathrm{x}} \mathrm{I}_\tau^{0,\mathrm{t}}\bfv)- \bfS(\cdot,\cdot,\bfD_{\mathrm{x}} \mathrm{I}_\tau^{0,\mathrm{t}}\bfv),\bfD_{\mathrm{x}}\Pi_h^V\mathrm{I}_\tau^{0,\mathrm{t}}\bfv-\bfD_{\mathrm{x}}\mathrm{I}_\tau^{0,\mathrm{t}}\bfv)_{\smash{Q_T^m}}
		 \\&\quad+(\bfS(\cdot,\cdot,\bfD_{\mathrm{x}} \mathrm{I}_\tau^{0,\mathrm{t}}\bfv)- \bfS(\cdot,\cdot,\bfD_{\mathrm{x}}\bfv),\bfD_{\mathrm{x}}\Pi_h^V\mathrm{I}_\tau^{0,\mathrm{t}}\bfv-\bfD_{\mathrm{x}}\mathrm{I}_\tau^{0,\mathrm{t}}\bfv)_{\smash{Q_T^m}}
	   \\&\eqqcolon I_{m,h}^{2,1}+I_{m,h}^{2,2}+I_{m,h}^{2,3}\,,
	   \end{aligned}
	\end{align} 
	and, again, we estimate $I_{m,h}^{2,i}$, $i=1,2,3$, separately for  all $m\in \{1,\ldots,M\}$:  
		
	\textit{ad $I_{m,h}^{2,1}$.} Resorting to the $\varepsilon$-Young inequality \eqref{ineq:young} with $\psi=(\varphi_h^\tau)_{\vert\bfD_{\mathrm{x}}\mathrm{I}_\tau^{0,\mathrm{t}}\bfv\vert}$, \eqref{eq:hammera}, and Lemma~\ref{lem:PiV_FI}\eqref{lem:PiV_FI.2}, for every $m = 1,\ldots,M$, we find that
	\begin{align}\label{thm:main.4}
		\begin{aligned} 
			\vert I_{m,h}^{2,1}\vert 
            &\leq \varepsilon\,\|\bfF_h^{\tau}(\cdot,\cdot,\bfD_{\mathrm{x}} \bfv_h^{\tau})-\bfF_h^{\tau}(\cdot,\cdot,\bfD_{\mathrm{x}}\mathrm{I}_\tau^{0,\mathrm{t}}\bfv)\|_{2,\smash{Q_T^m}}^2
            \\&\quad+c_\varepsilon\, \|\bfF_h^{\tau}(\cdot,\cdot,\bfD_{\mathrm{x}}\mathrm{I}_\tau^{0,\mathrm{t}}\bfv)-\bfF_h^{\tau}(\cdot,\cdot,\bfD_{\mathrm{x}}\Pi_h^V\mathrm{I}_\tau^{0,\mathrm{t}}\bfv)\|_{2,\smash{Q_T^m}}^2
			\\&\leq \varepsilon\,\|\bfF_h^{\tau}(\cdot,\cdot,\bfD_{\mathrm{x}} \bfv_h^{\tau})-\bfF_h^{\tau}(\cdot,\cdot,\bfD_{\mathrm{x}}\mathrm{I}_\tau^{0,\mathrm{t}}	\bfv)\|_{2,\smash{Q_T^m}}^2
			\\&\quad+ c_\varepsilon\,(\tau^{2\alpha_{\mathrm{t}}}+h^{2\alpha_{\mathrm{x}}})\,\big(1+{\sup}_{t\in (0,t_m)}{\big\{\rho_{p(t,\cdot)s,\Omega}(\bfD_{\mathrm{x}} \mathbf{v}(t))\big\}}\big)
            \\&\quad+c_\varepsilon\,\tau^{2\beta_{\mathrm{t}}}\,[\bfF(\cdot,\cdot,\bfD_{\mathrm{x}}\bfv)]_{N^{\beta_{\mathrm{t}},2}((0,t_m);L^2(\Omega))}^2
			\\&\quad+c_\varepsilon\,h^{2\beta_{\mathrm{x}}}\,[\bfF(\cdot,\cdot,\bfD_{\mathrm{x}}\bfv)]_{L^2((0,t_m);N^{\beta_{\mathrm{x}},2}(\Omega))}^2\,. 
		\end{aligned}
	\end{align}

	\textit{ad $I_{m,h}^{2,2}$.} By means of the $\varepsilon$-Young inequality \eqref{ineq:young} with $\psi=(\varphi_h^\tau)_{\vert\bfD_{\mathrm{x}}\mathrm{I}_\tau^{0,\mathrm{t}}\bfv\vert}$, \eqref{eq:hammera}, Lemma \ref{lem:A-Ah}\eqref{eq:Ah-A}, and Lemma \ref{lem:PiV_FI}\eqref{lem:PiV_FI.2}, for every $m = 1,\ldots,M$, we obtain
	\begin{align}\label{thm:main.5}
		\begin{aligned} 
			\vert I_{m,h}^{2,2}\vert
            & \leq c\,\|(\bfF_h^\tau)^*(\cdot,\cdot,\bfS_h^\tau(\cdot,\cdot,\bfD_{\mathrm{x}}\mathrm{I}_\tau^{0,\mathrm{t}} \bfv))-(\bfF_h^\tau)^*(\cdot,\cdot,\bfS(\cdot,\cdot,\bfD_{\mathrm{x}}\mathrm{I}_\tau^{0,\mathrm{t}} \bfv))\|_{2,\smash{Q_T^m}}^2
            \\&\quad+c\,\|\bfF_h^{\tau}(\cdot,\cdot,\bfD_{\mathrm{x}}\mathrm{I}_\tau^{0,\mathrm{t}}\bfv)-\bfF_h^{\tau}(\cdot,\cdot,\bfD_{\mathrm{x}}\Pi_h^V\mathrm{I}_\tau^{0,\mathrm{t}}\bfv)\|_{2,\smash{Q_T^m}}^2
			\\&\leq c\,(\tau^{2\alpha_{\mathrm{t}}}+h^{2\alpha_{\mathrm{x}}})\,\big(1+{\sup}_{t\in (0,t_m)}{\big\{\rho_{p(t,\cdot)s,\Omega}(\bfD_{\mathrm{x}}\bfv(t))\big\}}\big)
            \\&\quad+c\,\tau^{2\beta_{\mathrm{t}}}\,[\bfF(\cdot,\cdot,\bfD_{\mathrm{x}}\bfv)]_{N^{\beta_{\mathrm{t}},2}((0,t_m);L^2(\Omega))}^2
            \\&\quad+
            c\,h^{2\beta_{\mathrm{x}}}\,[\bfF(\cdot,\cdot,\bfD_{\mathrm{x}}\bfv)]_{L^2((0,t_m);N^{\beta_{\mathrm{x}},2}(\Omega))}^2\,,
		\end{aligned}
    \end{align} 
	where, for every $m = 1,\ldots,M$,  we used that
	\begin{align}\label{thm:main.6}
		\begin{aligned} 
			\rho_{p(\cdot,\cdot)s,Q_T^m}(\bfD_{\mathrm{x}}\mathrm{I}_\tau^{0,\mathrm{t}} \mathbf{v})
            &\leq c\,\big(1+\rho_{(\mathrm{I}_\tau^{0,\mathrm{t}} p)(\cdot,\cdot)s,Q_T^m}(\bfD_{\mathrm{x}}\mathrm{I}_\tau^{0,\mathrm{t}} \mathbf{v})\big)
			\\&\leq c\,\big(1+{\sup}_{t\in (0,t_m)}{\big\{\rho_{p(t,\cdot)s,\Omega}(\bfD_{\mathrm{x}}\bfv(t))\big\}}\big)\,.
		\end{aligned}
	\end{align}
		
	\textit{ad $I_{m,h}^{2,3}$.} Using the $\varepsilon$-Young inequality \eqref{ineq:young} with $\psi=(\varphi_h^\tau)_{\vert\bfD_{\mathrm{x}}\mathrm{I}_\tau^{0,\mathrm{t}}\bfv\vert}$, \eqref{eq:hammera}, Lemma \ref{lem:Inod_Fh}\eqref{lem:Inod_Fh.2}, and Lemma \ref{lem:PiV_FI}\eqref{lem:PiV_FI.2}, for every $m = 1,\ldots,M$, we see that 
	\begin{align}\label{thm:main.7}
		\begin{aligned} 
			\vert I_{m,h}^{2,3}\vert
            & \leq c\,\|\bfF_h^\tau(\cdot,\cdot,\bfD_{\mathrm{x}}\mathrm{I}_\tau^{0,\mathrm{t}} \bfv)-\bfF_h^\tau(\cdot,\cdot,\bfD_{\mathrm{x}}\bfv)\|_{2,\smash{Q_T^m}}^2
            \\&\quad+c\,\|\bfF_h^\tau(\cdot,\cdot,\bfD_{\mathrm{x}}\mathrm{I}_\tau^{0,\mathrm{t}}\bfv)-\bfF_h^\tau(\cdot,\cdot,\bfD_{\mathrm{x}}\Pi_h^V\mathrm{I}_\tau^{0,\mathrm{t}}\bfv)\|_{2,\smash{Q_T^m}}^2
			\\&\leq c\,(\tau^{2\alpha_{\mathrm{t}}}+h^{2\alpha_{\mathrm{x}}})\,\big(1+{\sup}_{t\in (0,t_m)}{\big\{\rho_{p(t,\cdot)s,\Omega}(\bfD_{\mathrm{x}} \mathbf{v}(t))\big\}}\big)
            \\&\quad+c\,\tau^{2\beta_{\mathrm{t}}}\,[\bfF(\cdot,\cdot,\bfD_{\mathrm{x}}\bfv)]_{N^{\beta_{\mathrm{t}},2}((0,t_m);L^2(\Omega))}^2
            \\&\quad+c\,h^{2\beta_{\mathrm{x}}}\,[\bfF(\cdot,\cdot,\bfD_{\mathrm{x}}\bfv)]_{L^2((0,t_m);N^{\beta_{\mathrm{x}},2}(\Omega))}^2\,.
		\end{aligned}
	\end{align} 
	In summary, combining \eqref{thm:main.4}, \eqref{thm:main.5}, and \eqref{thm:main.7} in \eqref{thm:main.3}, for every $m = 1,\ldots,M$, we arrive at\enlargethispage{16mm}
	\begin{align}\label{thm:main.9}
		\begin{aligned} 
			\vert I_{m,h}^2\vert 	
            &\leq \varepsilon\,\|\bfF_h^{\tau}(\cdot,\cdot,\bfD_{\mathrm{x}} \bfv_h^{\tau})-\bfF_h^{\tau}(\cdot,\cdot,\bfD_{\mathrm{x}}\mathrm{I}_\tau^{0,\mathrm{t}}	\bfv)\|_{2,\smash{Q_T^m}}^2
			\\&\quad+ c_\varepsilon\,(\tau^{2\alpha_{\mathrm{t}}}+h^{2\alpha_{\mathrm{x}}})\,\big(1+{\sup}_{t\in (0,t_m)}{\big\{\rho_{p(t,\cdot)s,\Omega}(\bfD_{\mathrm{x}} \mathbf{v}(t))\big\}}\big)
            \\&\quad+c_\varepsilon\,\tau^{2\beta_{\mathrm{t}}}\,[\bfF(\cdot,\cdot,\bfD_{\mathrm{x}}\bfv)]_{N^{\beta_{\mathrm{t}},2}((0,t_m);L^2(\Omega))}^2
			\\&\quad+c_\varepsilon\,h^{2\beta_{\mathrm{x}}}\,[\bfF(\cdot,\cdot,\bfD_{\mathrm{x}}\bfv)]_{L^2((0,t_m);N^{\beta_{\mathrm{x}},2}(\Omega))}^2\,. 
		\end{aligned}
	\end{align}
		
	\textit{ad $I_{m,h}^3$.} Using the $\varepsilon$-Young inequality \eqref{ineq:young} with $\psi=(\varphi_h^\tau)_{\vert\bfD_{\mathrm{x}}\mathrm{I}_\tau^{0,\mathrm{t}}\bfv\vert}$, \eqref{eq:hammera}, and Lemma \ref{lem:A-Ah}\eqref{eq:Ah-A}, for every $m=1,\ldots,M$, we observe that
	\begin{align}\label{thm:main.10}
		\begin{aligned} 
			\vert I_{m,h}^3\vert
            & \leq c_\varepsilon\,\|(\bfF_h^\tau)^*(\cdot,\cdot,\bfS_h^\tau(\cdot,\cdot,\bfD_{\mathrm{x}}\bfv))-(\bfF_h^\tau)^*(\cdot,\cdot,\bfS(\cdot,\cdot,\bfD_{\mathrm{x}}\bfv))\|_{2,\smash{Q_T^m}}^2
            \\&\quad+\varepsilon\,\|\bfF_h^{\tau}(\cdot,\cdot,\bfD_{\mathrm{x}}\bfv_h^\tau)-\bfF_h^{\tau}(\cdot,\cdot,\bfD_{\mathrm{x}}\mathrm{I}_\tau^{0,\mathrm{t}}\bfv)\|_{2,\smash{Q_T^m}}^2
            \\& \leq c_\varepsilon\,\smash{(\tau^{2\alpha_{\mathrm{t}}}+h^{2\alpha_{\mathrm{x}}})}\,\big(1+{\sup}_{t\in (0,t_m)}{\big\{\rho_{p(t,\cdot)s,\Omega}(\bfD_{\mathrm{x}}\bfv(t))\big\}}\big)
            \\&\quad+\varepsilon\,\|\bfF_h^{\tau}(\cdot,\cdot,\bfD_{\mathrm{x}}\bfv_h^\tau)-\bfF_h^{\tau}(\cdot,\cdot,\bfD_{\mathrm{x}}\mathrm{I}_\tau^{0,\mathrm{t}}\bfv)\|_{2,\smash{Q_T^m}}^2\,.
		\end{aligned}
	\end{align} 
		
	\textit{ad $I_{m,h}^4$.} Using the $\varepsilon$-Young inequality \eqref{ineq:young} with $\psi=(\varphi_h^\tau)_{\vert\bfD_{\mathrm{x}}\mathrm{I}_\tau^{0,\mathrm{t}}\bfv\vert}$, \eqref{eq:hammera}, and Lemma \ref{lem:Inod_Fh}\eqref{lem:Inod_Fh.2}, for every $m = 1,\ldots,M$, we find that
	\begin{align}\label{thm:main.11}
		\begin{aligned} 
			\vert I_{m,h}^4\vert
            & \leq c_\varepsilon\,\|\bfF_h^{\tau}(\cdot,\cdot,\bfD_{\mathrm{x}}\bfv)-\bfF_h^{\tau}(\cdot,\cdot,\bfD_{\mathrm{x}}\mathrm{I}_\tau^{0,\mathrm{t}}\bfv)\|_{2,\smash{Q_T^m}}^2
            \\&\quad+\varepsilon\,\|\bfF_h^{\tau}(\cdot,\cdot,\bfD_{\mathrm{x}}\bfv_h^\tau)-\bfF_h^{\tau}(\cdot,\cdot,\bfD_{\mathrm{x}}\mathrm{I}_\tau^{0,\mathrm{t}}\bfv)\|_{2,\smash{Q_T^m}}^2
			\\&\leq c_\varepsilon\,  (\tau^{2\alpha_{\mathrm{t}}}+h^{2\alpha_{\mathrm{x}}})\,\big(1+{\sup}_{t\in (0,t_m)}{\big\{\rho_{p(t,\cdot)s,\Omega}(\bfD_{\mathrm{x}}\bfv(t))\big\}}\big)
            \\&\quad+c_\varepsilon\,\tau^{2\beta_{\mathrm{t}}}\,[\bfF(\cdot,\cdot,\bfD_{\mathrm{x}}\bfv)]_{N^{\beta_{\mathrm{t}},2}((0,t_m);L^2(\Omega))}^2
            \\&\quad+ 
		    \varepsilon\,\|\bfF_h^{\tau}(\cdot,\cdot,\bfD_{\mathrm{x}}\bfv_h^\tau)-\bfF_h^{\tau}(\cdot,\cdot,\bfD_{\mathrm{x}}\mathrm{I}_\tau^{0,\mathrm{t}}\bfv)\|_{2,\smash{Q_T^m}}^2\,.
	    \end{aligned}
	\end{align} 
		 
	\textit{ad $I_{m,h}^1$.} \hspace{-0.1mm}Testing \hspace{-0.1mm}the \hspace{-0.1mm}first \hspace{-0.1mm}lines \hspace{-0.1mm}of \hspace{-0.1mm}Problem \hspace{-0.1mm}(\hyperlink{Qh}{Q$_h^\tau$}) \hspace{-0.1mm}and \hspace{-0.1mm}Problem \hspace{-0.1mm}(\hyperlink{Q}{Q}) \hspace{-0.1mm}with \hspace{-0.1mm}${\boldsymbol{\varphi}_h^{\tau}\hspace{-0.175em}=\hspace{-0.175em}\Pi_h^V\bfe_h^\tau\chi_{(0,t_m)}\hspace{-0.175em}\in\hspace{-0.175em} \mathbb{P}^0(\mathcal{I}_\tau;\Vo_{h,0})}$ (more precisely, for Problem (\hyperlink{Q}{Q}), for every $m\hspace*{-0.175em}=\hspace*{-0.175em}1,\ldots,M$,  in Remark \ref{rem:equiv_form}, we choose ${\boldsymbol{\varphi}_{I_m}\hspace*{-0.175em}=\hspace*{-0.175em}\Pi_h^V\bfe_h^\tau\chi_{I_m}\hspace*{-0.175em}\in \hspace*{-0.175em}\Vo_{h,0}}$ and sum with respect to $m=1,\ldots,M$), then, subtracting the resulting equations as well as using that $(q,\mathrm{div}_{\mathrm{x}}  \Pi_h^V\bfe_h^\tau)_{Q_T^m} =(\Pi_\tau^{0,\mathrm{t}}\Pi_h^{k-1,\mathrm{x}}q,\mathrm{div}_{\mathrm{x}} \Pi_h^V\bfe_h^\tau)_{Q_T^m}$  and $(q_h,\mathrm{div}_{\mathrm{x}}  \Pi_h^V\bfe_h)_{Q_T^m} =0 =(\Pi_\tau^{0,\mathrm{t}}\Pi_h^Q q,\mathrm{div}_{\mathrm{x}}  \Pi_h^V\bfe_h)_{Q_T^m}$ (\textit{cf}.\ Assumption \ref{ass:PiV} (ii)), we find that
	\begin{align}\label{thm:main.12}
		\begin{aligned} 
			I_{m,h}^1&=(\Pi_\tau^{0,\mathrm{t}}(\Pi_h^Qq-\Pi_h^{k-1,\mathrm{x}}q),\mathrm{div}_{\mathrm{x}}  \Pi_h^V \bfe_h^\tau)_{\smash{Q_T^m}}+(\mathrm{d}_\tau \bfe_h^\tau, \Pi_h^V \bfe_h^\tau)_{\smash{Q_T^m}}
            \eqqcolon  I_{m,h}^{1,1}+ I_{m,h}^{1,2} \,.
		\end{aligned}
	\end{align} 
	Hence, let us next estimate $I_{m,h}^{1,1}$ and $I_{m,h}^{1,2}$ separately for all $m \in\{ 1,\ldots,M\}$:
		
	\textit{ad $I_{m,h}^{1,1}$.} Using the decomposition \eqref{thm:main.1},  the $\varepsilon$-Young inequality~\eqref{ineq:young} for 
	$\psi =(\varphi_h^\tau)_{\vert\bfD \mathrm{I}_\tau^{0,\mathrm{t}}\bfv\vert}$,~and~\eqref{eq:hammera}, for every $m = 1,\ldots,M$, 
    we find that\footnote{Here, by $\mathbf{I}_d=(\delta_{ij})_{i,j\in \{1,\ldots,d\}}\in \mathbb{R}^{d\times d}$ we denote the identity matrix.}  
    \begin{align}\label{thm:main.13}
		\begin{aligned}
			I_{m,h}^{1,1}
			  &=(\Pi_\tau^{0,\mathrm{t}}(\Pi_h^Qq-\Pi_h^{k-1,\mathrm{x}}q)\mathbf{I}_d,\bfD_{\mathrm{x}}\bfe_h^\tau+(\bfD_{\mathrm{x}}\Pi_h^V \mathrm{I}_\tau^{0,\mathrm{t}}\bfv-\bfD_{\mathrm{x}}\mathrm{I}_\tau^{0,\mathrm{t}}\bfv))_{\smash{Q_T^m}} 
			\\&\leq c_\varepsilon\,\rho_{((\varphi_h^{\tau})_{\vert\bfD_{\mathrm{x}}\mathrm{I}_\tau^{0,\mathrm{t}}\bfv\vert})^*,\smash{Q_T^m}}\big(\Pi_\tau^{0,\mathrm{t}}(\Pi_h^Qq-\Pi_h^{k-1,\mathrm{x}}q)\big)
            \\&\quad+
			\varepsilon\, c\,\|\bfF_h^{\tau}(\cdot,\cdot,\bfD_{\mathrm{x}}\bfv_h^{\tau})-\bfF_h^{\tau}(\cdot,\cdot,\bfD_{\mathrm{x}}\mathrm{I}_\tau^{0,\mathrm{t}}\bfv)\|_{2,\smash{Q_T^m}}^2
            \\&\quad+\varepsilon\,c\,\|\bfF_h^{\tau}(\cdot,\cdot,\bfD_{\mathrm{x}}\mathrm{I}_\tau^{0,\mathrm{t}}\bfv)-\bfF_h^{\tau}(\cdot,\cdot,\bfD_{\mathrm{x}}\Pi_h^V\mathrm{I}_\tau^{0,\mathrm{t}}\bfv)\|_{2,\smash{Q_T^m}}^2\,.
		\end{aligned}
	\end{align}
    Moreover, using the shift change Lemma \ref{lem:shift-change}\eqref{lem:shift-change.3}, for every $m = 1,\ldots,M$, we find that 
	\begin{align}\label{thm:main.13.2}
		\begin{aligned}
			\rho_{((\varphi_h^{\tau})_{\vert\bfD_{\mathrm{x}}\bfv\vert})^*,\smash{Q_T^m}}\big(\Pi_\tau^{0,\mathrm{t}}(\Pi_h^Qq-\Pi_h^{k-1,\mathrm{x}}q)\big)&\leq c\,\rho_{((\varphi_h^{\tau})_{\vert\bfD_{\mathrm{x}}\mathrm{I}_\tau^{0,\mathrm{t}}\bfv\vert})^*,\smash{Q_T^m}}\big(\Pi_\tau^{0,\mathrm{t}}(\Pi_h^Qq-\Pi_h^{k-1,\mathrm{x}}q)\big)
            \\&\quad + c\,\|\bfF_h^\tau(\cdot,\cdot,\bfD_{\mathrm{x}}\mathrm{I}_\tau^{0,\mathrm{t}} \bfv)-\bfF_h^\tau(\cdot,\cdot,\bfD_{\mathrm{x}}\bfv)\|_{2,\smash{Q_T^m}}^2\,.
		\end{aligned}
	\end{align}
	Using Lemma \ref{lem:A-Ah}\eqref{eq:phih-phi} (with $\lambda=h^{\widetilde{\gamma}_{\mathrm{x}}}$, where $\widetilde{\gamma}_{\mathrm{x}}\coloneqq \frac{\alpha_{\mathrm{x}}}{\min\{2,(p^+)'\}}$, $g=h^{-\widetilde{\gamma}_{\mathrm{x}}}(\Pi_h^Q q-\Pi_h^{k-1,\mathrm{x}}q)$,~and~${\bfA=\bfD_{\mathrm{x}}\bfv}$), and Lemma \ref{lem:PiQ}\eqref{lem:PiQ.2} together with Lemma~\ref{lem:poincare_F}\eqref{lem:poincare_F.4}, for every $m = 1,\ldots,M$, we obtain that
	\begin{align}\label{thm:main.14}
		\begin{aligned}
            \rho_{((\varphi_h^{\tau})_{\vert\bfD_{\mathrm{x}}\bfv\vert})^*,\smash{Q_T^m}}\big(\Pi_\tau^{0,\mathrm{t}}(\Pi_h^Qq-\Pi_h^{k-1,\mathrm{x}}q)\big)
			&\leq c\,\rho_{(\varphi_{\vert\bfD_{\mathrm{x}}\bfv\vert})^*,\smash{Q_T^m}}\big(\Pi_\tau^{0,\mathrm{t}}(\Pi_h^Qq-\Pi_h^{k-1,\mathrm{x}}q)\big)
            \\&\quad+ c\,h^{\alpha_{\mathrm{x}}}(\tau^{\alpha_{\mathrm{t}}}+h^{\alpha_{\mathrm{x}}})\,\big(1+{\sup}_{t\in (0,t_m)}{\big\{\rho_{p(t,\cdot)s,\Omega}(\bfD_{\mathrm{x}}\bfv(t))\big\}}\big)
			\\&\quad + c\,h^{\alpha_{\mathrm{x}}}(\tau^{\alpha_{\mathrm{t}}}+h^{\alpha_{\mathrm{x}}})\,\rho_{p'(\cdot,\cdot)s,\smash{Q_T^m}}\big(h^{-\widetilde{\gamma}_{\mathrm{x}}}\Pi_\tau^{0,\mathrm{t}}(\Pi_h^Qq-\Pi_h^{k-1,\mathrm{x}}q)\big)
			\\&\leq c\,\rho_{((\varphi_h^{\tau})_{\vert\bfD_{\mathrm{x}}\bfv\vert})^*,\smash{Q_T^m}}\big(\Pi_\tau^{0,\mathrm{t}}(\Pi_h^Qq-\Pi_h^{k-1,\mathrm{x}}q)\big)
            \\&\quad+ c\,  (\tau^{2\alpha_{\mathrm{t}}}+h^{2\alpha_{\mathrm{x}}})\,\big(1+{\sup}_{t\in (0,t_m)}{\big\{\rho_{p(t,\cdot)s,\Omega}(\bfD_{\mathrm{x}}\bfv(t))\big\}}\big)
			\\&\quad + c\,h^{2\alpha_{\mathrm{x}}}\,\rho_{p'(\cdot,\cdot)s,\smash{Q_T^m}}
			\big(h^{-\widetilde{\gamma}_{\mathrm{x}}}\Pi_\tau^{0,\mathrm{t}}(\Pi_h^Qq-\Pi_h^{k-1,\mathrm{x}}q)\big)
			\\&\leq c\,(\tau^{2\alpha_{\mathrm{t}}}+h^{2\alpha_{\mathrm{x}}})\,\big(1+{\sup}_{t\in (0,t_m)}{\big\{\rho_{p(t,\cdot)s,\Omega}(\bfD_{\mathrm{x}}\bfv(t))\big\}}\big)
            \\&\quad+c\,\tau^{2\beta_{\mathrm{t}}}\,[\bfF(\cdot,\cdot,\bfD_{\mathrm{x}}\bfv)]_{N^{\beta_{\mathrm{t}},2}((0,t_m);L^2(\Omega))}^2
            \\&\quad+c\,h^{2\beta_{\mathrm{x}}}\,[\bfF(\cdot,\cdot,\bfD_{\mathrm{x}}\bfv)]_{L^2((0,t_m);N^{\beta_{\mathrm{x}},2}(\Omega))}^2
            \\&\quad+ c\,\rho_{(\varphi_{\vert \bfD_{\mathrm{x}}\bfv\vert})^*,\smash{Q_T^m}}\big(h^{\gamma_{\mathrm{x}}}\vert\nabla_{\mathrm{x}}^{\gamma_{\mathrm{x}}}q\vert\big)
			\\&\quad + c\,h^{2\alpha_{\mathrm{x}}}\,\rho_{p'(\cdot,\cdot)s,\smash{Q_T^m}}(h^{-\widetilde{\gamma}_{\mathrm{x}}}\Pi_\tau^{0,\mathrm{t}}(\Pi_h^Qq-\Pi_h^{k-1,\mathrm{x}}q))\,.
		\end{aligned}\hspace{-10mm}
	\end{align}
	Therefore, it is left to estimate the last term on the right-hand side of \eqref{thm:main.14}. To this end,~note~first~that   $r^{p'(t,x)s}\leq  
    c\, (1+r^{p'(t_m,\xi_K)\widetilde{s}})$ for all $r\ge 0$, $m=0,\ldots,M$, $K\in \mathcal{T}_h$, and $(t,x)^\top\in I_m\times K$, where $\widetilde{s}>1$ exists a constant with
    $\widetilde{s}\searrow  1$ as $\tau+h\searrow0$. Then, a (local) inverse inequality (\textit{cf}.\ \mbox{\cite[Lem.\ 12.1]{EG21}}),  that $h\sim h_K $ for all $K\in \mathcal{T}_h$, and ${\sum_{i\in \mathbb{N}}{\vert a_i\vert^{\widetilde{s}}}\leq (\sum_{i\in \mathbb{N}}{\vert a_i\vert})^{\widetilde{s}}}$ for all~$(a_i)_{i\in \mathbb{N}}\subseteq\ell^1(\mathbb{N})$,~yield~that
	\begin{align} 
			\rho_{p'(\cdot,\cdot)s,\smash{Q_T}}
            \big(h^{-\widetilde{\gamma}_{\mathrm{x}}}\Pi_\tau^{0,\mathrm{t}}(\Pi_h^Q q\hspace{-0.15em}-\hspace{-0.15em}\Pi_h^{k-1,\mathrm{x}}q)\big)
            &\leq c\,\big(1\hspace{-0.15em}+ \hspace{-0.15em}\rho_{p'_h(\cdot,\cdot)\widetilde{s},\smash{Q_T}}
			\big(h^{-\widetilde{\gamma}_{\mathrm{x}}}\Pi_\tau^{0,\mathrm{t}}(\Pi_h^Q q\hspace{-0.15em}-\hspace{-0.15em}\Pi_h^{k-1,\mathrm{x}}q)\big)\big) \label{thm:main.15}
			\\&\leq  c\,\big(1\hspace{-0.15em}+\hspace{-0.15em} 
			(\tau\hspace{-0.15em}+\hspace{-0.15em}h)^{\smash{(d+1)(1-\widetilde{s})}}\rho_{p'_h(\cdot,\cdot),\smash{Q_T}}
			\big(h^{-\widetilde{\gamma}_{\mathrm{x}}} \Pi_\tau^{0,\mathrm{t}}(\Pi_h^Q q\hspace{-0.15em}-\hspace{-0.15em}\Pi_h^{k-1,\mathrm{x}}q)\big)^{\widetilde{s}}\big)\,.\notag
	\end{align}
    Due to  Lemma \ref{lem:PiQ}\eqref{lem:PiQ.2} (with $\delta=0$, $\eta=h^{-\widetilde{\gamma}_{\mathrm{x}}}q$, $\bfA=\mathbf{0}$, and $n=\min\{2,(p^+)'\}(\gamma_{\mathrm{x}}-\widetilde{\gamma}_{\mathrm{x}})$),~we~have~that\enlargethispage{12.5mm}
    \begin{align*}
        \rho_{p'(\cdot,\cdot),\smash{Q_T}}
			\big(h^{-\widetilde{\gamma}_{\mathrm{x}}} \Pi_\tau^{0,\mathrm{t}}(\Pi_h^Q q-\Pi_h^{k-1,\mathrm{x}}q)\big) &\leq c\,h^{\min\{2,(p^+)'\}(\gamma_{\mathrm{x}}-\widetilde{\gamma}_{\mathrm{x}})} +c\,\rho_{p'(\cdot,\cdot),\smash{Q_T}}
			\big(h^{\gamma_{\mathrm{x}}-\widetilde{\gamma}_{\mathrm{x}}}\vert \nabla^{\gamma_{\mathrm{x}}}_{\mathrm{x}}q\vert\big)
            \\&\leq c\,h^{\min\{2,(p^+)'\}(\gamma_{\mathrm{x}}-\widetilde{\gamma}_{\mathrm{x}})}\big(1+\rho_{p'(\cdot,\cdot),\smash{Q_T^m}}
			(\vert \nabla^{\gamma_{\mathrm{x}}}_{\mathrm{x}}q\vert) \big)\,.
    \end{align*}
	Next, using the norm equivalence $\|\cdot\|_{p'_h(\cdot,\cdot),Q_T}\sim \|\cdot\|_{p'(\cdot,\cdot),Q_T}$ on $\mathbb{P}^{\max\{\ell,k-1\}}(\mathcal{I}_\tau\times\mathcal{T}_h)$ (\textit{cf}.\ \cite[Lem.~4.12]{berselli2023error}), where $\sim$ depends on $k$, $\ell$, $p^-$, $p^+$, $[p]_{\log,Q_T}$, and $\omega_0$, and \cite[Lem.\ A.1]{berselli2023error},~we~find~that 
	\begin{align*}
		\begin{aligned}
			\big\|h^{-\widetilde{\gamma}_{\mathrm{x}}}\Pi_\tau^{0,\mathrm{t}}(\Pi_h^Q q-\Pi_h^{k-1,\mathrm{x}}q)\big\|_{p'_h(\cdot,\cdot),Q_T}
            &\leq  c\, \big\|h^{-\widetilde{\gamma}_{\mathrm{x}}}\Pi_\tau^{0,\mathrm{t}}(\Pi_h^Q q-\Pi_h^{k-1,\mathrm{x}}q)\big\|_{p'(\cdot,\cdot),Q_T} 
			\\[-1.5mm]&\leq c\, h^{\frac{\min\{2,(p^+)'\}}{\max\{2,(p^-)'\}}(\gamma_{\mathrm{x}}-\widetilde{\gamma}_{\mathrm{x}})}\big(1+\rho_{p'(\cdot,\cdot),\smash{Q_T}}
			(\vert \nabla^{\gamma_{\mathrm{x}}}_{\mathrm{x}}q\vert) \big)^{\frac{1}{\max\{2,(p^+)'\}}}
			\,,
		\end{aligned}
	\end{align*}
	which, appealing to \cite[Lem.\ 3.2.5]{dhhr},  implies that 
	\begin{align}\label{thm:main.17}
		\rho_{p'_h(\cdot,\cdot),\smash{Q_T}}
		\big(h^{-\widetilde{\gamma}_{\mathrm{x}}}(\Pi_h^Q q-\Pi_h^{k-1,\mathrm{x}}q)\big) 
		&\leq c\, \smash{h^{\frac{\min\{2,(p^+)'\}}{\max\{2,(p^-)'\}}(\gamma_{\mathrm{x}}-\widetilde{\gamma}_{\mathrm{x}})}}\,. 
	\end{align}
    If we choose $\widetilde{s}\hspace{-0.1em}>\hspace{-0.1em}1$ close to $1$  such that $(d+1)(1-\widetilde{s})+\widetilde{s}\frac{\min\{2,(p^+)'\}}{\max\{2,(p^-)'\}}(\gamma_{\mathrm{x}}-\widetilde{\gamma}_{\mathrm{x}})\hspace{-0.1em}\ge\hspace{-0.1em}  0$,~which~is~possible~as~$\gamma_{\mathrm{x}}\hspace{-0.1em}>\hspace{-0.1em}\widetilde{\gamma}_{\mathrm{x}}$, then from \eqref{thm:main.15} together with \eqref{thm:main.17} in \eqref{thm:main.14}, for every $m=1,\ldots,M$,~we~deduce~that
	\begin{align}\label{thm:main.19}
		\begin{aligned}
            \rho_{((\varphi_h^{\tau})_{\vert\bfD_{\mathrm{x}}\bfv\vert})^*,\smash{Q_T^m}}\big(\Pi_\tau^{0,\mathrm{t}}(\Pi_h^Q q-\Pi_h^{k-1,\mathrm{x}}q)\big)
            &\leq c\,(\tau^{2\alpha_{\mathrm{t}}}+h^{2\alpha_{\mathrm{x}}})\,\big(1+{\sup}_{t\in (0,t_m)}{\big\{\rho_{p(t,\cdot)s,\Omega}(\bfD_{\mathrm{x}}\bfv(t))\big\}}\big)
            \\&\quad+c\,\tau^{2\beta_{\mathrm{t}}}\,[\bfF(\cdot,\cdot,\bfD_{\mathrm{x}}\bfv)]_{N^{\beta_{\mathrm{t}},2}((0,t_m);L^2(\Omega))}^2
            \\&\quad+c\,h^{2\beta_{\mathrm{x}}}\,[\bfF(\cdot,\cdot,\bfD_{\mathrm{x}}\bfv)]_{L^2((0,t_m);N^{\beta_{\mathrm{x}},2}(\Omega))}^2
            \\&\quad+ c\,\rho_{(\varphi_{\vert \bfD_{\mathrm{x}}\bfv\vert})^*,\smash{Q_T^m}}\big(h^{\gamma_{\mathrm{x}}}\vert\nabla_{\mathrm{x}}^{\gamma_{\mathrm{x}}}q\vert\big)\,.
		\end{aligned}
	\end{align}
	Then, using, in turn, \eqref{thm:main.14} together with \eqref{thm:main.19} in  \eqref{thm:main.13}, for every $m=1,\ldots,M$,  we deduce that\vspace{-0.25mm}
	\begin{align}\label{thm:main.20}
		\begin{aligned}
			\vert I_{m,h}^{1,1}\vert 
            &\leq c_\varepsilon\,(\tau^{2\alpha_{\mathrm{t}}}+h^{2\alpha_{\mathrm{x}}})\,\big(1+{\sup}_{t\in (0,t_m)}{\big\{\rho_{p(t,\cdot)s,\Omega}(\bfD_{\mathrm{x}}\bfv(t))\big\}}\big)
            \\[-0.25mm]&\quad+c_\varepsilon\,\tau^{2\beta_{\mathrm{t}}}\,[\bfF(\cdot,\cdot,\bfD_{\mathrm{x}}\bfv)]_{N^{\beta_{\mathrm{t}},2}((0,t_m);L^2(\Omega))}^2
            \\[-0.25mm]&\quad+c_\varepsilon\,h^{2\beta_{\mathrm{x}}}\,[\bfF(\cdot,\cdot,\bfD_{\mathrm{x}}\bfv)]_{L^2((0,t_m);N^{\beta_{\mathrm{x}},2}(\Omega))}^2
            \\[-0.25mm]&\quad+ c_\varepsilon\,\rho_{(\varphi_{\vert \bfD_{\mathrm{x}}\bfv\vert})^*,\smash{Q_T^m}}\big(h^{\gamma_{\mathrm{x}}}\vert\nabla_{\mathrm{x}}^{\gamma_{\mathrm{x}}}q\vert\big)
			\\[-0.25mm]&\quad+\varepsilon\,c\,\|\bfF_h^{\tau}(\cdot,\cdot,\bfD_{\mathrm{x}}\bfv_h^{\tau})-\bfF_h^{\tau}(\cdot,\cdot,\bfD_{\mathrm{x}}\mathrm{I}_\tau^{0,\mathrm{t}}\bfv)\|_{2,\smash{Q_T^m}}^2\,.
		\end{aligned}
	\end{align}
		
	\textit{ad $I_{m,h}^{1,2}$.} Using the discrete integration-by-parts formula \eqref{eq:4.2} and Young's inequality,~we~observe~that\vspace{-0.75mm}
	\begin{align}\label{thm:main.21}
		\begin{aligned} 
			I_{m,h}^{1,2}
            &=(\mathrm{d}_\tau \bfe_h^\tau,  \bfe_h^\tau)_{\smash{Q_T^m}}+(\mathrm{d}_\tau \bfe_h^\tau, \mathrm{I}_\tau^{0,\mathrm{t}}\bfv-
			\Pi_h^V \mathrm{I}_\tau^{0,\mathrm{t}}\bfv)_{\smash{Q_T^m}}
			\\[-0.25mm]&\ge  \tfrac{1}{2}\|\bfv_h^\tau(t_m)-\bfv(t_m)\|_{2,\Omega}^2-\tfrac{1}{2}\|\bfv_h^0-\bfv_0\|_{2,\Omega}^2+ \tfrac{\tau}{2}\|\mathrm{d}_\tau \bfe_h^\tau\|_{2,\smash{Q_T^m}}^2
            \\[-0.25mm]&\quad- \tfrac{\tau}{2}\|\mathrm{d}_\tau \bfe_h^\tau\|_{2,\smash{Q_T^m}}^2- \tfrac{1}{2\tau}\|\mathrm{I}_\tau^{0,\mathrm{t}}\bfv-
			\Pi_h^V \mathrm{I}_\tau^{0,\mathrm{t}}\bfv\|_{2,\smash{Q_T^m}}^2\,.
		\end{aligned}
	\end{align}
    In the case $p^-\ge 2$, using \cite[Lem.\ B.5]{berselli2023error}, for every $m=1,\ldots,M$, we have that\vspace{-0.25mm}
    \begin{align}\label{thm:main.21.1}
        \begin{aligned}
            \tfrac{1}{\tau}\|\mathrm{I}_\tau^{0,\mathrm{t}}\bfv-
			\Pi_h^V \mathrm{I}_\tau^{0,\mathrm{t}}\bfv\|_{2,\smash{Q_T^m}}^2
            &\leq  c\,\smash{\tfrac{h^2}{\tau}}\, \|\bfD_{\mathrm{x}}\mathrm{I}_\tau^{0,\mathrm{t}}\bfv-
			\bfD_{\mathrm{x}}\Pi_h^V \mathrm{I}_\tau^{0,\mathrm{t}}\bfv\|_{2,\smash{Q_T^m}}^2
            \\[-0.25mm]&\leq  c\, \|\bfF_h^\tau(\cdot,\cdot,\bfD_{\mathrm{x}}\mathrm{I}_\tau^{0,\mathrm{t}}\bfv)-\bfF_h^\tau(\cdot,\cdot,\bfD_{\mathrm{x}}\Pi_h^V \mathrm{I}_\tau^{0,\mathrm{t}}\bfv)\|_{2,\smash{Q_T^m}}^2\,.
        \end{aligned}
    \end{align}
    In the case $p^-\le 2$, using the approximation properties of $\Pi_h^V$, for every $m=1,\ldots,M$,  we have that\vspace{-0.25mm}
    \begin{align}\label{thm:main.21.2}
        \begin{aligned}
            \tfrac{1}{\tau}\|\mathrm{I}_\tau^{0,\mathrm{t}}\bfv-
			\Pi_h^V \mathrm{I}_\tau^{0,\mathrm{t}}\bfv\|_{2,\smash{Q_T^m}}^2&
            \leq  c\,\smash{\tfrac{h^{2+2\beta_{\mathrm{x}}}}{\tau}}\, [\mathrm{I}_\tau^{0,\mathrm{t}}\bfv]_{L^2((0,t_m);N^{1+\beta_{\mathrm{x}},2}(\Omega))}^2
            \\[-0.25mm]&\leq c\,\smash{h^{2\beta_{\mathrm{x}}}}\, [ \bfv]_{L^\infty((0,t_m);N^{1+\beta_{\mathrm{x}},2}(\Omega))}^2\,.
        \end{aligned}
    \end{align}
	In summary, combining \eqref{thm:main.20} and \eqref{thm:main.21} together with \eqref{thm:main.21.1} or \eqref{thm:main.21.2} in \eqref{thm:main.12} and~using~the~approxi-mation properties of $\smash{\Pi_h^{V,L^2}}\!\!$ 
    together with $\bfv_0\hspace{-0.15em}\in\hspace{-0.15em} \smash{(N^{\beta_{\mathrm{x}},2}(\Omega))^d}$ (\textit{cf}.\ Remark \ref{rem:reg_initial}), for every $m\hspace{-0.15em}=\hspace{-0.15em}1,\ldots,M$,~we~get\vspace{-0.25mm}\enlargethispage{12mm}
	\begin{align}\label{thm:main.22}
		\begin{aligned} 
			I_{m,h}^1
            &\ge  \tfrac{1}{2}\|\bfv_h^\tau(t_m)-\bfv(t_m)\|_{2,\Omega}^2-\tfrac{1}{2}\|\bfv_h^0-\bfv_0\|_{2,\Omega}^2
			\\[-0.25mm]&\quad-c_\varepsilon\,(\tau^{2\alpha_{\mathrm{t}}}+h^{2\alpha_{\mathrm{x}}})\,\big(1+{\sup}_{t\in (0,t_m)}{\big\{\rho_{p(t,\cdot)s,\Omega}(\bfD_{\mathrm{x}}\bfv(t))\big\}}\big)
			\\[-0.25mm]&\quad-c_\varepsilon\,\tau^{2\beta_{\mathrm{t}}}\,[\bfF(\cdot,\cdot,\bfD_{\mathrm{x}}\bfv)]_{N^{\beta_{\mathrm{t}},2}((0,t_m);L^2(\Omega))}^2
            \\[-0.25mm]&\quad-c_\varepsilon\,h^{2\beta_{\mathrm{x}}}\,\big([\bfF(\cdot,\cdot,\bfD_{\mathrm{x}}\bfv)]_{L^2((0,t_m);N^{\beta_{\mathrm{x}},2}(\Omega))}^2+\epsilon^2(\bfv)\big)
			\\[-0.25mm]&\quad-c_\varepsilon\,\rho_{(\varphi_{\vert \bfD_{\mathrm{x}}\bfv\vert})^*,Q_T^m}\big(h^{\gamma_{\mathrm{x}}}\vert\nabla_{\mathrm{x}}^{\gamma_{\mathrm{x}}}q\vert\big)
			\\[-0.25mm]&\quad-\varepsilon\,c\,\|\bfF_h^{\tau}(\cdot,\cdot,\bfD_{\mathrm{x}}\bfv_h^{\tau})-\bfF_h^{\tau}(\cdot,\cdot,\bfD_{\mathrm{x}}\mathrm{I}_\tau^{0,\mathrm{t}}\bfv)\|_{2,\smash{Q_T^m}}^2\,.
		\end{aligned}
	\end{align}
	Using \eqref{thm:main.9}--\eqref{thm:main.11} and \eqref{thm:main.22} in \eqref{thm:main.2}, for $\varepsilon>0$~sufficiently~small, for every $m=1,\ldots,M$,~we~arrive~at\vspace{-0.25mm}
	\begin{align}\label{thm:main.23}
		\begin{aligned} 
		  \|\bfv_h^\tau(t_m)-\bfv(t_m)\|_{2,\Omega}^2
            &+\|\bfF_h^{\tau}(\cdot,\cdot,\bfD_{\mathrm{x}}\bfv_h^{\tau})-\bfF_h^{\tau}(\cdot,\cdot,\bfD_{\mathrm{x}}\mathrm{I}_\tau^{0,\mathrm{t}}\bfv)\|_{2,\smash{Q_T^m}}^2
			\\[-0.25mm]&\leq 
            c\,(\tau^{2\alpha_{\mathrm{t}}}+h^{2\alpha_{\mathrm{x}}})\,\big(1+{\sup}_{t\in (0,t_m)}{\big\{\rho_{p(t,\cdot)s,\Omega}(\bfD_{\mathrm{x}}\bfv(t))\big\}}\big)
			\\[-0.25mm]&\quad+c\,\tau^{2\beta_{\mathrm{t}}}\,[\bfF(\cdot,\cdot,\bfD_{\mathrm{x}}\bfv)]_{N^{\beta_{\mathrm{t}},2}((0,t_m);L^2(\Omega))}^2
            \\[-0.25mm]&\quad+c\,h^{2\beta_{\mathrm{x}}}\,\big([\bfF(\cdot,\cdot,\bfD_{\mathrm{x}}\bfv)]_{L^2((0,t_m);N^{\beta_{\mathrm{x}},2}(\Omega))}^2+\epsilon^2(\bfv)\big)
			\\[-0.25mm]&\quad+c\,\rho_{(\varphi_{\vert \bfD_{\mathrm{x}}\bfv\vert})^*,Q_T^m}\big(h^{\gamma_{\mathrm{x}}}\vert\nabla_{\mathrm{x}}^{\gamma_{\mathrm{x}}}q\vert\big)\,.
		\end{aligned}
	\end{align}
    Eventually, from  \eqref{thm:main.23}, 
    we conclude that the claimed  
	\textit{a priori} error estimate applies.
\end{proof}
	
\begin{proof}[Proof (of Corollary \ref{cor:main}).]
		
	\textit{ad \eqref{cor:main.1}.}
	Using $(\varphi_a)^*(t,x,hr)\lesssim
	h^{\smash{\min\{2,p'(t,x)\}}} (\varphi_a)^*(t,x,r)$~for~all~${a,r\ge 0}$, $h\in (0,1]$, and $(t,x)^\top\in Q_T$,  we deduce that\vspace{-0.75mm}
	\begin{align*}
		\smash{\rho_{(\varphi_{\vert \bfD_{\mathrm{x}}\bfv\vert})^*,Q_T}\big(h^{\gamma_{\mathrm{x}}}\vert\nabla_{\mathrm{x}}^{\gamma_{\mathrm{x}}}q\vert\big) \lesssim h^{\smash{\min\{2,(p^+)'\}\gamma_{\mathrm{x}}}}\rho_{(\varphi_{\vert \bfD_{\mathrm{x}}\bfv\vert})^*,Q_T}(\vert\nabla_{\mathrm{x}}^{\gamma_{\mathrm{x}}}q\vert) \,,}
	\end{align*}
	so that from Theorem \ref{thm:main}, it follows the claimed \textit{a priori} error estimate \eqref{cor:main.1}.
		
	\textit{ad \eqref{cor:main.2}.} 
	Using $(\varphi_a)^*(t,x,hr)\sim ( (\delta+a)^{p(t,x)-1}+hr\big )^{p'(t,x)-2}h^2 r^2	\le (\delta+a)^{2-p(t,x)} h^2 r^2 $~for~all~${a,r\ge 0}$, $(t,x)^\top\in Q_T$, and $h\in (0,1]$ (\textit{cf}.\ \eqref{rem:phi_a.2}), due to $p^-\ge 2$, we deduce  that\vspace{-0.25mm}
	\begin{align*}
		\smash{\rho_{(\varphi_{\vert \bfD_{\mathrm{x}}\bfv\vert})^*,Q_T}\big(h^{\gamma_{\mathrm{x}}}\vert\nabla_{\mathrm{x}}^{\gamma_{\mathrm{x}}}q\vert\big) \lesssim h^{\smash{2\gamma_{\mathrm{x}}}}\,
		\|(\delta+\vert \bfD_{\mathrm{x}}\bfv\vert )^{2-p(\cdot,\cdot)}\vert \nabla_{\mathrm{x}}^{\gamma_{\mathrm{x}}} q\vert^2\|_{1,Q_T}\,,}
	\end{align*}
	so that from Theorem \ref{thm:main}, it follows the claimed \textit{a priori} error estimate \eqref{cor:main.2}.
\end{proof} 
 
\section{Numerical Experiments}\label{sec:exp} 
 
 \hspace*{5mm}In this section, we review the error decay rates derived in Corollary \ref{cor:main} for optimality.\vspace{-2.5mm}\enlargethispage{11mm} 

 \subsection{Implementation details}\vspace{-0.75mm}

\hspace{5mm}All numerical experiments were carried out using the finite element software \texttt{FEniCS} (version~2019.1.0, \textit{cf}.~\cite{LW10}). To keep the computational costs moderate, we restrict to the case $d=2$.  
	As quadrature points of the one-point~quadrature~rule~used~to discretize the power-law index
	we employ barycenters~of~elements, \textit{i.e.}, we employ $\xi_K \coloneqq  \frac{1}{3}\sum_{\nu \in \mathcal{N}_h\cap K}{\nu}\in K$ for all $K\in  \mathcal{T}_h$, where $\mathcal{N}_h$ denotes the set of vertices~of~$\mathcal{T}_h$. As a conforming, discretely inf-sup stable FE couple, we consider the Taylor--Hood element (\textit{cf}.\ \cite{TH73}), \textit{i.e.}, we choose $V_h\coloneqq (\mathbb{P}^2_c(\mathcal{T}_h))^2$ and $Q_h\coloneqq \mathbb{P}^1_c(\mathcal{T}_h)$.
	
	We approximate the discrete solution $(\bfv_h^\tau,q_h ^\tau)^{\top}\in  \mathbb{P}^0(\mathcal{I}_h^0;\Vo_{h,0})\times \mathbb{P}^0(\mathcal{I}_h;\Qo_h)$ of Problem~(\hyperlink{Qh}{Q$_h^\tau$}) iteratively employing the Newton solver from \mbox{\texttt{PETSc}} (version~3.17.3, \textit{cf}.~\cite{LW10}), with an absolute tolerance of $\tau_{abs}= 1.0\times10^{-8}$ and a relative tolerance of $\tau_{rel}=1.0\times10^{-8}$.~The~linear~system emerging~in~each~Newton iteration is solved using a sparse direct solver from \texttt{MUMPS} (version~5.5.0,~\textit{cf}.~\cite{mumps}).~In~the~\mbox{implementation}, the zero mean condition included in the discrete pressure space $\Qo_h$
is enforced via adding one dimension to the linear system emerging~in~each~Newton iteration.

\subsection{Experimental set-up}\vspace{-0.75mm}

\hspace{5mm}We consider an analogous experimental set-up to \cite[Sec.\ 9.3]{berselli2024convergence}:
\begin{itemize}[noitemsep,topsep=2pt]
    \item[$\bullet$] \textit{Power-law index}. For $p^-\hspace{-0.1em}\in\hspace{-0.1em}  \{1.5,1.75,2.0,2.25,2.5\}$, $p^+\hspace{-0.1em}\coloneqq\hspace{-0.1em} p^-+1$, and ${\alpha\hspace{-0.1em}\coloneqq\hspace{-0.1em} \alpha_{\mathrm{t}}\hspace{-0.1em}=\hspace{-0.1em}\alpha_{\mathrm{x}}\hspace{-0.1em}\in\hspace{-0.1em}  \{1.0,0.75,0.5\}}$, 
    the power-law index $ p\in C^{0,\alpha,\alpha}(\overline{Q_T})$, where $T=0.1$ and $\Omega =(0,1)^2$,  for every $(t,x)^\top\in \overline{Q_T}$, is defined by 
\begin{align*}
	p(t,x)\coloneqq \big(1-\tfrac{\vert x\vert^{\alpha}}{2^{\alpha/2}}\big)\, p^++\tfrac{\vert x\vert^{\alpha}}{2^{\alpha/2}}\,(p^-+t)\,.
\end{align*} 

    \item[$\bullet$] \textit{Manufactured velocity vector field}. For $\beta=\beta_{\mathrm{t}}=\beta_{\mathrm{x}}\in \{1.0,0.75,0.5\}$, the  velocity vector field $\mathbf{v}\colon Q_T\to \mathbb{R}^2$, 
for every $(t,x)^{\top}=(t,x_1,x_2)^\top\in Q_T$, is defined by
\begin{align}\label{solutions.v}
\begin{aligned}
\mathbf{v}(t,x)&\coloneqq0.1\, t\,\vert x\vert^{\rho_{\mathbf{v}}(t,x)} (x_2,-x_1)^\top\,,\\
\text{ where }\quad\rho_{\mathbf{v}}(t,x)&\coloneqq 2\tfrac{\beta-1}{p(t,x)}+\delta\,.
\end{aligned}
\end{align} 
Then, we have that
\begin{align}\label{regurarity_v}
	\begin{aligned}
	\mathbf{F}(\cdot,\cdot,\mathbf{D}_{\mathrm{x}}\mathbf{v})&\in L^2(I;(N^{\beta,2}(\Omega))^{2\times 2})\cap N^{\beta,2}(I;(L^2(\Omega))^{2\times 2})\,,\\ 
	\mathbf{v}&\in L^\infty(I;(N^{\beta,2}(\Omega))^2)\,.
\end{aligned}
\end{align}
Note \hspace{-0.1mm}that \hspace{-0.1mm}for \hspace{-0.1mm}$p^-\hspace{-0.175em}\in\hspace{-0.175em} \{1.5,1.75\}$, \hspace{-0.1mm}we \hspace{-0.1mm}have \hspace{-0.1mm}that \hspace{-0.1mm}$\bfv\hspace{-0.175em}\notin\hspace{-0.175em} L^\infty(I;(N^{1+\beta,2}(\Omega))^2)$, \hspace{-0.1mm}so \hspace{-0.1mm}that~\hspace{-0.1mm}with~\hspace{-0.1mm}the~\hspace{-0.1mm}\mbox{manufactured} velocity vector field defined by \eqref{solutions.v} we are only in the position to review the optimality of the error decay rates derived in Corollary \ref{cor:main} in the case $p^-\ge 2$.

    \item[$\bullet$] \textit{Manufactured kinematic pressure}. For $\gamma=\gamma_{\mathrm{x}}\in \{1.0,0.75,0.5\}$, the kinematic pressure $q \colon Q_T\to \mathbb{R}$, 
for every $(t,x)^{\top}\in Q_T$, is defined by
\begin{align}\label{solutions.q}
\begin{aligned} 
q(t,x)&\coloneqq 100\,t\,(\vert x\vert^{\rho_{q}(t,x)}-\langle\,\vert \cdot\vert^{\rho_{q}(t,\cdot)}\,\rangle_\Omega)\,, \\
\text{ where }\quad\rho_{q}(t,x)&\coloneqq
\begin{cases}
    \gamma-\frac{2}{p'(t,x)} +\delta &\text{(Case \hypertarget{C1}{1})}\,,\\
    \rho_{\mathbf{v}}(t,x)\frac{p(t,x)-2}{2} +\gamma +0.01&\text{(Case \hypertarget{C2}{2})}\,.
\end{cases} 
\end{aligned}
\end{align}
Then, we have that $q\in L^{p'(\cdot,\cdot)}(Q_T)$ and $q\in C^{\gamma,p'(t,\cdot)}(\Omega)$ for a.e.\ $t\in I$ with 
\begin{align}\label{regurarity_q}
    \begin{cases}
        \vert \nabla_{\mathrm{x}}^{\gamma} q\vert  \in L^{p'(\cdot,\cdot)}(Q_T)&\text{(Case \hyperlink{C1}{1})}\,,\\
        (\delta+\vert\bfD_{\mathrm{x}}\bfv\vert)^{2-p(t,x)}\vert \nabla_{\mathrm{x}}^{\gamma} q\vert^2\in L^1(Q_T)&\text{(Case \hyperlink{C2}{2})}\,.
    \end{cases}
\end{align}

\item[$\bullet$] \textit{Discretization of time-space cylinder}.
We construct triangulations $\mathcal T_{h_n}$, $n=0,\ldots,7$, where $h_{n+1}=\frac{h_n}{2}$ for all $n=0,\ldots,7$, using uniform mesh-refinement starting from 
 an initial triangulation~$\mathcal
T_{h_0}$, where $h_0=1$, obtained 
 by subdividing the domain $\Omega=(0,1)^2$ along its diagonals into four triangles with different orientations.
 Moreover, we employ the~time~step~sizes~$\tau_n\coloneqq 2^{-n-2}$,~$n=0,\ldots,7$. 

 \item[$\bullet$] \textit{Discretization \hspace{-0.1mm}of \hspace{-0.1mm}right-hand \hspace{-0.1mm}side}.
\hspace{-0.1mm}As \hspace{-0.1mm}the \hspace{-0.1mm}manufactured \hspace{-0.1mm}solutions \hspace{-0.1mm}\eqref{solutions.v} \hspace{-0.1mm}and \hspace{-0.1mm}\eqref{solutions.q}~\hspace{-0.1mm}are~\hspace{-0.1mm}smooth~\hspace{-0.1mm}in~\hspace{-0.1mm}time and  \hspace{-0.1mm}as 
\hspace{-0.1mm}\eqref{regurarity_v} \hspace{-0.1mm}and 
\hspace{-0.1mm}\eqref{regurarity_q} \hspace{-0.1mm}imply \hspace{-0.1mm}higher \hspace{-0.1mm}integrability \hspace{-0.1mm}of \hspace{-0.1mm}the \hspace{-0.1mm}right-hand \hspace{-0.1mm}side~\hspace{-0.1mm}in~\hspace{-0.1mm}Problem~\hspace{-0.1mm}(\hyperlink{Qh}{Q$_{h_n}^{\tau_n}$}),~\hspace{-0.1mm}we~\hspace{-0.1mm}have~\hspace{-0.1mm}that
 $\partial_{\mathrm{t}}\mathbf{g}\in  (L^{(p^-)'}(Q_T))^2$ and $\mathbf{G},\partial_{\mathrm{t}}\mathbf{G}\in  (L^{p'(\cdot,\cdot)+\eta}(Q_T))^{2\times 2}$ for a certain $\eta>0$. Therefore, for a simple implementation, \hspace{-0.1mm}in \hspace{-0.1mm}Problem~\hspace{-0.1mm}(\hyperlink{Qh}{Q$_{h_n}^{\tau_n}$}), \hspace{-0.1mm}we \hspace{-0.1mm}replace \hspace{-0.1mm}$\mathbf{g}$, \hspace{-0.1mm}$\mathbf{G}$ \hspace{-0.1mm}by \hspace{-0.1mm}$\mathrm{I}_{\tau_n}^{0,\mathrm{t}}\mathbf{g}$, \hspace{-0.1mm}$\mathrm{I}_{\tau_n}^{0,\mathrm{t}}\mathbf{G}$, \hspace{-0.1mm}respectively.~\hspace{-0.1mm}Then,~\hspace{-0.1mm}we~\hspace{-0.1mm}have~\hspace{-0.1mm}that\vspace{-0.5mm}
\begin{align*}
	\rho_{(p^-)',Q_T}(\bfg-\mathrm{I}_{\tau_n}^{0,\mathrm{t}}\bfg) 
		&	\leq  \smash{\tau_n^{(p^-)'}} \, 	\rho_{(p^-)',Q_T}(\partial_{\mathrm{t}}  \boldsymbol{f})\,,\\ 
		\rho_{p'_h(\cdot,\cdot),Q_T}
		(\bfG-\mathrm{I}_{\tau_n}^{0,\mathrm{t}}\bfG) 
		&\leq 
	\smash{\tau_n^{(p^+)'}}\,\big(1 +\rho_{p'(\cdot,\cdot)+\eta,Q_T}
		(\partial_{\mathrm{t}} \bfG)\big) \,. 
\end{align*} 
so that the error decay rates derived in Corollary \ref{cor:main}~are~not~affected.\enlargethispage{10mm}
\end{itemize}

Then, for $\alpha=\beta=\gamma\in \{1.0,0.75,0.5\}$,  we compute $(\bfv_{h_n}^{\tau_n},q_{h_n}^{\tau_n})^\top\in \mathbb{P}^0(\mathcal{I}_{\tau_n}^0;\Vo_{h_n,0})\times \mathbb{P}^0(\mathcal{I}_{\tau_n};\Qo_{h_n,0})$ solving Problem (\hyperlink{Qh}{Q$_{h_n}^{\tau_n}$}) all $n=0,\ldots,7$,
and, for every $n=0,\ldots,7$, the  error quantities
\begin{align}\label{eq:errors}
	\begin{aligned}
		e_{\mathbf{F},n}&\coloneqq \|\mathbf{F}_{h_n}^{\tau_n}(\cdot,\cdot,\mathbf{D}_{\mathrm{x}}\bfv_{h_n}^{\tau_n})-\mathbf{F}_{h_n}^{\tau_n}(\cdot,\cdot,\mathbf{D}_{\mathrm{x}}\mathrm{I}_{\tau_n}^{0,\mathrm{t}}\mathbf{v})\|_{2,Q_T}\,,\\ 
			e_{\mathbf{F}^*,n}&\coloneqq \|(\mathbf{F}_{h_n}^{\tau_n})^*(\cdot,\cdot,\mathbf{S}_{h_n}^{\tau_n}(\cdot,\cdot,\mathbf{D}_{\mathrm{x}}\bfv_{h_n}^{\tau_n}))-(\mathbf{F}_{h_n}^{\tau_n})^*(\cdot,\cdot,\mathbf{S}_{h_n}^{\tau_n}(\cdot,\cdot,\mathbf{D}_{\mathrm{x}}\mathrm{I}_{\tau_n}^{0,\mathrm{t}}\mathbf{v}))\|_{2,Q_T}\,,\\ 
		e_{\varphi^*,n}&\coloneqq  \rho_{((\varphi_{h_n}^{\tau_n})_{\smash{\vert \mathbf{D}_{\mathrm{x}}\mathrm{I}_{\tau_n}^{0,\mathrm{t}}\mathbf{v}\vert}})^*}(q_{h_n}^{\tau_n}-\mathrm{I}_{\tau_n}^{0,\mathrm{t}}q)\,,\\[-0.5mm]
        e_{L^2,n}&\coloneqq \|\bfv_{h_n}^{\tau_n}-\mathrm{I}_{\tau_n}^{0,\mathrm{t}}\mathbf{v}\|_{L^\infty(I;L^2(\Omega))} \,.
	\end{aligned}
\end{align} 

In order to measure convergence rates,  we compute experimental order of convergence~(EOC),~\textit{i.e.}, 
\begin{align} \label{eq:eoc}
	\texttt{EOC}_n(e_n)\coloneqq\frac{\log\big(\frac{e_n}{e_{n-1}}\big)}{\log\big(\frac{h_n+\tau_n}{h_{n-1}+\tau_{n-1}}\big)}\,, \quad n=1,\ldots,7\,,
\end{align}
where, for every $n= 0,\ldots,7$, we denote by $e_n$
either 
$e_{\mathbf{F},n}$, $e_{\mathbf{F}^*,n}$,  $e_{\varphi^*,n}$,  or $e_{L^2,n}$, respectively. 

In view of Corollary \ref{cor:main} and the manufactured regularity properties \eqref{regurarity_v} and \eqref{regurarity_q}, in~the~case~${p^-\hspace{-0.1em}\ge\hspace{-0.1em} 2}$,
we expect the error decay rate $	\texttt{EOC}_n(e_n)=\min\big\{1,\smash{\frac{(p^+)'}{2}}\big\}\alpha$, $n= 1,\ldots,7$, in Case \hyperlink{C1}{1} and $	\texttt{EOC}_n(e_n)=\alpha$, $n= 1,\ldots,7$, in Case \hyperlink{C2}{2} for the error quantities $e_n\in \{e_{\mathbf{F},n},e_{L^2,n}\}$,~${n=0,\ldots,7}$,~(\textit{cf}.~\eqref{eq:errors}).~Even~if~not~covered in Corollary \ref{cor:main}, for the error quantities $e_n\in \{e_{\mathbf{F}^*,n},e_{\varphi^*,n}\}$,~${n=0,\ldots,7}$,~(\textit{cf}.~\eqref{eq:errors}), we expect the same error decay rates.

For different values of $p^- \in \{1.5, 1.75,2.0,2.25, 2.5\}$, fractional exponents $\alpha =\beta=\gamma\in \{1.0,0.75,0.5\}$, and   triangulations $\mathcal{T}_{h_n}$,
$n = 0, \ldots , 7$, obtained as described above, the EOCs (\textit{cf}.\ \eqref{eq:eoc}) with respect to the error quantities \eqref{eq:errors} are
computed and presented in the~Tables~\mbox{\ref{tab:1}--\ref{tab:4}}: in the Tables~\mbox{\ref{tab:1}--\ref{tab:3}},~for~$e_n\in \{e_{\mathbf{F},n}, e_{\mathbf{F}^*,n},e_{\varphi^*,n}\}$, $n=0,\ldots,7$,   we report
the expected error decay rate of $\texttt{EOC}_n(e_n) \approx \min\big\{1,\smash{\frac{(p^+)'}{2}}\big\}\alpha $, ${n=1,\ldots,7}$, in Case \hyperlink{C1}{1} and $\texttt{EOC}_n(e_n) \approx \alpha $, ${n=1,\ldots,7}$, in Case \hyperlink{C2}{2};  in Table \ref{tab:4}, for   $e_{L^2,n}$, $n=0,\ldots,7$, we report the increased error decay~rate~${\texttt{EOC}_n(e_{L^2,n}) \gtrapprox  1}$,~${n=1,\ldots,7}$. In~summary, in the case $p^-\ge 2$, this confirms the optimality of the derived error decay rates in Corollary \ref{cor:main}.\vspace{-0.5mm}

\begin{table}[H]
	\setlength\tabcolsep{1pt}
	\centering
	\begin{tabular}{c|c|c|c|c|c|c|c|c|c|c|c|c|c|c|c|} \toprule 
		\multicolumn{1}{|c||}{\cellcolor{lightgray}$\alpha$}	
		& \multicolumn{5}{c||}{\cellcolor{lightgray}$1.0$}   & \multicolumn{5}{c||}{\cellcolor{lightgray}$0.75$} & \multicolumn{5}{c|}{\cellcolor{lightgray}$0.50$} \\ 
		\toprule\toprule
		
		\multicolumn{1}{|c||}{\cellcolor{lightgray}\diagbox[height=1.1\line,width=0.15\dimexpr\linewidth]{\vspace{-0.65mm}$n$}{\\[-4.5mm] $p^-$}}
		& \cellcolor{lightgray}1.5 & \cellcolor{lightgray}1.75  & \cellcolor{lightgray}2.0 & \cellcolor{lightgray}2.25 & \multicolumn{1}{c||}{\cellcolor{lightgray}2.5}
		& \cellcolor{lightgray}1.5 & \cellcolor{lightgray}1.75  & \cellcolor{lightgray}2.0 & \cellcolor{lightgray}2.25 & \multicolumn{1}{c||}{\cellcolor{lightgray}2.5}
		& \cellcolor{lightgray}1.5 & \cellcolor{lightgray}1.75  & \cellcolor{lightgray}2.0 & \cellcolor{lightgray}2.25 & \multicolumn{1}{c|}{\cellcolor{lightgray}2.5} 
        \\ \toprule\toprule 
        \multicolumn{16}{|c|}{\cellcolor{lightgray}Case \hyperlink{C1}{1}}\\ \toprule\toprule 
		\multicolumn{1}{|c||}{\cellcolor{lightgray}$4$} & 0.807 & 0.765 & 0.734 & 0.709 & \multicolumn{1}{c||}{0.689} & 0.614 & 0.587 & 0.565 & 0.547 & \multicolumn{1}{c||}{0.532} & 0.410 & 0.397 & 0.384 & 0.373 & 0.363 \\ \hline
		\multicolumn{1}{|c||}{\cellcolor{lightgray}$5$} & 0.830 & 0.782 & 0.747 & 0.720 & \multicolumn{1}{c||}{0.699} & 0.627 & 0.594 & 0.568 & 0.548 & \multicolumn{1}{c||}{0.532} & 0.422 & 0.401 & 0.385 & 0.372 & 0.361 \\ \hline
		\multicolumn{1}{|c||}{\cellcolor{lightgray}$6$} & 0.836 & 0.787 & 0.751 & 0.723 & \multicolumn{1}{c||}{0.701} & 0.629 & 0.594 & 0.568 & 0.547 & \multicolumn{1}{c||}{0.530} & 0.424 & 0.401 & 0.384 & 0.370 & 0.358 \\ \hline
		\multicolumn{1}{|c||}{\cellcolor{lightgray}$7$} & 0.837 & 0.788 & 0.751 & 0.723 & \multicolumn{1}{c||}{0.701} & 0.629 & 0.594 & 0.566 & 0.545 & \multicolumn{1}{c||}{0.528} & 0.424 & 0.400 & 0.382 & 0.368 & 0.356 \\ \hline\hline  
		\multicolumn{1}{|c||}{\cellcolor{lightgray}\small $\min\{1,\hspace{-1mm}\frac{(p^+)'}{2}\}\alpha$} & \cellcolor{green!25!white}0.833 & \cellcolor{green!25!white}0.786 & \cellcolor{green!25!white}0.750 & \cellcolor{green!25!white}0.722 & \multicolumn{1}{c||}{\cellcolor{green!25!white}0.700} & \cellcolor{green!25!white}0.625 &\cellcolor{green!25!white} 0.589 & \cellcolor{green!25!white}0.563 & \cellcolor{green!25!white}0.541 & \multicolumn{1}{c||}{\cellcolor{green!25!white}0.525} & \cellcolor{green!25!white}0.417 & \cellcolor{green!25!white}0.393 & \cellcolor{green!25!white}0.375 & \cellcolor{green!25!white}0.361 & \cellcolor{green!25!white}0.350 \\ \toprule\toprule
		\multicolumn{16}{|c|}{\cellcolor{lightgray}Case \hyperlink{C2}{2}}\\ \toprule\toprule 
		\multicolumn{1}{|c||}{\cellcolor{lightgray}$4$} &  ---  &  ---  & 0.784 & 0.725 & \multicolumn{1}{c||}{0.659} &  ---  &  ---  & 0.731 & 0.732 & \multicolumn{1}{c||}{0.732} &  ---  &  ---  & 0.511 & 0.513 & 0.513 \\ \hline
		\multicolumn{1}{|c||}{\cellcolor{lightgray}$5$} &  ---  &  ---  & 0.912 & 0.892 & \multicolumn{1}{c||}{0.864} &  ---  &  ---  & 0.747 & 0.748 & \multicolumn{1}{c||}{0.747} &  ---  &  ---  & 0.515 & 0.515 & 0.514 \\ \hline
		\multicolumn{1}{|c||}{\cellcolor{lightgray}$6$} &  ---  &  ---  & 0.961 & 0.954 & \multicolumn{1}{c||}{0.944} &  ---  &  ---  & 0.753 & 0.753 & \multicolumn{1}{c||}{0.753} &  ---  &  ---  & 0.515 & 0.514 & 0.513 \\ \hline
		\multicolumn{1}{|c||}{\cellcolor{lightgray}$7$} &  ---  &  ---  & 0.983 & 0.980 & \multicolumn{1}{c||}{0.975} &  ---  &  ---  & 0.755 & 0.755 & \multicolumn{1}{c||}{0.755} &  ---  &  ---  & 0.514 & 0.513 & 0.512 \\ \toprule\toprule
		\multicolumn{1}{|c||}{\cellcolor{lightgray}\small $\alpha$} &  ---  &  ---  & \cellcolor{green!25!white}1.000 & \cellcolor{green!25!white}1.000 & \multicolumn{1}{c||}{\cellcolor{green!25!white}1.000} &  ---  &  ---  & \cellcolor{green!25!white}0.750 & \cellcolor{green!25!white}0.750 & \multicolumn{1}{c||}{\cellcolor{green!25!white}0.750} &  ---  &  ---  & \cellcolor{green!25!white}0.500 & \cellcolor{green!25!white}0.500 & \cellcolor{green!25!white}0.500 \\ \toprule
	\end{tabular}\vspace{-2mm}
	\caption{Experimental order of convergence: $\texttt{EOC}_n(e_{\mathbf{F},n})$,~${n=4,\dots,7}$.}
	\label{tab:1} 
	\end{table}

	\begin{table}[H]
		\setlength\tabcolsep{1pt}
		\centering
		\begin{tabular}{c|c|c|c|c|c|c|c|c|c|c|c|c|c|c|c|} \toprule 
			\multicolumn{1}{|c||}{\cellcolor{lightgray}$\alpha$}	
			& \multicolumn{5}{c||}{\cellcolor{lightgray}$1.0$}   & \multicolumn{5}{c||}{\cellcolor{lightgray}$0.75$} & \multicolumn{5}{c|}{\cellcolor{lightgray}$0.50$} \\ 
			\toprule\toprule
			
			\multicolumn{1}{|c||}{\cellcolor{lightgray}\diagbox[height=1.1\line,width=0.15\dimexpr\linewidth]{\vspace{-0.65mm}$n$}{\\[-4.5mm] $p^-$}}
			& \cellcolor{lightgray}1.5 & \cellcolor{lightgray}1.75  & \cellcolor{lightgray}2.0 & \cellcolor{lightgray}2.25 & \multicolumn{1}{c||}{\cellcolor{lightgray}2.5}
			& \cellcolor{lightgray}1.5 & \cellcolor{lightgray}1.75  & \cellcolor{lightgray}2.0 & \cellcolor{lightgray}2.25 & \multicolumn{1}{c||}{\cellcolor{lightgray}2.5}
			& \cellcolor{lightgray}1.5 & \cellcolor{lightgray}1.75  & \cellcolor{lightgray}2.0 & \cellcolor{lightgray}2.25 & \multicolumn{1}{c|}{\cellcolor{lightgray}2.5} 
			\\ \toprule\toprule 
			\multicolumn{16}{|c|}{\cellcolor{lightgray}Case \hyperlink{C1}{1}}\\ \toprule\toprule 
			\multicolumn{1}{|c||}{\cellcolor{lightgray}$4$} & 0.795 & 0.757 & 0.728 & 0.704 & \multicolumn{1}{c||}{0.685} & 0.597 & 0.575 & 0.556 & 0.540 & \multicolumn{1}{c||}{0.526} & 0.385 & 0.380 & 0.372 & 0.363 & 0.356 \\ \hline
			\multicolumn{1}{|c||}{\cellcolor{lightgray}$5$} & 0.824 & 0.778 & 0.744 & 0.718 & \multicolumn{1}{c||}{0.697} & 0.617 & 0.587 & 0.563 & 0.544 & \multicolumn{1}{c||}{0.529} & 0.405 & 0.390 & 0.377 & 0.365 & 0.355 \\ \hline
			\multicolumn{1}{|c||}{\cellcolor{lightgray}$6$} & 0.833 & 0.784 & 0.749 & 0.722 & \multicolumn{1}{c||}{0.700} & 0.624 & 0.590 & 0.565 & 0.544 & \multicolumn{1}{c||}{0.528} & 0.413 & 0.394 & 0.378 & 0.365 & 0.355 \\ \hline
			\multicolumn{1}{|c||}{\cellcolor{lightgray}$7$} & 0.836 & 0.786 & 0.750 & 0.723 & \multicolumn{1}{c||}{0.700} & 0.626 & 0.591 & 0.565 & 0.544 & \multicolumn{1}{c||}{0.527} & 0.417 & 0.395 & 0.378 & 0.365 & 0.354 \\ \hline\hline  
			\multicolumn{1}{|c||}{\cellcolor{lightgray}\small $\min\{1,\hspace{-1mm}\frac{(p^+)'}{2}\}\alpha$} & \cellcolor{green!25!white}0.833 & \cellcolor{green!25!white}0.786 & \cellcolor{green!25!white}0.750 & \cellcolor{green!25!white}0.722 & \multicolumn{1}{c||}{\cellcolor{green!25!white}0.700} & \cellcolor{green!25!white}0.625 &\cellcolor{green!25!white} 0.589 & \cellcolor{green!25!white}0.563 & \cellcolor{green!25!white}0.541 & \multicolumn{1}{c||}{\cellcolor{green!25!white}0.525} & \cellcolor{green!25!white}0.417 & \cellcolor{green!25!white}0.393 & \cellcolor{green!25!white}0.375 & \cellcolor{green!25!white}0.361 & \cellcolor{green!25!white}0.350 \\ \toprule\toprule
			\multicolumn{1}{|c||}{\cellcolor{lightgray}$4$} &  ---  &  ---  & 0.777 & 0.721 & \multicolumn{1}{c||}{0.656} &  ---  &  ---  & 0.722 & 0.725 & \multicolumn{1}{c||}{0.726} &  ---  &  ---  & 0.499 & 0.503 & 0.506 \\ \hline
			\multicolumn{1}{|c||}{\cellcolor{lightgray}$5$} &  ---  &  ---  & 0.907 & 0.889 & \multicolumn{1}{c||}{0.862} &  ---  &  ---  & 0.742 & 0.743 & \multicolumn{1}{c||}{0.744} &  ---  &  ---  & 0.506 & 0.508 & 0.509 \\ \hline
			\multicolumn{1}{|c||}{\cellcolor{lightgray}$6$} &  ---  &  ---  & 0.959 & 0.952 & \multicolumn{1}{c||}{0.942} &  ---  &  ---  & 0.750 & 0.750 & \multicolumn{1}{c||}{0.750} &  ---  &  ---  & 0.509 & 0.509 & 0.510 \\ \hline
			\multicolumn{1}{|c||}{\cellcolor{lightgray}$7$} &  ---  &  ---  & 0.981 & 0.979 & \multicolumn{1}{c||}{0.974} &  ---  &  ---  & 0.753 & 0.753 & \multicolumn{1}{c||}{0.753} &  ---  &  ---  & 0.510 & 0.510 & 0.510 \\ \toprule\toprule
			\multicolumn{1}{|c||}{\cellcolor{lightgray}\small $\alpha$} &  ---  &  ---  & \cellcolor{green!25!white}1.000 & \cellcolor{green!25!white}1.000 & \multicolumn{1}{c||}{\cellcolor{green!25!white}1.000} &  ---  &  ---  & \cellcolor{green!25!white}0.750 & \cellcolor{green!25!white}0.750 & \multicolumn{1}{c||}{\cellcolor{green!25!white}0.750} &  ---  &  ---  & \cellcolor{green!25!white}0.500 & \cellcolor{green!25!white}0.500 & \cellcolor{green!25!white}0.500 \\ \toprule
		\end{tabular}\vspace{-2mm}
		\caption{Experimental order of convergence: $\texttt{EOC}_n(e_{\mathbf{F}^*,n})$,~${n=4,\dots,7}$.}\vspace{2mm}
		\label{tab:2} 

			\setlength\tabcolsep{1pt}
			\centering
			\begin{tabular}{c|c|c|c|c|c|c|c|c|c|c|c|c|c|c|c|} \toprule 
				\multicolumn{1}{|c||}{\cellcolor{lightgray}$\alpha$}	
				& \multicolumn{5}{c||}{\cellcolor{lightgray}$1.0$}   & \multicolumn{5}{c||}{\cellcolor{lightgray}$0.75$} & \multicolumn{5}{c|}{\cellcolor{lightgray}$0.50$} \\ 
				\toprule\toprule
				
				\multicolumn{1}{|c||}{\cellcolor{lightgray}\diagbox[height=1.1\line,width=0.15\dimexpr\linewidth]{\vspace{-0.65mm}$n$}{\\[-4.5mm] $p^-$}}
				& \cellcolor{lightgray}1.5 & \cellcolor{lightgray}1.75  & \cellcolor{lightgray}2.0 & \cellcolor{lightgray}2.25 & \multicolumn{1}{c||}{\cellcolor{lightgray}2.5}
				& \cellcolor{lightgray}1.5 & \cellcolor{lightgray}1.75  & \cellcolor{lightgray}2.0 & \cellcolor{lightgray}2.25 & \multicolumn{1}{c||}{\cellcolor{lightgray}2.5}
				& \cellcolor{lightgray}1.5 & \cellcolor{lightgray}1.75  & \cellcolor{lightgray}2.0 & \cellcolor{lightgray}2.25 & \multicolumn{1}{c|}{\cellcolor{lightgray}2.5} 
				\\ \toprule\toprule 
				\multicolumn{16}{|c|}{\cellcolor{lightgray}Case \hyperlink{C1}{1}}\\ \toprule\toprule 
				\multicolumn{1}{|c||}{\cellcolor{lightgray}$4$} & 0.858 & 0.803 & 0.764 & 0.735 & \multicolumn{1}{c||}{0.711} & 0.656 & 0.615 & 0.585 & 0.562 & \multicolumn{1}{c||}{0.543} & 0.462 & 0.429 & 0.405 & 0.387 & 0.373 \\ \hline
				\multicolumn{1}{|c||}{\cellcolor{lightgray}$5$} & 0.849 & 0.796 & 0.758 & 0.729 & \multicolumn{1}{c||}{0.706} & 0.645 & 0.606 & 0.577 & 0.554 & \multicolumn{1}{c||}{0.536} & 0.448 & 0.418 & 0.395 & 0.379 & 0.365 \\ \hline
				\multicolumn{1}{|c||}{\cellcolor{lightgray}$6$} & 0.844 & 0.792 & 0.755 & 0.726 & \multicolumn{1}{c||}{0.704} & 0.639 & 0.600 & 0.572 & 0.549 & \multicolumn{1}{c||}{0.532} & 0.439 & 0.411 & 0.390 & 0.373 & 0.360 \\ \hline
				\multicolumn{1}{|c||}{\cellcolor{lightgray}$7$} & 0.840 & 0.790 & 0.753 & 0.725 & \multicolumn{1}{c||}{0.702} & 0.634 & 0.597 & 0.568 & 0.547 & \multicolumn{1}{c||}{0.529} & 0.434 & 0.406 & 0.386 & 0.370 & 0.358 \\ \hline\hline  
				\multicolumn{1}{|c||}{\cellcolor{lightgray}\small $\min\{1,\hspace{-1mm}\frac{(p^+)'}{2}\}\alpha$} & \cellcolor{green!25!white}0.833 & \cellcolor{green!25!white}0.786 & \cellcolor{green!25!white}0.750 & \cellcolor{green!25!white}0.722 & \multicolumn{1}{c||}{\cellcolor{green!25!white}0.700} & \cellcolor{green!25!white}0.625 &\cellcolor{green!25!white} 0.589 & \cellcolor{green!25!white}0.563 & \cellcolor{green!25!white}0.541 & \multicolumn{1}{c||}{\cellcolor{green!25!white}0.525} & \cellcolor{green!25!white}0.417 & \cellcolor{green!25!white}0.393 & \cellcolor{green!25!white}0.375 & \cellcolor{green!25!white}0.361 & \cellcolor{green!25!white}0.350 \\ \toprule\toprule
				\multicolumn{1}{|c||}{\cellcolor{lightgray}$4$} &  ---  &  ---  & 0.961 & 0.969 & \multicolumn{1}{c||}{0.978} &  ---  &  ---  & 0.766 & 0.763 & \multicolumn{1}{c||}{0.762} &  ---  &  ---  & 0.537 & 0.532 & 0.528 \\ \hline
				\multicolumn{1}{|c||}{\cellcolor{lightgray}$5$} &  ---  &  ---  & 0.976 & 0.977 & \multicolumn{1}{c||}{0.979} &  ---  &  ---  & 0.762 & 0.760 & \multicolumn{1}{c||}{0.759} &  ---  &  ---  & 0.527 & 0.523 & 0.521 \\ \hline
				\multicolumn{1}{|c||}{\cellcolor{lightgray}$6$} &  ---  &  ---  & 0.987 & 0.987 & \multicolumn{1}{c||}{0.986} &  ---  &  ---  & 0.760 & 0.759 & \multicolumn{1}{c||}{0.758} &  ---  &  ---  & 0.522 & 0.519 & 0.517 \\ \hline
				\multicolumn{1}{|c||}{\cellcolor{lightgray}$7$} &  ---  &  ---  & 0.995 & 0.994 & \multicolumn{1}{c||}{0.993} &  ---  &  ---  & 0.759 & 0.758 & \multicolumn{1}{c||}{0.758} &  ---  &  ---  & 0.518 & 0.516 & 0.514 \\ \toprule\toprule
				\multicolumn{1}{|c||}{\cellcolor{lightgray}\small $\alpha$} &  ---  &  ---  & \cellcolor{green!25!white}1.000 & \cellcolor{green!25!white}1.000 & \multicolumn{1}{c||}{\cellcolor{green!25!white}1.000} &  ---  &  ---  & \cellcolor{green!25!white}0.750 & \cellcolor{green!25!white}0.750 & \multicolumn{1}{c||}{\cellcolor{green!25!white}0.750} &  ---  &  ---  & \cellcolor{green!25!white}0.500 & \cellcolor{green!25!white}0.500 & \cellcolor{green!25!white}0.500 \\ \toprule
			\end{tabular}\vspace{-2mm}
			\caption{Experimental order of convergence: $\texttt{EOC}_n(e_{\varphi^*,n})$,~${n=4,\dots,7}$.}\vspace{2mm}
			\label{tab:3} 

				\setlength\tabcolsep{1pt}
				\centering
				\begin{tabular}{c|c|c|c|c|c|c|c|c|c|c|c|c|c|c|c|} \toprule 
					\multicolumn{1}{|c||}{\cellcolor{lightgray}$\alpha$}	
					& \multicolumn{5}{c||}{\cellcolor{lightgray}$1.0$}   & \multicolumn{5}{c||}{\cellcolor{lightgray}$0.75$} & \multicolumn{5}{c|}{\cellcolor{lightgray}$0.50$} \\ 
					\toprule\toprule
					
					\multicolumn{1}{|c||}{\cellcolor{lightgray}\diagbox[height=1.1\line,width=0.15\dimexpr\linewidth]{\vspace{-0.65mm}$n$}{\\[-4.5mm] $p^-$}}
					& \cellcolor{lightgray}1.5 & \cellcolor{lightgray}1.75  & \cellcolor{lightgray}2.0 & \cellcolor{lightgray}2.25 & \multicolumn{1}{c||}{\cellcolor{lightgray}2.5}
					& \cellcolor{lightgray}1.5 & \cellcolor{lightgray}1.75  & \cellcolor{lightgray}2.0 & \cellcolor{lightgray}2.25 & \multicolumn{1}{c||}{\cellcolor{lightgray}2.5}
					& \cellcolor{lightgray}1.5 & \cellcolor{lightgray}1.75  & \cellcolor{lightgray}2.0 & \cellcolor{lightgray}2.25 & \multicolumn{1}{c|}{\cellcolor{lightgray}2.5} 
					\\ \toprule\toprule 
					\multicolumn{16}{|c|}{\cellcolor{lightgray}Case \hyperlink{C1}{1}}\\ \toprule\toprule 
					\multicolumn{1}{|c||}{\cellcolor{lightgray}$4$} & 1.754 & 1.733 & 1.713 & 1.691 & \multicolumn{1}{c||}{1.666} & 1.660 & 1.662 & 1.659 & 1.653 & \multicolumn{1}{c||}{1.643} & 1.548 & 1.571 & 1.586 & 1.596 & 1.601 \\ \hline
					\multicolumn{1}{|c||}{\cellcolor{lightgray}$5$} & 1.805 & 1.775 & 1.750 & 1.725 & \multicolumn{1}{c||}{1.699} & 1.683 & 1.679 & 1.676 & 1.670 & \multicolumn{1}{c||}{1.660} & 1.562 & 1.580 & 1.594 & 1.604 & 1.610 \\ \hline
					\multicolumn{1}{|c||}{\cellcolor{lightgray}$6$} & 1.831 & 1.797 & 1.769 & 1.742 & \multicolumn{1}{c||}{1.714} & 1.692 & 1.687 & 1.684 & 1.678 & \multicolumn{1}{c||}{1.668} & 1.563 & 1.580 & 1.595 & 1.606 & 1.613 \\ \hline
					\multicolumn{1}{|c||}{\cellcolor{lightgray}$7$} & 1.846 & 1.811 & 1.782 & 1.754 & \multicolumn{1}{c||}{1.723} & 1.697 & 1.692 & 1.689 & 1.683 & \multicolumn{1}{c||}{1.673} & 1.561 & 1.579 & 1.594 & 1.606 & 1.614 \\ \hline\hline  
					\multicolumn{1}{|c||}{\cellcolor{lightgray}\small $\min\{1,\hspace{-1mm}\frac{(p^+)'}{2}\}\alpha$} & \cellcolor{red!25!white}0.833 & \cellcolor{red!25!white}0.786 & \cellcolor{red!25!white}0.750 & \cellcolor{red!25!white}0.722 & \multicolumn{1}{c||}{\cellcolor{red!25!white}0.700} & \cellcolor{red!25!white}0.625 & \cellcolor{red!25!white}0.589 & \cellcolor{red!25!white}0.563 & \cellcolor{red!25!white}0.541 & \multicolumn{1}{c||}{\cellcolor{red!25!white}0.525} & \cellcolor{red!25!white}0.417 & \cellcolor{red!25!white}0.393 & \cellcolor{red!25!white}0.375 & \cellcolor{red!25!white}0.361 & \cellcolor{red!25!white}0.350 \\ \toprule\toprule
					\multicolumn{16}{|c|}{\cellcolor{lightgray}Case \hyperlink{C2}{2}}\\ \toprule\toprule 
					\multicolumn{1}{|c||}{\cellcolor{lightgray}$4$} &  ---  &  ---  & 1.642 & 1.578 & \multicolumn{1}{c||}{1.484} &  ---  &  ---  & 1.715 & 1.702 & \multicolumn{1}{c||}{1.680} &  ---  &  ---  & 1.645 & 1.652 & 1.650 \\ \hline
					\multicolumn{1}{|c||}{\cellcolor{lightgray}$5$} &  ---  &  ---  & 1.763 & 1.721 & \multicolumn{1}{c||}{1.664} &  ---  &  ---  & 1.750 & 1.737 & \multicolumn{1}{c||}{1.716} &  ---  &  ---  & 1.658 & 1.666 & 1.665 \\ \hline
					\multicolumn{1}{|c||}{\cellcolor{lightgray}$6$} &  ---  &  ---  & 1.828 & 1.787 & \multicolumn{1}{c||}{1.740} &  ---  &  ---  & 1.769 & 1.755 & \multicolumn{1}{c||}{1.732} &  ---  &  ---  & 1.664 & 1.672 & 1.672 \\ \hline
					\multicolumn{1}{|c||}{\cellcolor{lightgray}$7$} &  ---  &  ---  & 1.863 & 1.819 & \multicolumn{1}{c||}{1.770} &  ---  &  ---  & 1.781 & 1.766 & \multicolumn{1}{c||}{1.742} &  ---  &  ---  & 1.666 & 1.674 & 1.674 \\ \toprule\toprule
					\multicolumn{1}{|c||}{\cellcolor{lightgray}\small $\alpha$} &  ---  &  ---  & \cellcolor{red!25!white}1.000 & \cellcolor{red!25!white}1.000 & \multicolumn{1}{c||}{\cellcolor{red!25!white}1.000} &  ---  &  ---  & \cellcolor{red!25!white}0.750 & \cellcolor{red!25!white}0.750 & \multicolumn{1}{c||}{\cellcolor{red!25!white}0.750} &  ---  &  ---  & \cellcolor{red!25!white}0.500 & \cellcolor{red!25!white}0.500 & \cellcolor{red!25!white}0.500 \\ \toprule
				\end{tabular}\vspace{-2mm}
				\caption{Experimental order of convergence: $\texttt{EOC}_n(e_{L^2,n})$,~${n=4,\dots,7}$.}
				\label{tab:4} 
				\end{table}

 \appendix

 \section{Outlook}

 \hspace{5mm}In \hspace{-0.1mm}the \hspace{-0.1mm}present \hspace{-0.1mm}paper, \hspace{-0.1mm}we \hspace{-0.1mm}examined \hspace{-0.1mm}a \hspace{-0.1mm}fully-discrete \hspace{-0.1mm}FE \hspace{-0.1mm}approximation \hspace{-0.1mm}of \hspace{-0.1mm}the \hspace{-0.1mm}unsteady~\hspace{-0.1mm}\mbox{$p(\cdot,\cdot)$-Stokes}~\hspace{-0.1mm}equa-tions~\hspace{-0.1mm}\eqref{eq:ptxStokes}, \hspace{-0.1mm}employing \hspace{-0.1mm}a \hspace{-0.1mm}backward \hspace{-0.1mm}Euler \hspace{-0.1mm}step \hspace{-0.1mm}in \hspace{-0.1mm}time \hspace{-0.1mm}and \hspace{-0.1mm}conforming, \hspace{-0.1mm}discretely \hspace{-0.1mm}inf-sup~\hspace{-0.1mm}\mbox{stable}~\hspace{-0.1mm}FEs~\hspace{-0.1mm}in~\hspace{-0.1mm}space, for \textit{a priori} error estimates. More precisely, we derived error decay rates for the velocity vector field imposing only fractional regularity assumptions on the velocity vector field and the kinematic pressure.\linebreak In the case $p^-\hspace{-0.15em}\ge\hspace{-0.15em} 2$, we confirmed the optimality of the derived error~decay~rates~via~numerical~\mbox{experiments}. In the case $p^-\hspace{-0.15em}\leq\hspace{-0.15em} 2$, however, the fractional regularity assumption $\mathbf{v}\in L^\infty(I;(N^{1+\beta_{\mathrm{x}},2}(\Omega))^d)$, $\beta_{\mathrm{x}}\in (0,1]$, turns out to be too restrictive. In fact, compared to \cite{ptx_lap}, this additional fractional regularity assumption is merely necessary because, owing to the discrete incompressibility constraint in Problem (\hyperlink{Qh}{Q$_h^\tau$}), in the \textit{a priori} error analysis  (more precisely, in the estimation of the term $\smash{I^{1,2}_{m,h}}$ in~the~proof~of~Theorem~\ref{thm:main}), we need to work with the only locally $W^{1,1}$-stable FE projection operator $\smash{\Pi_h^V}$, while  a locally $L^1$-stable projection operator could be employed in \cite{ptx_lap}. If the  FE projection~operator~$\Pi_h^V$~would~be~locally~\mbox{$L^1$-stable}, then in \eqref{thm:main.21.2} (\textit{i.e.}, the sole point where this additional fractional regularity assumption~comes~into~play),  we could instead estimate as follows:
\begin{align*}
        \begin{aligned}
            \tfrac{1}{\tau}\|\mathrm{I}_\tau^{0,\mathrm{t}}\bfv-
			\Pi_h^V \mathrm{I}_\tau^{0,\mathrm{t}}\bfv\|_{2,\smash{Q_T^m}}^2
            &
            \leq  c\,\smash{\tfrac{h^{2\beta_{\mathrm{x}}}}{\tau}}\, [\mathrm{I}_\tau^{0,\mathrm{t}}\bfv]_{L^2((0,t_m);N^{\beta_{\mathrm{x}},2}(\Omega))}^2
            \\&
            \leq c\,\smash{\tfrac{h^{2\beta_{\mathrm{x}}}}{\tau}}\, [ \bfv]_{L^\infty((0,t_m);N^{\beta_{\mathrm{x}},2}(\Omega))}^2\,,
        \end{aligned}
    \end{align*}
    so that, similar to \cite{ptx_lap}, imposing the condition $h^{2\beta_{\mathrm{x}}}\lesssim \tau^{2\min\{\alpha_{\mathrm{t}},\beta_{\mathrm{t}}\}+1}$ instead of $h^2\lesssim \tau$, we could prove the assertion of Theorem \ref{thm:main}, in the case $p^-\leq 2$, without this additional fractional~regularity~assumption. However, the FE projection operator $\Pi_h^V$ cannot be expected to be locally $L^1$-stable. Instead, we expect that one can generalize the procedure of the recent contribution 
    \cite{br-pnse}
    to the framework of~the~present~paper, in order to remove the additional fractional regularity assumption. This, however, would certainly~be~out of the scope of the present paper and, therefore, will be  content~of~future~research.
 
\section{Appendix}
\subsection{Discrete-to-continuous-and-vice-versa inequalities}
	\hspace*{5mm}The following result estimates the error caused by switching from  $\bfF_h^\tau\colon Q_T\times\mathbb{R}^{d\times d}\to \mathbb{R}^{d\times d}_{\textup{sym}}$,~${\tau,h\in (0,1]}$, to $\bfF\colon Q_T\times\mathbb{R}^{d\times d}\to \mathbb{R}^{d\times d}_{\textup{sym}}$, from $\bfS_h^\tau\colon Q_T\times\mathbb{R}^{d\times d}\to \mathbb{R}^{d\times d}_{\textup{sym}}$,~${\tau,h\in (0,1]}$, to $\bfS\colon Q_T\times\mathbb{R}^{d\times d}\to \mathbb{R}^{d\times d}_{\textup{sym}}$,~or~from
	$(\varphi_h^{\tau})^*\colon Q_T\times\mathbb{R}_{\ge 0}\to \mathbb{R}_{\ge 0}$, $\tau, h\in (0,1]$, to ${\varphi^*\colon Q_T\times\mathbb{R}_{\ge 0}\to \mathbb{R}_{\ge 0}}$ and vice versa, respectively. 
	
\begin{lemma}\label{lem:A-Ah}
	Assume that $p\in C^0(\overline{Q_T})$ with $p^->1$. Then, there exists a constant $s>1$ with $s\searrow1$
    as $\vert t-s\vert+\vert x-y\vert\searrow0$, such that for every $(t,x)^\top,(s,y)^\top\hspace*{-0.1em}\in\hspace*{-0.1em} \overline{Q_T}$, $r\hspace*{-0.1em}\ge\hspace*{-0.1em} 0$, $\bfA\hspace*{-0.1em}\in\hspace*{-0.1em} \mathbb{R}^{d\times d}$,~and~${\lambda\hspace*{-0.1em}\in\hspace*{-0.1em} [0,1]}$,~there~holds
	\begin{align}
		\vert \bfF(t,x,\bfA)-\bfF(s,y,\bfA)\vert^2
        &\lesssim \vert p(t,x)-p(s,y)\vert^2\,(1 +\vert \bfA\vert^{p(t,x)s})\,,\label{eq:Fh-F}\\
		\vert \bfF^*(t,x,\bfS(t,x,\bfA))-\bfF^*(t,x,\bfS(s,y,\bfA))\vert^2&\lesssim\vert p(t,x)-p(s,y)\vert^2\,(1 +\vert \bfA\vert^{p(t,x)s}) \,,\label{eq:Ah-A}\\ 
		(\varphi_{\vert \bfA\vert})^*(t,x,\lambda\,r) &\lesssim (\varphi_{\vert  \bfA\vert})^*(s,y,\lambda\,r)\label{eq:phih-phi}\\&\quad+
		\smash{\lambda^{\smash{\max\{2,(p^+)'\}}}}\,\vert p(t,x)-p(s,y)\vert \,(1+\vert \bfA\vert^{p(s,y)s}+r^{p'(s,y)s})\,,\notag
	\end{align}
		where the implicit constant in $\lesssim$ depends  on $p^-$, $p^+$, and $s$.
\end{lemma}

\begin{proof}
	See \cite[Prop.\  2.10]{BK23_pxDirichlet}.
\end{proof} 

\subsection{Local and global Poincar\'e inequalities in terms of the natural distance}

\hspace*{5mm}The following result contains local and global Poincar\'e inequalities in terms of the natural distance with respect to the discretization of the time-space cylinder given only fractional parabolic~regularity assumptions.\pagebreak
	
\begin{lemma}\label{lem:poincare_F}
 Suppose that $p\in C^{0,\alpha_{\mathrm{t}},\alpha_{\mathrm{x}}}(\overline{Q_T})$, $\alpha_{\mathrm{t}},\alpha_{\mathrm{x}}\in (0,1]$, with $p^->1$,  and let $\bfA\in (L^{p(\cdot,\cdot)}(Q_T))^{d\times d}$ be such that
	$\bfF(\cdot,\cdot,\bfA)\in N^{\beta_{\mathrm{t}},2}(I;(L^2(\Omega))^{d\times d})\cap L^2(I;(N^{\beta_{\mathrm{x}},2}(\Omega))^{d\times d})$, $\beta_{\mathrm{t}}\in (\frac{1}{2},1]$, $\beta_{\mathrm{x}}\in (0,1]$. Then, there exists a constant $s>1$ with $s\searrow1$ as $\tau+ h_K\searrow0$  such that for every $J\in \mathcal{I}_\tau$ and $K\in \mathcal{T}_h$,~there~holds 
	\begin{align}
		\|\bfF(\cdot,\cdot,\bfA)-\bfF(\cdot,\cdot,\langle\bfA\rangle_{J\times K})\|_{2,J\times K}^2
        &\lesssim (\tau^{2\alpha_{\mathrm{t}}}+h_K ^{2\alpha_{\mathrm{x}}})\,\|1+\vert \bfA\vert^{p(\cdot,\cdot)s}\|_{1,J\times K}\notag
        \\&\quad+\tau^{2\beta_{\mathrm{t}}}\,[ \bfF(\cdot,\cdot,\bfA)]_{N^{\beta_{\mathrm{t}},2}(J;L^2(K))}^2\label{lem:poincare_F.1}
        \\&\quad+h_K ^{2\beta_{\mathrm{x}}}\,[ \bfF(\cdot,\cdot,\bfA)]_{L^2(J;N^{\beta_{\mathrm{x}},2}(\omega_K ))}^2\,,\notag\\
		\|\bfF(\cdot,\cdot,\bfA)-\bfF(\cdot,\cdot,\langle\bfA\rangle_{J\times\omega_K })\|_{2,J\times\omega_K }^2&\lesssim (\tau^{2\alpha_{\mathrm{t}}}+h_K ^{2\alpha_{\mathrm{x}}})\,\|1+\vert \bfA\vert^{p(\cdot,\cdot)s}\|_{1,J\times\omega_K }\notag
        \\&\quad+\tau^{2\beta_{\mathrm{t}}}\,[ \bfF(\cdot,\cdot,\bfA)]_{N^{\beta_{\mathrm{t}},2}(J;L^2(\omega_K ))}^2\label{lem:poincare_F.2}
        \\&\quad+h_K ^{2\beta_{\mathrm{x}}}\,[ \bfF(\cdot,\cdot,\bfA)]_{L^2(J;N^{\beta_{\mathrm{x}},2}(\omega_K ^{2\times}))}^2 \,,\notag
	\end{align}
	where $\omega_K ^{2\times}\coloneqq \bigcup_{K'\in \omega_K }{\omega_{K'}}$ and the implicit constant in $\lesssim$ depends on $p^-$, $p^+$, $[p]_{\alpha_{\mathrm{t}},\alpha_{\mathrm{x}},Q_T}$, $s$, and $\omega_0$.
	In particular, it follows that
	\begin{align}
		\|\bfF(\cdot,\cdot,\bfA)-\bfF(\cdot,\cdot,\Pi_h^{0,\mathrm{t}}\Pi_h^{0,\mathrm{x}}\bfA)\|_{2,Q_T}^2
        &\lesssim (\tau^{2\alpha_{\mathrm{t}}}+h^{2\alpha_{\mathrm{x}}})\,\|1+\vert \bfA\vert^{p(\cdot,\cdot)s}\|_{1,Q_T}\notag
        \\&\quad+\tau^{2\beta_{\mathrm{t}}}\,[ \bfF(\cdot,\cdot,\bfA)]_{N^{\beta_{\mathrm{t}},2}(I;L^2(\Omega))}^2\label{lem:poincare_F.3}\\&\quad+h^{2\beta_{\mathrm{x}}}\,[ \bfF(\cdot,\cdot,\bfA)]_{L^2(I;N^{\beta_{\mathrm{x}},2}(\Omega))}^2 \,,\notag\\
		\sum_{K\in \mathcal{T}_h}{	\|\bfF(\cdot,\cdot,\bfA)-\bfF(\cdot,\cdot,\langle\bfA\rangle_{J\times\omega_K})\|_{2,\Omega}^2}&\lesssim 
        (\tau^{2\alpha_{\mathrm{t}}}+h^{2\alpha_{\mathrm{x}}})\,\|1+\vert \bfA\vert^{p(\cdot,\cdot)s}\|_{1,Q_T}\notag
        \\[-3mm]&\quad+\tau^{2\beta_{\mathrm{t}}}\,[ \bfF(\cdot,\cdot,\bfA)]_{N^{\beta_{\mathrm{t}},2}(I;L^2(\Omega))}^2\label{lem:poincare_F.4}\\&\quad+h^{2\beta_{\mathrm{x}}}\,[ \bfF(\cdot,\cdot,\bfA)]_{L^2(I;N^{\beta_{\mathrm{x}},2}(\Omega))}^2 \,.\notag
	\end{align} 
\end{lemma}

\begin{proof} 
    The claimed local and global Poincar\'e inequalities \eqref{lem:poincare_F.1}--\eqref{lem:poincare_F.4} are obtained analogously to \cite[Lem.\ B.10]{berselli2023error} up to minor adjustments.
\end{proof}

	{\setlength{\bibsep}{0pt plus 0.0ex}\small 
    \def\cprime{$'$} \def\cprime{$'$} \def\cprime{$'$}
\providecommand{\bysame}{\leavevmode\hbox to3em{\hrulefill}\thinspace}
\providecommand{\noopsort}[1]{}
\providecommand{\mr}[1]{\href{http://www.ams.org/mathscinet-getitem?mr=#1}{MR~#1}}
\providecommand{\zbl}[1]{\href{http://www.zentralblatt-math.org/zmath/en/search/?q=an:#1}{Zbl~#1}}
\providecommand{\jfm}[1]{\href{http://www.emis.de/cgi-bin/JFM-item?#1}{JFM~#1}}
\providecommand{\arxiv}[1]{\href{http://www.arxiv.org/abs/#1}{arXiv~#1}}
\providecommand{\doi}[1]{\url{https://doi.org/#1}}
\providecommand{\MR}{\relax\ifhmode\unskip\space\fi MR }
\providecommand{\MRhref}[2]{%
  \href{http://www.ams.org/mathscinet-getitem?mr=#1}{#2}
}
\providecommand{\href}[2]{#2}

	}
	
\end{document}